\documentclass[12pt,leqno]{article}

\usepackage{mathtools}
\usepackage{amsthm}
\usepackage{amssymb}
\usepackage[svgnames]{xcolor}
\usepackage[
colorlinks=true,
linkcolor=MidnightBlue,
citecolor=MidnightBlue,
urlcolor=MidnightBlue,
pagebackref=true,
linktocpage=true]{hyperref}
\usepackage[alphabetic]{amsrefs}
\usepackage{pinlabel}
\usepackage{enumitem}
\usepackage{fullpage}
\usepackage[margin=30pt]{caption}

\newcommand{\sslash}{\mathbin{/\mkern-6mu/}}
\newcommand{\bd}{\frac{\de}{(N+1)L^N}}

\newcommand{\N}{\mathbb{N}}
\newcommand{\Z}{\mathbb{Z}}
\newcommand{\R}{\mathbb{R}}
\newcommand{\C}{\mathbb{C}}
\renewcommand{\H}{\mathbb{H}}
\newcommand{\I}{\mathcal{I}}
\newcommand{\U}{\mathcal{U}}

\newcommand{\p}{P}
\newcommand{\s}{S}
\newcommand{\D}{\mathcal{D}}

\newcommand{\De}{\Delta}
\newcommand{\Ga}{\Gamma}
\newcommand{\La}{\Lambda}
\newcommand{\al}{\alpha}
\newcommand{\be}{\beta}
\newcommand{\de}{\delta}
\newcommand{\ep}{\epsilon}

\newcommand{\vS}{\varSigma}
\newcommand{\ga}{\gamma}
\newcommand{\la}{\lambda}
\newcommand{\si}{\sigma}
\newcommand{\ad}{a_{\mathrm{mod}}}
\newcommand{\sid}{{\sigma_{\mathrm{mod}}}}
\newcommand{\tad}{{\tau_{\mathrm{mod}}}}
\newcommand{\ef}{\mathcal E}

\newcommand{\Homeo}{\mathrm{LHomeo}}
\newcommand{\Diff}{\mathrm{Diff}}
\newcommand{\Hom}{\mathrm{Hom}}
\newcommand{\Isom}{\mathrm{Isom}}
\newcommand{\PGL}{\mathrm{PGL}}
\newcommand{\PSL}{\mathrm{PSL}}
\newcommand{\PSO}{\mathrm{PSO}}
\newcommand{\GL}{\mathrm{GL}}
\newcommand{\SL}{\mathrm{SL}}
\newcommand{\SO}{\mathrm{SO}}
\newcommand{\Aff}{\mathrm{Aff}}
\newcommand{\Hit}{\mathrm{Hit}}
\newcommand{\Ba}{\mathrm{Bar}}
\newcommand{\Lip}{\mathrm{Lip}}
\newcommand{\Flag}{\mathrm{Flag}}
\newcommand{\intr}{\mathrm{int}\,}
\newcommand{\Code}{\mathrm{Code}}
\newcommand{\Ray}{\mathrm{Ray}}
\newcommand{\End}{\mathrm{End}}
\newcommand{\pa}{\partial_\infty}
\newcommand{\id}{e}
\newcommand{\ov}{\overline}
\newcommand{\wt}{\widetilde}
\newcommand{\diam}{\mathrm{diam}}
\newcommand{\card}{\mathrm{card}}
\newcommand{\ost}{\mathrm{ost}}
\newcommand{\st}{\mathrm{st}}

\swapnumbers
\numberwithin{equation}{section}

\theoremstyle{plain}
\newtheorem{theorem}[equation]{Theorem}
\newtheorem{corollary}[equation]{Corollary}
\newtheorem{lemma}[equation]{Lemma}
\newtheorem{proposition}[equation]{Proposition}
\newtheorem*{claim*}{Claim}
\newtheorem*{theorem*}{Theorem}

\theoremstyle{definition}
\newtheorem{definition}[equation]{Definition}
\newtheorem{example}[equation]{Example}
\newtheorem{question}[equation]{Question}
\newtheorem{conjecture}[equation]{Conjecture}

\theoremstyle{remark}
\newtheorem{remark}[equation]{Remark}

\title{Structural stability of meandering-hyperbolic group actions}
\author{Michael Kapovich, Sungwoon Kim and Jaejeong Lee}
\date{}

\begin{document}

\maketitle

\begin{abstract}
In his 1985 paper Sullivan sketched a proof of his structural stability theorem for differentiable group actions satisfying certain expansion-hyperbolicity axioms. In this paper we relax Sullivan's axioms and introduce a notion of \emph{meandering hyperbolicity} for group actions on geodesic metric spaces. This generalization is substantial enough to encompass actions of certain non-hyperbolic groups, such as actions of \emph{uniform lattices} in semisimple Lie groups on flag manifolds. At the same time, our notion is sufficiently robust and we prove that meandering-hyperbolic actions are still structurally stable. We also prove some basic results on meandering-hyperbolic actions and give other examples of such actions.
\end{abstract}

\tableofcontents

\section{Introduction}

In \cite{Sul} Sullivan stated his structural stability theorem for group actions satisfying certain expansion-hyperbolicity axioms in the $C^1$-setting. Its main application was to the stability for convex-cocompact Kleinian groups, but this particular result of Sullivan appears to have been relatively unknown (see Remark~\ref{rem:sul}) presumably due to the lack of a detailed proof.

The goal of this paper is two-fold. First, we generalize Sullivan's result in two directions and provide a detailed proof along the way. We work with locally bi-Lipschitz group actions on geodesic metric spaces. We then weaken Sullivan's axioms to what we call \emph{meandering hyperbolicity}. This condition still implies the structural stability in the sense of Sullivan even though the groups need \emph{not} be word-hyperbolic. Second, we establish some basic properties of meandering-hyperbolic actions and explore various examples of such actions, most unexpected among which are actions of \emph{uniform lattices} in semisimple Lie groups on flag manifolds. As for word-hyperbolic groups, other interesting examples include Anosov actions on flag manifolds as well as actions with invariant subsets non-homeomorphic to Gromov boundaries.

We now explain Sullivan's structural stability theorem in more detail and exhibit the key idea of our generalization.

\begin{theorem*}[\cite{Sul}*{\S9. Theorem~II}]
Consider a group action $\Ga\to\Diff^1(M)$ on a Riemannian manifold $M$ with a compact invariant subset $\La\subset M$. If the action satisfies the expansion-hyperbolicity axioms, then it is structurally stable in the sense of $C^1$-dynamics.
\end{theorem*}

Here the \emph{expansion} axiom means that for all $x\in\La$ there exists $g\in\Ga$ such that $\|g'(x)\|>1$ with respect to the Riemannian metric. We generalize this condition accordingly for locally bi-Lipschitz actions on geodesic metric spaces (Definition~\ref{def:expansion}). For each point $x\in\La$ the expansion condition enables us to choose a sequence (called a \emph{code}) $\al$ in a finite generating set of $\Ga$, along with the corresponding \emph{ray} $c^\al:\N_0\to\Ga$ in (the Cayley graph of) the group $\Ga$ (Definitions~\ref{def:code} and \ref{def:ray}).

Sullivan's hyperbolicity axiom, which we refer to as \emph{S-hyperbolicity} (with the letter S standing for Sullivan), is an additional requirement to the expansion axiom, stating that any two rays associated to a point $x\in\La$ are within Hausdorff distance $N$ from each other for some uniform constant $N>0$ (Definition~\ref{def:shyp}), which is a form of the \emph{fellow-traveling condition}. Under this condition, Sullivan's theorem asserts that the action $\rho:\Ga\to\Diff^1(M)$ is \emph{structurally stable in the sense of $C^1$-dynamics}, that is, for every action $\rho'$ sufficiently $C^1$-close to $\rho$, there exists a $\rho'$-invariant compact subset $\La'\subset M$ and an equivariant homeomorphism $\phi:\La\to\La'$. The S-hyperbolicity condition plays a critical role in defining the image $\phi(x)$ unambiguously, regardless of the choice made for codes $\al$.

The key point in our generalization is the simple observation that such a fellow-traveling property of rays in the S-hyperbolicity condition is too strong a requirement and that much more relaxed equivalence relation between rays would still suffice for the existence of the map $\phi$ (Section~\ref{sec:well}). Namely, given two rays for a point in $\La$, we require that they are Hausdorff-close only on infinite subsets of each. Furthermore, we allow interpolation of such a relation by other rays. The equivalence relation generated in this way is the main feature of our generalization of S-hyperbolicity, which we call the \emph{meandering hyperbolicity} condition (Definition~\ref{def:mhyp}). See Remark~\ref{rem:N-equiv}(b) for our reasoning behind the nomenclature.

With this relaxed definition we generalize Sullivan's structural stability theorem as follows. Let $\Homeo(M)$ denote the group of locally bi-Lipschitz homeomorphisms of $M$.

\begin{theorem}[Theorem~\ref{thm:main}(1)]\label{thm:main1}
Consider an action $\Ga\to\Homeo(M)$ of a finitely generated group $\Ga$ on a proper geodesic metric space $M$ with a compact invariant subset $\La\subset M$, no point of which is isolated in $M$. If the action is meandering-hyperbolic, then it is structurally stable in the sense of Lipschitz dynamics.
\end{theorem}

\noindent See Theorem~\ref{thm:main} for the full statement. We say that an action $\rho:\Ga\to\Homeo(M)$ is \emph{structurally stable in the sense of Lipschitz dynamics} if, for every action $\rho':\Ga\to\Homeo(M)$ sufficiently close to $\rho$ with respect to the compact-open Lipschitz topology (see Section~\ref{sec:hyp->stability}), there exists a $\rho'$-invariant subset $\La'\subset M$ along with a topological conjugacy $\phi:\La\to\La'$.

This generalization of Sullivan's theorem is general enough to encompass actions of certain \emph{non-hyperbolic} groups as we exhibit below. Its proof will take the entire Section~\ref{sec:proof}. We mostly follow Sullivan's idea of proof, filling in the details he sometimes left out, except in Section~\ref{sec:well}, where the meandering hyperbolicity condition comes in and plays a crucial role, and in Sections~\ref{sec:expand} and \ref{sec:uhyp}, which are not covered in his paper.

\medskip
We investigate various examples of expanding and meandering-hyperbolic actions. Most significantly, in Section~\ref{sec:ulattice}, we establish the following result regarding \emph{uniform lattices} of semisimple Lie groups:

\begin{theorem}[Theorem~\ref{thm:ulattice} and Corollary~\ref{cor:ulattice}]\label{thm:main2}
Suppose that $G$ is a semisimple Lie group, $P< G$ is a parabolic subgroup and $\Ga< G$ is a uniform lattice. Then
\begin{enumerate}[label=\textup{(\arabic*)},nosep,leftmargin=*]
\item
the $\Ga$-action on the flag manifold $G/P$ is (uniformly) meandering-hyperbolic;
\item
consequently, the action is structurally stable in the sense of Lipschitz dynamics.
\end{enumerate}
\end{theorem}

\noindent While stability of the action is consistent with various other rigidity properties of uniform lattices, the hyperbolic nature of the actions of higher rank uniform lattices on flag manifolds is a new phenomenon: higher rank uniform lattices are traditionally regarded as having no hyperbolic features. The main tools in proving Theorem~\ref{thm:main2}(1) are Morse quasigeodesics in higher rank symmetric spaces and asymptotic properties of \emph{regular} sequences in $G$, established in earlier papers by Kapovich, Leeb and Porti \cites{KL18b, KLP17, KLP18}. 

We hope that other interesting non-hyperbolic group actions satisfy meandering hyperbolicity as well. For example: 

\begin{conjecture}
Let $\Ga$ be a group of automorphisms of a $\mathrm{CAT}(0)$ cube complex $X$, acting on $X$ properly discontinuously and cocompactly. Then the action of $\Ga$ on the \emph{Roller boundary} $\partial_R X$ of $X$ is meandering-hyperbolic for a suitable metric on $\partial_R X$. 
\end{conjecture}

Even for word-hyperbolic groups $\Ga$, meandering-hyperbolic actions $\Ga\to\Homeo(M)$ provide interesting results and examples. In this setting, we find a relation between the $\Ga$-invariant set $\La$ and the Gromov boundary $\pa\Ga$ of $\Ga$ as follows.

\begin{theorem}[Definition~\ref{def:pi} and Theorem~\ref{thm:pi}]\label{thm:codingmap}
Let $\Ga$ be a non-elementary word-hyperbolic group. If $\Ga\to\Homeo(M)$ is a meandering-hyperbolic action with a compact invariant subset $\La$, then there exists an equivariant continuous surjective \emph{postal map} $\pi:\La\to\pa\Ga$ to the Gromov boundary of $\Ga$; the map $\pi$ restricts to a quasi-open map on each minimal non-empty closed $\Ga$-invariant subset $\La_\mu\subset\La$.
\end{theorem}

\noindent Here, a map between topological spaces is said to be \emph{nowhere constant} (resp. \emph{quasi-open}) if the image of every non-empty open subset is not a singleton (resp. has non-empty interior).

We then explore examples of meandering-hyperbolic actions of word-hyperbolic groups where the postal maps $\pi$ above are increasingly more complicated. Simple examples are convex-cocompact Kleinian groups and, more generally, Anosov subgroups (see Section~\ref{sec:anosov}), for which the invariant subsets $\La$ are equivariantly homeomorphic to the Gromov boundary $\pa\Ga$ (via the postal map $\pi$). In contrast, there are examples where the postal map $\pi$ can be a covering map (Examples~\ref{ex:E0} and \ref{ex:E1}), can be open but fail to be a local homeomorphism (Example~\ref{ex:branch}), and can even fail to be an open map (Example~\ref{ex:blowup}). 

Conversely to Theorem~\ref{thm:codingmap}, we also prove:

\begin{theorem}[Theorem~\ref{thm:e=>h}]\label{thm:main3}
Let $\Ga$ be a non-elementary word-hyperbolic group. Suppose that $\Ga\to\Homeo(M)$ is an action expanding at $\La$, for which there exists an equivariant continuous nowhere constant map $f:\La\to \pa\Ga$. Then the $\Ga$-action is S-hyperbolic (and thus, meandering-hyperbolic) at $\La$ and the map $f$ equals the postal map $\pi$ (of the previous theorem). 
\end{theorem} 

We note that in the case of real-analytic expanding $\Gamma$-actions on the circle, a sharper result is proven by Deroin in \cite{Deroin}. Without assuming hyperbolicity of $\Ga$ (he only assumes finite generation and \emph{local discreteness} of the action), Deroin proves that the action comes from a lift of a Fuchsian group action under a finite covering map $S^1\to S^1$. Compare Example~\ref{ex:E0}.

\medskip
Building upon the work in \cite{KLP17} and Theorem~\ref{thm:main3}, we give a (yet another) new characterization of Anosov subgroups in terms of expanding actions. See Theorem~\ref{thm:equiv}. This characterization shows, among other things, that the action of any Anosov group on its flag-limit set in the partial flag manifold is S-hyperbolic. Thus, thanks to Theorem~\ref{thm:main1}, we obtain the stability of Anosov groups in a broader context than those in \cite{GW}*{Theorem 5.13} and \cite{KLP14}*{Theorems 1.11 and 7.36}. Based on this we also obtain an alternative proof for the openness of the Anosov property in the representation variety (Corollary~\ref{cor:stability}). 

We remark that Bochi, Potrie and Sambarino \cite{BPS} gave a purely dynamical proof of the structural stability of Anosov representations, which shares many similarities with the concepts in the present paper. For example, one may compare the expansion subsets with the notion of multicones in \cite{BPS}*{Section~5}. The existence of codings appears in \cite{BPS}*{Lemma 3.20} and the expansion property in \cite{BPS}*{Lemma 3.21}. (We thank the referee for pointing out these similarities to us.) On the other hand, in the setting of S-hyperbolic actions, such as Anosov actions, all the ideas in the proof of Theorem~\ref{thm:main1} go back to Sullivan's original paper \cite{Sul}.

After the initial version of this paper was completed, Mann and Manning posted an interesting preprint \cite{MM} related to the main theme of our paper: they prove topological stability of word-hyperbolic group actions on their boundaries under the assumption that the latter are topological spheres. Unlike our work, they only assume $C^0$ rather than Lipschitz perturbation.

\paragraph{Acknowledgments.}
We thank Bernhard Leeb and Inkang Kim for their interest and encouragement. We are grateful to the referee for reading the manuscript extremely carefully and providing many valuable comments that helped us to clarify and streamline the exposition. We incorporated our result on uniform lattices in the present paper as per the referee's suggestion as well.

M.~Kapovich was partly supported by the NSF grant DMS-16-04241, by KIAS (the Korea Institute for Advanced Study) through the KIAS scholar program, and by a Simons Foundation Fellowship, grant number 391602.
S.~Kim was supported by the grants NRF-2015R1D1A1A09058742 and NRF-2018R1D1A1B07043321.
J.~Lee thanks Jeju National University for its hospitality during his visit; he was supported by the grants NRF-2014R1A2A2A01005574, NRF-2017R1A2A2A05001002, and NRF-2019R1F1A1047703.

\section{Notation and preliminaries}\label{sec:prelim}

The identity element of an abstract group will be denoted by $e$. We will use the following notation for the sets of non-negative integers and natural numbers: 
\[
\N_0=\{0,1,2,\ldots\}\quad\textup{and}\quad\N=\{1,2,3,\ldots\}.
\]

We will follow the Bourbaki convention that neighborhoods of a point $a$ (resp. a subset $A$) in a topological space $X$ need not be open but are only required to contain an open subset which, in turn, contains $a$ (resp. $A$). In particular, a topological space $X$ is locally compact if and only if every point in $X$ admits a neighborhood basis consisting of compact subsets of $X$.

A topological space is called \emph{perfect} if it has no isolated points and has cardinality $\ge 2$.

A map between topological spaces is \emph{nowhere constant} if the image of every open subset is not a singleton. A map is said to be \emph{open} if it sends open sets to open sets. A map $f:X\to Y$ is \emph{open at a point $x\in X$} if it sends every neighborhood of $x$ to a neighborhood of $f(x)$. We let $O_f$ denote the subset of $X$ consisting of points where $f$ is open. Thus, a map $f$ is open if and only if $O_f= X$.

A map $f: X\to Y$ is said to be \emph{quasi-open} (or \emph{quasi-interior}) if for every subset $A\subset X$ with non-empty interior, the image $f(A)$ has non-empty interior in $Y$. If $f:X\to Y$ is a continuous map between locally compact metrizable spaces then it is quasi-open if and only if the subset $O_f\subset X$ is comeagre (that is, its complement is a countable union of nowhere dense subsets). For instance, the map $\R\to\R, x\mapsto x^2$, is quasi-open but not open. A more interesting example of a (non-open) quasi-open map is a \emph{Cantor function} $f:C\to[0,1]$, which is a continuous surjective monotonic function from a Cantor set $C\subset\R$. It has the property that $x_1<x_2$ implies $f(x_1)< f(x_2)$ unless $x_1, x_2$ are boundary points of a component of $\R-C$. Thus, $C-O_f$ is the countable subset consisting of boundary points of components of $\R-C$.

\medskip 
Let $(X,d)$ be a metric space. Given $x\in X$ and $r>0$, the open (resp. closed) $r$-ball centered at $x$ is denoted by $B_r(x)$ (resp. $\ov{B}_r(x)$). Given a subset $\La\subset X$, its open (resp. closed) $r$-neighborhood is denoted by $N_r(\La)$ (resp. $\ov{N}_r(\La)$). A \emph{Lebesgue number} of an open cover $\U$ of $\La$ is defined to be a number $\de>0$ such that, for every $x\in\La$, the $\de$-ball $B_\de(x)$ is contained in some member of $\U$; we denote
\[
\de_\U=\sup\{\de\mid\de\textup{ is a Lebesgue number of }\U\}.
\]
For a subset $U\subset X$ and $r>0$ we define 
\begin{align}\label{eqn:ur}
U^r= \{x\in U\mid B_{r}(x)\subset U\}\subset U. 
\end{align}
A sequence of subsets $W_k\subset X$ is said to be \emph{exponentially shrinking} if the diameters of these subsets converge to zero exponentially fast, that is, there exist constants $A,C>0$ such that
\[
\diam(W_k)\le A\, e^{-C k} 
\]
for all $k$.

If $X$ is a Riemannian manifold and $\Phi$ is a diffeomorphism of $X$, the \emph{expansion factor} of $\Phi$ at $x\in X$ is defined as
\begin{align}\label{eqn:ep} 
\ef(\Phi,x)=\inf_{0\neq v\in T_x X}\frac{\|D_x\Phi(v)\|}{\|v\|}.
\end{align}

\medskip
We now present some dynamical and geometric preliminaries to be used later. For more details we refer the readers to \cite{BH} and \cite{DK}.

\subsection{Topological dynamics} \label{sec:dynamics} 

A continuous action $\Ga\times Z\to Z$ of a topological group on a topological space is \emph{minimal} if $Z$ contains no proper closed $\Ga$-invariant subsets or, equivalently, if every $\Ga$-orbit is dense in $Z$. A point $z\in Z$ is a \emph{wandering point} for an action $\Ga\times Z\to Z$ if there exists a neighborhood $U$ of $z$ such that $g U\cap U=\emptyset$ for all but finitely many $g\in\Ga$. If the space $Z$ is metrizable, then a point $z\in Z$ is \emph{not} a wandering point if and only if there exist a sequence $(g_k)$ of distinct elements in $\Ga$ and a sequence $(z_k)$ in $Z$ converging to $z$ such that $g_kz_k\to z$. For further discussion of dynamical relations between points under group actions, see \cite{KL18a}*{\S4.3}.

A continuous action $\Ga\times Z\to Z$ of a discrete group $\Ga$ on a compact metrizable topological space $Z$ is a \emph{convergence action} if the product action of $\Ga$ on $Z^3$ restricts to a properly discontinuous action on
\[
T(Z)=\{(z_1, z_2, z_3)\in Z^3\mid \card\{z_1, z_2, z_3\}=3\}. 
\]
Equivalently, a continuous action of a discrete group is a convergence action if every sequence $(g_k)$ contains a subsequence $(g_{k_j})$ which is constant or converges to a point $z_+\in Z$ uniformly on compacts in $Z-\{z_-\}$ for some $z_-\in Z$; see \cite{Bowditch}*{Proposition 7.1}. In this situation, the inverse sequence $(g_{k_j}^{-1})$ converges to $z_-$ uniformly on compacts in $Z-\{z_+\}$. The set of such limit points $z_+$ is the \emph{limit set} $\La$ of the action of $\Ga$; this is a closed $\Ga$-invariant subset of $Z$. Observe that a convergence action need not be faithful but it necessarily has finite kernel, provided that $T(Z)\ne \emptyset$. A convergence action on $Z$ is called \emph{uniform} if it is cocompact on $T(Z)$.

\begin{remark}\label{rem:consubgroups} 
Since proper discontinuity of an action is preserved under taking subgroups, if the action of $\Ga$ on $Z$ is a convergence action, then its restriction to any subgroup $\Ga_0<\Ga$ is still a convergence action. 
\end{remark}

Item (1) of the following theorem can be found in \cite{Tukia}*{Theorem 2S}; for item (2) see \cite{Tukia98}*{Theorem 1A} for instance.

\begin{theorem}\label{thm:ne-con-act}
Suppose $\Ga\times Z\to Z$ is a convergence action with limit set $\La$ such that $\card(\La)\ge3$. Then
\begin{enumerate}[label=\textup{(\arabic*)},nosep,leftmargin=*]
\item
$\La$ is perfect and the action is minimal on $\La$.
\item
If the action is uniform and $Z$ is perfect, then $Z=\La$.
\end{enumerate}
\end{theorem}

\subsection{Coarse geometry}\label{sec:coarse_geometry}

A metric space $(X,d)$ is \emph{proper} if the closed ball $\ov{B}_r(x)$ is compact for every $x\in X$ and every $r>0$. Note that proper metric spaces are complete. A metric space $(X,d)$ is called a \emph{geodesic space} if every pair of points $x,y \in X$ can be joined by a geodesic segment $xy$, that is, an isometric embedding of an interval into $X$ joining $x$ to $y$.

\begin{definition}[Quasi-geodesic]
Let $I$ be an interval of $\R$ (or its intersection with $\Z$) and $(X,d)$ a metric space. A map $c:I\to X$ is called an \emph{$(A,C)$-quasigeodesic} with constants $A\geq 1$ and $C\geq 0$ if for all $t,t'\in I$,
\[
\frac{1}{A}|t-t'|-C\le d(c(t),c(t'))\le A|t-t'|+C.
\]
\end{definition}

\begin{definition}[Hyperbolic space]
Let $\de\ge 0$. A geodesic space $X$ is said to be \emph{$\de$-hyperbolic} if for any geodesic triangle in $X$, each side of the triangle is contained in the closed $\de$-neighborhood of the union of the other two sides. A geodesic space is said to be \emph{hyperbolic} if it is $\de$-hyperbolic for some $\de\ge 0$.
\end{definition}

Let $X$ be a proper $\de$-hyperbolic space. Two geodesic rays $\R_{\ge0}\to X$ are said to be \emph{asymptotic} if the Hausdorff distance between their images is finite. Being asymptotic is an equivalence relation on the set of geodesic rays. The set of equivalence classes of geodesic rays in $X$ is called the \emph{visual boundary} of $X$ and denoted by $\pa X$. In view of the Morse lemma for hyperbolic spaces (see \cite{BH}*{Theorem III.H.1.7} or \cite{DK}*{Lemma 11.105} for example), one can also define $\pa X$ as the set of equivalence classes of quasigeodesic rays $\R_{\ge0}\to X$. We will use the notation $x\xi$ for a geodesic ray in $X$ emanating from $x$ and representing the point $\xi\in \pa X$.

Fix $r>2\de$ and let $c_0:\R_{\ge0}\to X$ be a geodesic ray representing $\xi \in \pa X$ with $c_0(0)=x$. A topology on $\pa X$ is given by setting the basis of neighborhoods of $\xi$ to be the collection $\{V_k(\xi)\}_{k\in \N}$, where $V_k(\xi)$ is the set of equivalence classes of geodesic rays $c$ such that $c(0)=x$ and $d(c(k),c_0(k))<r$. This topology extends to the \emph{visual compactification} of $X$
\[
\ov{X}:= X\cup \pa X,
\]
which is a compact metrizable space. We refer to \cite{BH}*{III.H.3.6} for details.

Let $x,y,z \in X$. The \emph{Gromov product} of $y$ and $z$ with respect to $x$ is defined by
\[
(y \cdot z)_x := \frac{1}{2}(d(x,y)+d(x,z)-d(y,z)).
\]
The Gromov product is extended to $X\cup\pa X$ by
\[
(y\cdot z)_x:=\sup \liminf_{k,j\to \infty} (y_k\cdot z_j)_x,
\]
where the supremum is taken over all sequences $(y_k)$ and $(z_j)$ in $X$ such that $\lim y_k=y$ and $\lim z_j=z$.

\begin{definition}[Visual metric]\label{vmetric}
Let $X$ be a hyperbolic space with base point $x\in X$. A metric $d_a$ on $\pa X$ is called a \emph{visual metric} with parameter $a>1$ if there exist constants $k_1, k_2>0$ such that
\[
k_1a^{-(\xi\cdot \xi')_x} \leq d_a(\xi,\xi') \leq k_2a^{-(\xi\cdot \xi')_x}
\]
for all $\xi,\xi'\in \pa X$.
\end{definition}

\noindent For every $a>1$ sufficiently close to $1$, a proper hyperbolic space admits a visual metric $d_a$ which induces the same topology as the topology on $\pa X$ described above. We refer to \cite{BH}*{Chapter III.H.3} for more details on constructing visual metrics.

\medskip
In the rest of the section we discuss hyperbolic groups and their relation to convergence actions.

\begin{definition}[Hyperbolic group]
A finitely generated group $\Ga$ is \emph{word-hyperbolic} (or simply \emph{hyperbolic}) if its Cayley graph with respect to a finite generating set of $\Ga$ is a hyperbolic metric space. A hyperbolic group is called \emph{elementary} if it contains a cyclic subgroup of finite index and \emph{non-elementary} otherwise.
\end{definition}

\noindent The \emph{Gromov boundary} $\pa\Ga$ of a hyperbolic group $\Ga$ is defined as the visual boundary of a Cayley graph $X$ of $\Ga$. The closure of $\Ga\subset X$ in the visual compactification $\ov{X}$ equals $\Ga\cup \pa \Ga$ and is denoted $\ov\Ga$; it is the \emph{visual compactification of $\Ga$}.

Every hyperbolic group $\Ga$ acts on its visual compactification $\ov{\Ga}$ by homeomorphisms. This action is a convergence action; see \cite{Tukia}*{Theorem 3.A} and \cite{Freden}. If a sequence $(c_k)$ in $\Ga$ represents a quasigeodesic ray within bounded distance from a geodesic ray $g\xi$ ($g\in\Ga$, $\xi\in\pa\Ga$), then this sequence, regarded as a sequence of maps $\ov\Ga\to\ov\Ga$, converges to $\xi$ uniformly on compacts in $\ov\Ga-\{\xi'\}$ for some $\xi'\in\pa\Ga$. We will use the following refinement of this property later in the proof of Theorem~\ref{thm:e=>h}:

\begin{lemma}\label{lem:conv}
Let $\Ga$ be a word-hyperbolic group. Suppose that $c:\N_0\to\Ga$, $k\mapsto c_k$, is an $(A,C)$-quasigeodesic ray in $\Ga$ such that
\begin{itemize}[nosep]
\item
the word length of $c_0$ is $\le 1$, and
\item
there exists a subsequence $(c_{k_j})$ converging to a point $\xi\in\pa\Ga$ pointwise on a subset $S\subset\pa\Ga$ with $\card(S)\ge2$.
\end{itemize}
Then the image $c(\N_0)$ is $D$-Hausdorff close to a geodesic $e\xi$ in the Cayley graph $X$ of $\Ga$, where $D$ depends only on $(A,C)$ and the hyperbolicity constant of $X$.
\end{lemma}

\begin{proof}
Since the word length of $c_0$ is $\le 1$, the extended Morse lemma for hyperbolic groups in \cite{DK}*{Lemma 11.105} states that there is a geodesic ray $\id\mu$ ($\mu\in\pa\Ga$) starting at the identity $\id\in\Ga$ such that the Hausdorff distance between the image $c(\N_0)$ and the ray $\id\mu$ in $X$ is bounded above by a uniform constant $D>0$ depending only on $(A,C)$ and $X$.

By the above property, the sequence $(c_k)$ converges to $\mu$ uniformly on compacts in $\ov\Ga-\{\mu'\}$ for some $\mu'\in\pa\Ga$. On the other hand, since $\card(S)\ge2$, there is a point $\nu\in S$ distinct from $\mu'$ such that the subsequence $(c_{k_j})$ converges to $\xi$ on $\{\nu\}\subset\pa\Ga$. Therefore, we must have $\mu=\xi$ and $\id\mu=\id\xi$.
\end{proof}

Furthermore, the action of $\Ga$ on $\pa\Ga$ is a uniform convergence action. In particular, if $\Ga$ is non-elementary then this action has finite kernel (the unique maximal finite normal subgroup of $\Ga$), is minimal, and $\pa\Ga$ is a perfect topological space; compare Theorem~\ref{thm:ne-con-act}. We refer to \cite{DK}*{Lemma 11.130} for more details.

Conversely, Bowditch \cite{Bowditch} gave a topological characterization of hyperbolic groups and their Gromov boundaries as uniform convergence actions $\Ga\times Z\to Z$ of discrete groups on perfect metrizable topological spaces:

\begin{theorem}[Bowditch]\label{thm:bowditch} 
Suppose that $Z$ is a compact perfect metrizable space of cardinality $\ge2$ and $\Ga\times Z\to Z$ is a continuous action of a discrete group, which is a uniform convergence action. Then $\Ga$ is a non-elementary hyperbolic group and $Z$ is equivariantly homeomorphic to the Gromov boundary $\pa\Ga$.
\end{theorem}

\subsection{Symmetric spaces}\label{sec:sspace}

We collect several definitions regarding the geometry of symmetric spaces that are necessary for our presentation in Sections~\ref{sec:ulattice}, \ref{sec:wordhyp} and \ref{sec:examples}. For more details, readers may refer to \cite{BGS}, \cite{Eber} or \cite{KLP17}*{Section 2}.

Throughout Sections~\ref{sec:sspace} and \ref{sec:anosov}, let $G$ be a semisimple Lie group and $X$ the associated symmetric space.

The visual boundary $\pa X$ of $X$ has a topological spherical building structure, the Tits building associated to $X$. Let $\ad$ denote the model apartment of this building and $W$ the Weyl group acting isometrically on $\ad$. We will fix a chamber $\sid\subset\ad$ and call it the \emph{spherical model Weyl chamber}; it is a fundamental domain for the $W$-action on $\ad$. Let $w_0\in W$ denote the unique element sending $\sid$ to the opposite chamber $-\sid$. Then the \emph{opposition involution} $\iota$ of the model chamber $\sid$ is defined as $\iota=-w_0$.

We denote by $\De:=\De_{\mathrm{mod}}$ the \emph{model Weyl chamber} corresponding to $\sid$, a fundamental domain for the $W$-action on the model (maximal) flat $F_{\mathrm{mod}}$ of $X$. We will use the notation $d_\De$ for the $G$-invariant 
\emph{$\De$-valued distance function} on $X$.

Consider the induced action of $G$ on $\pa X$. Every orbit intersects every chamber exactly once, so there is a natural identification $\pa X/G\cong\sid$. The projection $\theta:\pa X\to\pa X/G$ is called the \emph{type} map. Let $\tad$ be a face of $\sid$. We will assume, in what follows, that the simplex $\tad$ is $\iota$-invariant. The \emph{$\tad$-flag manifold} $\Flag(\tad)$ is the space of simplices of type $\tad$ in $\pa X$. It has a structure of a compact smooth manifold and can be identified with the quotient space $G/P_\tau$, where $P_\tau<G$ is the stabilizer subgroup of a simplex $\tau\subset\pa X$ of type $\theta(\tau)=\tad$. Two simplices $\tau_1$ and $\tau_2$ in $\pa X$ are said to be \emph{antipodal} if they are opposite in an apartment containing both of them; their types are related by $\theta(\tau_2)=\iota\theta(\tau_1)$. A subset $E$ of $\Flag(\tad)$ is said to be \emph{antipodal} if any two distinct elements of $E$ are antipodal. A map into $\Flag(\tad)$ is \emph{antipodal} if it is injective and has antipodal image. We shall always assume that $\Flag(\tad)$ is equipped with an auxiliary Riemannian metric.

Let $\tau\subset\pa X$ be a simplex of type $\tad$. The \emph{star} $\st(\tau)\subset\pa X$ is the union of all chambers containing $\tau$. We denote by $\partial\,\st(\tau)$ the union of all simplices in $\st(\tau)$ which do not contain $\tau$, and define $\ost(\tau):=\st(\tau)-\partial\,\st(\tau)$. We denote by $W_{\tad}$ the $W$-stabilizer of $\tad$. A subset $\Theta\subset\sid$ is \emph{$W_\tad$-convex} if its symmetrization $W_\tad\Theta\subset\ad$ is convex; for such a subset $\Theta$ we define $\st_\Theta(\tau):=\st(\tau)\cap\theta^{-1}(\Theta)$.

For $x\in X$ and $S\subset\pa X$ the \emph{Weyl cone} $V(x,S)\subset X$ is defined to be the union of rays $x\xi$ ($\xi\in S$). For a flat $f\subset X$ the \emph{parallel set} $P(f)$ is the union of all flats $f'$ such that $\pa f'=\pa f$. For a pair of antipodal simplices $\tau_1,\tau_2\subset\pa X$ we define $P(\tau_1,\tau_2)$ to be the parallel set of the unique minimal flat whose ideal boundary contains $\{\tau_1,\tau_2\}$; equivalently, it is the union of all geodesics in $X$ which are forward/backward asymptotic to points in $\tau_1$ and $\tau_2$. 

\begin{definition}[Finsler geodesic]
A \emph{$\tad$-Finsler geodesic} is a continuous path $c:I\to X$ contained in a parallel set $P(\tau_1,\tau_2)$ for $\tau_1,\tau_2\in\Flag(\tad)$ such that $c(t_2)\in V(c(t_1),\st(\tau_2))$ for all subintervals $[t_1,t_2]\subset I$. It is \emph{uniformly $\tad$-regular} if $c(t_2)\in V(c(t_1),\st_\Theta(\tau_2))$ for some $W_\tad$-convex compact subset $\Theta\subset\mathrm{int}_\tad(\sid)$.
\end{definition}

Let $xy$ be a \emph{$\tad$-regular} oriented geodesic segment in $X$ and $x\xi$ the geodesic ray extending $xy$. We denote $\tau(xy)$ the unique simplex $\tau\in\Flag(\tad)$ such that $\xi\in\ost(\tau)$. Let $\Theta\subset\mathrm{int}_\tad(\sid)$. Then $xy$ is \emph{$\Theta$-regular} if $\theta(\xi)\in\Theta$. If $x_1x_2$ is $\Theta$-regular, we define its \emph{$\Theta$-diamond} by
\[
\Diamond_\Theta(x_1,x_2)=V(x_1,\st_\Theta(\tau(x_1x_2)))\cap V(x_2,\st_\Theta(\tau(x_2x_1))).
\]

\begin{definition}[Morse quasigeodesic]
A (discrete) quasigeodesic $q:I\cap\Z\to X$ is \emph{$(\Theta,R)$-Morse} if for every subinterval $[t_1,t_2]\subset I$ the subpath $q|_{[t_1,t_2]\cap\Z}$ is contained in the $R$-neighborhood of the diamond $\Diamond_\Theta(x_1,x_2)$ such that $d(x_i,q(t_i))\le R$ for $i=1,2$. 
\end{definition}

It is shown in \cite{KLP17}*{Theorem 5.53} that $(\Theta,R)$-Morse quasigeodesic rays (resp. lines) are uniformly Hausdorff close to a uniformly $\tad$-regular Finsler geodesic rays (resp. lines).

\subsection{Anosov subgroups}\label{sec:anosov}

We recall the definitions of \emph{Anosov} and \emph{non-uniformly Anosov} subgroups given by Kapovich, Leeb and Porti in \cite{KLP17}*{Definitions 5.43 and 5.62}. Many other equivalent characterizations of the Anosov subgroups are established in \cite{KLP17}*{Theorem 5.47} and their equivalence with the original definitions given by Labourie \cite{Lab} and Guichard and Wienhard \cite{GW} is proven in \cite{KLP17}*{Section 5.11}.

\begin{definition}\label{def:anosov}
Let $\tad\subset\sid$ be an $\iota$-invariant face. A subgroup $\Ga$ of a semisimple Lie group $G$ is \emph{$\tad$-boundary embedded} if
\begin{enumerate}[label=(\alph*),nosep,leftmargin=*]
\item
it is a hyperbolic group;
\item
there is an antipodal $\Ga$-equivariant continuous map (called a \emph{boundary embedding})
\[
\psi:\pa\Ga\to\Flag(\tad).
\]
\end{enumerate}
A $\tad$-boundary embedded subgroup $\Ga<G$ with a boundary embedding $\psi$ is \emph{$\tad$-Anosov} if
\begin{enumerate}[label=(\alph*),nosep,leftmargin=*,start=3]
\item
for every $\xi\in\pa\Ga$ and for every geodesic ray $r:\N_0\to\Ga$ starting at $\id\in\Ga$ and asymptotic to $\xi$, the expansion factor (see \eqref{eqn:ep}) satisfies
\begin{align*}
\ef(r(k)^{-1},\psi(\xi))\ge A\, e^{C k}
\end{align*}
for $k\in\N_0$ with constants $A,C>0$ independent of the point $\xi$ and the ray $r$.
\end{enumerate}
A $\tad$-boundary embedded subgroup $\Ga<G$ with a boundary embedding $\psi$ is \emph{non-uniformly $\tad$-Anosov} if
\begin{enumerate}[label=(\alph*),nosep,leftmargin=*,start=4]
\item
for every $\xi\in\pa\Ga$ and for every geodesic ray $r:\N_0\to\Ga$ starting at $\id\in\Ga$ and asymptotic to $\xi$, the expansion factor satisfies
\begin{align*}
\sup_{k\in\N_0}\,\ef(r(k)^{-1},\psi(\xi))=+\infty.
\end{align*}
\end{enumerate}
A $\tad$-boundary embedded subgroup $\Ga<G$ is said to be \emph{non-elementary} if it is a non-elementary hyperbolic group.
\end{definition}

\begin{remark}\phantomsection
\begin{enumerate}[label=(\alph*),nosep,leftmargin=*]
\item
In fact, for $\tad$-boundary embedded subgroups, the conditions (c) and (d) are equivalent; see \cite{KLP17}*{Theorem 5.47}. For the purpose of this paper, the definition of non-uniformly $\tad$-Anosov subgroups will suffice.
\item
A $\tad$-Anosov subgroup $\Ga<G$ may have other boundary embeddings $\varphi:\pa\Ga\to\Flag(\tad)$ besides the map $\psi$ which appears in the definition; see \cite{KLP14}*{Example 6.20}. However, the boundary embedding $\psi$ as in the conditions (c) and (d) is unique; its image is the \emph{$\tad$-limit set} of $\Ga$ in $\Flag(\tad)$,
\[
\psi(\pa\Ga)= \La_\Ga(\tad)\subset\Flag(\tad).
\]
We thus will refer to the map $\psi$ as the \emph{asymptotic embedding} for $\Ga< G$.
\end{enumerate}
\end{remark}

\section{Expansion and meandering hyperbolicity}\label{sec:statement}

Throughout Sections~\ref{sec:statement} and \ref{sec:proof}, we let $(M,d)$ be a proper geodesic metric space and suppose that a discrete group $\Ga$ acts continuously on $M$ with a non-empty invariant compact subset $\La\subset M$, no point of which is isolated in $M$.

\begin{remark}
Note that we do not assume faithfulness of the action of $\Ga$ on $\La$ and even on $M$.
\end{remark}

\noindent In this situation, one considers two conditions on $\rho$ called the \emph{expansion} and \emph{meandering hyperbolicity} conditions, respectively. The meandering hyperbolicity condition will require the expansion condition.

In the present section, we define the expansion condition, and draw the key Lemma~\ref{lem:nest} as well as its various consequences (Sections~\ref{sec:expansion} and \ref{sec:encoding}). We then define the meandering hyperbolicity condition (Section~\ref{sec:hyperbolicity}).

\subsection{The expansion condition}\label{sec:expansion}

In order to define the expansion condition we need a little preparation.

Let $f$ be a homeomorphism of $M$. Given $\la>1$ and $U\subset M$, we say $f$ is \emph{$(\la,U)$-expanding} (or \emph{$\la$-expanding on $U$}) if
\[
d(f(x),f(y))\ge\la\cdot d(x,y)
\]
for all $x,y\in U$. In this case, we also say $U$ is a \emph{$(\la,f)$-expanding subset}. Note that $f$ is $(\la,U)$-expanding if and only if $f^{-1}$ is \emph{$(\la^{-1},f(U))$-contracting}, that is,
\[
d(f^{-1}(x),f^{-1}(y))\le\frac{1}{\la}\cdot d(x,y)
\]
for all $x,y\in f(U)$.

Given $\de>0$, a $(\la,U)$-expanding homeomorphism $f$ is said to be \emph{$(\la,U;\de)$-expanding} if
\[
B_{\la\eta}(f(x))\subset f(B_\eta(x))\;\textup{ whenever }B_\eta(x)\subset U\textup{ and }\eta\le\de.
\]
Clearly, if $f$ is $(\la,U;\de)$-expanding then it is also $(\la,U;\de')$-expanding for every $\de'\le\de$. 

\begin{lemma}
If $M$ is a geodesic metric space, then every $(\la,U)$-expanding homeomorphism is also $(\la,U;\de)$-expanding for every $\de$.
\end{lemma}

\begin{proof}
Suppose to the contrary that there exists a point
\[
y\in B_{\la\eta}(f(x))\setminus f(B_\eta(x)).
\]
Let $c$ be a geodesic path in $M$ connecting $f(x)$ and $y$; the length of this geodesic is less than $\la\eta$. 
Since $f(B_\eta(x))$ is open in $M$, the path $c$ crosses the boundary of $f(B_\eta(x))$ 
at a point $z$. Because $f$ is $\la$-expanding on the set $U$ containing $B_\eta(x)$ and 
$f^{-1}(z)\in \partial B_\eta(x)$, we have
\[
d(f(x), y)\ge d(f(x), z)\ge \la d(x, f^{-1}(z))=\la \eta.  
\]
This contradicts the assumption that $y\in B_{\la\eta}(f(x))$, that is, $d(f(x), y)< \la\eta$. 
\end{proof}

The implication in the lemma does not hold for general metric spaces. As a simple example, consider a compact metric space $M$ of diameter $D$, which is not a singleton, but contains an isolated point $x$. Consider $f=\mathrm{Id}_M$ and take a neighborhood $U$ of $x$ and $\eta>0$ such that $\{x\}=B_\eta(x)\subset U=\{x\}$. Then $f$ is $(\la,U)$-expanding for {\em any} $\la>1$. Thus, taking $\la$ such that $\la\eta\ge D$, we see that
\[
M=B_{\la\eta}(f(x))\subsetneq f(B_\eta(x))=B_\eta(x)=\{x\}.
\]

However, we note the following fact:

\begin{lemma}\label{lem:smallU}
Suppose that $f$ is $(\la,U)$-expanding, where $U$ is a bounded open subset of $M$. Then for every $\de>0$ there exists $\de'= \de'_U>0$ such that $f$ is $(\la,U';\de')$-expanding with $U':=\intr U^\de \subset U$.
\end{lemma}

\begin{proof} 
Since $f(U^\de)$ is compact, we have $\ep:=d(f(U^\de), M-f(U))>0$. Now we let $\de':=\la^{-1}\cdot\min\{\de,\ep\}$. If $\eta\le\de'$ and $B_{\eta}(x)\subset U'$, in particular, $f(x)\in f(U^\de)$, then we have $B_{\la\eta}(f(x))\subset B_\ep(f(x))\subset f(U)$ and hence $f^{-1}[B_{\la\eta}(f(x))]\subset B_{\eta}(x)$, since $f^{-1}$ is $(1/\la,f(U))$-contracting. Therefore, we conclude that $B_{\la\eta}(f(x))\subset f(B_{\eta}(x))$.
\end{proof}

\medskip 
We are now ready to define the expansion condition. Let $\Homeo(M)$ denote the group of locally bi-Lipschitz homeomorphisms of $M$, that is, homeomorphisms of $M$ whose restrictions to compact subsets are bi-Lipschitz.

\begin{definition}[Expansion]\label{def:expansion}
Let $\rho:\Ga\to\Homeo(M)$ be an action with a non-empty compact invariant subset $\La$, no point of which is isolated in $M$. The action $\rho$ is said to be \emph{expanding at $\La$} if there exist
\begin{itemize}[parsep=0pt,itemsep=0pt]
\item
a finite index set $\I$,
\item
a cover $\U=\{U_i\subset M \mid i\in\I\}$ of $\La$ by open (and possibly empty) subsets $U_i$,
\item
a map $s:\I\to\Ga$, $i\mapsto s_i$,
\item 
and positive real numbers $L\ge\la>1$ and $\de\le\de_\U$
\end{itemize}
such that, for every $i\in\I$, the map $\rho(s_i^{-1})$ is
\begin{enumerate}[label=(\roman*),parsep=0pt,itemsep=0pt]
\item $L$-Lipschitz on $N_{\de}(\La)$, and
\item $(\la,U_\al;\de)$-expanding,
\end{enumerate}
and that the image $\vS:=\{s_i \mid i\in\I\}\subset\Ga$ of the map $s$ is a symmetric generating set of the group $\Ga$.

In this case, the datum $\D:=(\I,\U,\vS,\allowbreak\de,L,\la)$ (or, occasionally, any subset thereof) will be referred to as an \emph{expansion datum} of $\rho$.
\end{definition}

If $\rho:\Ga\to\Homeo(M)$ is an expanding action, let $|\cdot|_\vS$ (resp. $d_\vS$) denote the word length (resp. the word metric) on the group $\Ga$ with respect to the generating set $\vS$ from Definition~\ref{def:expansion}. Then the $L$-Lipschitz property (i) implies that
\begin{align}\label{eqn:lip}
\textup{the map }\rho(g)\textup{ is }L^k\textup{-Lipschitz on } N_{\de/L^{k-1}}(\La)
\end{align}
for every $g\in\Ga$ with $|g|_\vS=k\in\N$.

\begin{remark}\phantomsection\label{rem:expansion}
A few more remarks are in order.
\begin{enumerate}[label=(\alph*),nosep,leftmargin=*]
\item
Expanding actions appear naturally in the context of Anosov actions on flag manifolds \cite{KLP17}*{Definition 3.1} and hyperbolic group actions on their Gromov boundaries equipped with visual metrics \cite{Coo}. See Section~\ref{sec:wordhyp} for further discussion.
\item
The symmetry of the generating set $\vS$ means that $s\in\vS$ if and only if $s^{-1}\in\vS$. This implies that $\rho(s_i)$ is $L$-bi-Lipschitz on $N_{\de/L}(\La)$ for all $i\in\I$.
\item
Suppose that $M$ is a Riemannian manifold and the action $\rho$ is by $C^1$-diffeomorphisms. Then, $\rho$ is expanding provided that for every $x\in \La$ there exists $g\in\Ga$ such that $\ef(\rho(g),x)>1$ (see \eqref{eqn:ep}). Indeed, by compactness of $\La$, there exists a finite cover $\U=\{U_i \mid i\in\I\}$ of $\La$ and a collection $\vS$ of group elements $s_i\in\Ga$ such that each $\rho(s_i^{-1})$ is $(\la,U_i;\de_\U)$-expanding. By adding, if necessary, extra generators to $\vS$ with empty expanding subsets, we obtain the required symmetric generating set of $\Ga$.
\item
If $U_i=\emptyset$ for some $i\in\I$ then the condition (ii) is vacuous for this $i$. Otherwise, it implies that the inverse $\rho(s_i)$ is $(\la^{-1},\rho(s_i^{-1})[U_i])$-contracting.
\item
The condition (ii) can be relaxed to the mere $(\la,U_i)$-expanding condition. Namely, we may first modify the cover $\U=\{U_i\mid i \in\I\}$ so that $U_i$ are all bounded. Then, in view of Lemma~\ref{lem:smallU}, we can modify it further to $\U'=\{U'_i \mid  i \in \I \}$, where $U'_i:=\intr U^\de_i$ as in the lemma. For each $i\in\I$ we also let $\de'_{U_i}$ denote the number $\de'_{U}$ given by the lemma, and set
\[
\de':= \min\{\de_{\U'},\de'_{U_i}\mid i\in\I\}.
\]
After such modification $\U'$ is still an open cover of $\La$ and the maps $\rho(s_i^{-1})$ are $(\la,U'_i;\de')$-expanding.
\item
The map $s:\I\to\vS\subset\Ga$ is not necessarily injective: the $\rho$-image of an element of $\Ga$ can have several expansion subsets. See Examples~\ref{ex:E}, \ref{ex:E0} and \ref{ex:E1}.
\item
Clearly, the properties (i) and (ii) also hold on the closures $\ov{N}_\de(\La)$ and $\ov{U_i}$, respectively.
\end{enumerate} 
\end{remark}

It is possible that an action $\rho:\Ga\to\Homeo(M)$ is expanding with different expansion data. In such a case we define a relation on the expansion data as follows.

\begin{definition}\label{def:refine}  Let $\rho:\Ga\to\Homeo(M)$ be expanding at $\La$. Let $\D_0=(\I_0,\U_0,\vS_0,\allowbreak\de_0,L_0,\la_0)$ and $\D_1=(\I_1,\U_1,\vS_1,\allowbreak\de_1,L_1,\la_1)$ be its expansion data, where $\U_k=\{U_{k,i} \mid i\in\I_k\}$ and $\vS_k=\{s_{k,i} \mid i\in\I_k\}$ for $k=0,1$. We write $\U_0 \prec \U_1$ (resp. $\vS_0 \prec \vS_1$) if $\I_0\subset\I_1$ and $U_{0,i} \subset U_{1,i}$ (resp. $s_{0,i}=s_{1,i}$) for every $i\in \I_0$. Furthermore, we write $\D_0\prec\D_1$ and say $\D_1$ is a \emph{refinement} of $\D_0$ if  
\[
\I_0\subset\I_1,\; \U_0\prec\U_1,\; \vS_0\prec\vS_1,\; \de_0>\de_1,\; L_0\le L_1,\; \la_0\ge\la_1.
\]
\end{definition}

Note the strict inequality $\de_0>\de_1$, which will be pertinent to Section~\ref{sec:cont}. 

\begin{definition}\label{def:t-refine}
For any expansion datum $\D$ we can define a refinement by replacing $\de$ with a strictly smaller positive number and leaving other objects intact. We say such a refinement is \emph{trivial}.
\end{definition}

\subsection{Toy examples of expanding actions}\label{sec:toy}

The most basic example of an expanding action is a cyclic hyperbolic group of M\"{o}bius transformations acting on the unit circle $S^1$. Namely, we consider the Poincar\'e (conformal) disk model of $\H^2$ in $\C=\R^2$ and endow $M=S^1=\pa\H^2$ with the induced Euclidean metric. Let $\ga\in \Isom(\H^2)$ be a hyperbolic element and $\Ga= \langle \ga \rangle \cong \Z$. The \emph{limit set} of the group $\Ga$ is $\La=\{\ga_-, \ga_+\}$, with $\ga_+$ the attracting and $\ga_-$ the repelling fixed points of $\ga$ in $M$. Expanding subsets $U_\al, U_\be$ for $\ga, \ga^{-1}$ are sufficiently small arcs containing $\ga_-, \ga_+$, respectively.

Explicit expanding subsets can be found by considering the isometric circles $I_{\ga}$ and $I_{\ga^{-1}}$ of $\ga$ and $\ga^{-1}$, respectively. See Figure~\ref{fig:toy}(Left). (For the definition of isometric circles (or spheres) and their relation to Ford and Dirichlet fundamental domains, we refer to \cite{Maskit}*{IV.G}.) The arc of $I_\ga$ in $\H^2$ is a complete geodesic which is the perpendicular bisector of the points $o$ and $\ga^{-1}(o)$, where $o$ denotes the Euclidean center of the Poincar\'e disk. Then we obtain a $(\la,\ga)$-expanding subset $U_\al$ (with $\la>1$) by cutting down slightly the open arc of $S^1=\pa\H^2$ inside $I_\ga$. See also the discussion in the beginning of Section~\ref{sec:ex1}.

\begin{figure}[ht]
\labellist
\pinlabel {$o$} at 386 156
\pinlabel {$\ga$} at 374 256
\pinlabel {\scriptsize $\ga(o)$} at 280 196
\pinlabel {\scriptsize $\ga^{-1}(o)$} at 480 196
\pinlabel {$I_\ga$} at 434 400
\pinlabel {$I_{\ga^{-1}}$} at 330 400
\pinlabel {$U_\al$} at 560 200
\pinlabel {$U_\be$} at 196 200
\pinlabel {$\la_-$} at 520 290
\pinlabel {$\la_+$} at 234 290
\pinlabel {$\wt{U}_{\al_1}$} at 1170 390
\pinlabel {$\wt{U}_{\be_1}$} at 1060 390
\pinlabel {$\wt{U}_{\al_2}$} at 950 200
\pinlabel {$\wt{U}_{\be_2}$} at 1004 104
\pinlabel {$\wt{U}_{\al_3}$} at 1206 104
\pinlabel {$\wt{U}_{\be_3}$} at 1270 200
\endlabellist
\centering
\includegraphics[width=0.9\textwidth]{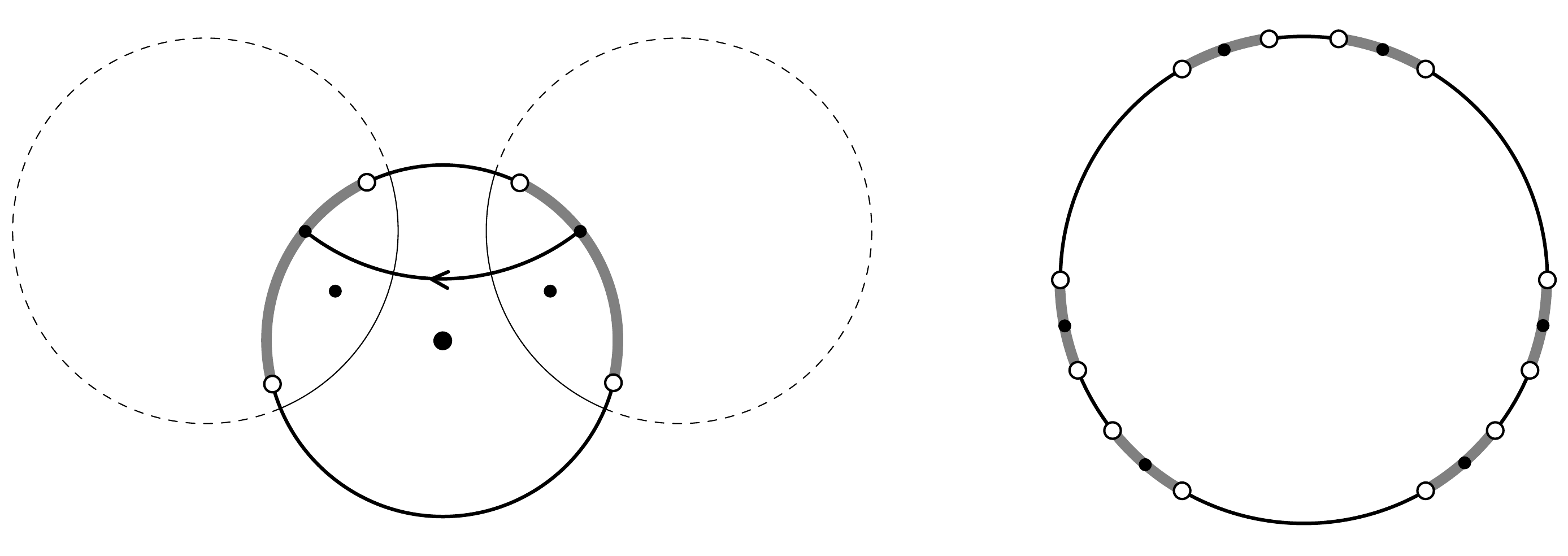}
\caption{Expanding arcs for $\ga$ and $\ga^{-1}$ are colored gray in both examples.\\(Left) A hyperbolic transformation $\ga$ of $\H^2$. (Right) A covering of degree $3$.}
\label{fig:toy}
\end{figure}

\begin{example}[$k$-fold non-trivial covering]\label{ex:E} 
A more interesting example is obtained by taking a degree $k >1$ covering $p: S^1\to S^1$ of the preceding example. See Figure~\ref{fig:toy}(Right) for the case of $k=3$. The preimage of $\La=\{\ga_-, \ga_+\}$ consists of $2k$ points and we can lift $\ga$ to a diffeomorphism $\wt\ga: S^1\to S^1$ fixing all these points. Let $\wt\rho: \Ga=\langle \ga \rangle \cong\Z\to\Diff(S^1)$ be the homomorphism sending the generator $\ga$ of $\Ga$ to $\wt\ga$. The preimages $p^{-1}(U_\al)$, $p^{-1}(U_\be)$ break into connected components
\[
\wt{U}_{\al_i},\;\wt{U}_{\be_i}\quad(i=1,\ldots,k)
\]
and the mappings $\wt\ga$ and $\wt\ga^{-1}$ act as expanding maps on each of these components. Therefore, we set
\[
\I= \{\al_1,\ldots,\al_k, \be_1,\ldots,\be_k\}
\]
and define the map $s:\I\to \vS$,
\[
\al_i\mapsto s_{\al_i}= \ga,\quad \be_i\mapsto s_{\be_i}= \ga^{-1}
\]
from this index set to the generating set $\vS=\{\ga, \ga^{-1}\}$ of $\Ga$. Then $\wt{U}_{\al_i}$, $\wt{U}_{\be_i}$ will be expanding subsets for the actions $\wt\rho(s_{\al_i})$, $\wt\rho(s_{\be_i})$ on $S^1$. (Note that the entire preimage $p^{-1}(U_\al)$ (resp. $p^{-1}(U_\be)$) is \emph{not} an expanding subset for the action $\wt\rho(s_{\al_i})$ (resp. $\wt\rho(s_{\be_i})$).)

The same construction works for surface group actions; see Example~\ref{ex:E0}.
\end{example}

A trivial example where $U_i=\emptyset$ for an index $i\in\I$ is the action of a cyclic group $\Ga= \langle \ga \rangle \cong \Z$ generated by a loxodromic transformation $\ga(z)=mz$, $|m|>1$, on $M=\C$ (with the standard Euclidean metric) and $\La=\{0\}$. Any open subset of $M$ containing $0$ is an expanding subset for $\ga$, while the expanding subset for $\ga^{-1}$ is empty.

\subsection{Expansion enables encoding}\label{sec:encoding}

Given an expanding action $\rho:\Ga\to\Homeo(M)$ with an invariant subset $\La$ the expansion condition (Definition~\ref{def:expansion}) enables us to encode points of $\La$ by sequences in the finite index set $\I$. Roughly speaking, for each $x\in\La$, we keep zooming in toward $x$ using the ``microscope,'' namely, the expanding maps $\rho(s_{\al(k)}^{-1})$. This is the heuristic Sullivan himself used in his talk according to F. Bonahon. (We thank Bonahon for telling us his personal anecdote.)

\begin{definition}[Codes]\label{def:code}
Suppose an action $\rho:\Ga\to\Homeo(M)$ is expanding at $\La$ with a datum $\D=(\I,\U,\vS,\allowbreak\de,L,\la)$. Let $x\in\La$ and fix a real number $\eta$ such that $0<\eta\le\de$. Given a sequence $\al:\N_0\to\I,\;k\mapsto\al(k)$, define a sequence $p:\N_0\to\La,\;k\mapsto p_k$ inductively by
\begin{align*}
p_0&=x,\\
p_{k+1}&=\rho(s_{\al(k)}^{-1})(p_k).
\end{align*}
The sequence $\al$ (or the pair $(\al,p)$ of sequences) is called a \emph{$(\D,\eta)$-code} for $x$ if
\[B_\eta(p_k)\subset U_{\al(k)}\]
for every $k\in \N$. (When $\D$ is understood we simply say an \emph{$\eta$-code} in order to emphasize the use of $\eta$-balls). A $(\D,\eta)$-code $\al$ for $x$ is said to be \emph{special} if $\al(0)$ satisfies $B_\eta(x)\subset U_{\al(0)}$. We denote by $\Code_x(\D,\eta)$ the set of all $(\D,\eta)$-codes for $x\in\La$.
\end{definition}

Every $\eta$-ball centered at a point on $\La$ is contained in some member of $\U$ since $\eta\le \de \le \de_{\U}$. This enables us to construct codes. A code gives rise to a family of $\eta$-balls centered at points in $\La$ each of which is a $\la$-expanding domain for the action on $M$ of some element of $\Ga$.

\begin{remark}\phantomsection\label{rem:code}
\begin{enumerate}[label=(\alph*),nosep,leftmargin=*]
\item
That we do not require $\al(0)$ to satisfy $B_\eta(x)\subset U_{\al(0)}$ in general is Sullivan's trick, which will be useful to prove equivariance of the conjugacy $\phi$ in Section~\ref{sec:equiv}.
\item
The requirement $B_\eta(p_k)\subset U_{\al(k)}$ implies $U_{\al(k)}\neq\emptyset$ for $k\in\N$. Thus only for the initial value $\al(0)$ of codes $\al$ can we possibly have $U_{\al(0)}=\emptyset$.
\item
Since $\de\le\de_{\U}$, we have $\Code_x(\D,\eta)\neq\emptyset$ for all $x\in\La$ and $0<\eta\le\de$, and special codes always exist for any point. Moreover, if $\eta'\le\eta\le\de$, then $\Code_x(\D,\eta)\subset\Code_x(\D,\eta')$. Indeed, if $(\al,p)$ is an $\eta$-code for $x$, then $B_{\eta'}(p_k)\subset B_\eta(p_k)\subset U_{\al(k)}$ for all $k\in\N$, which implies that $(\al,p)$ is an $\eta'$-code for $x$ as well.
\item
If $\D_0$ and $\D_1$ are expansion data of $\rho$ and $\D_0\prec \D_1$, then for any $\eta\le\de_1\le\de_0$, a $(\D_0,\eta)$-code $\al:\N_0\to\I_0\subset\I_1$ for $x$ may be regarded as a $(\D_1,\eta)$-code. Slightly abusing the notation we write this relation as $\Code_x(\D_0,\eta)\subset\Code_x(\D_1,\eta)$.
\end{enumerate}
\end{remark}

\begin{definition}[Rays]\label{def:ray}
Let $\rho:\Ga\to\Homeo(M)$ be expanding at $\La$ with the datum $\D$. Given a $(\D,\eta)$-code $\al$ for $x\in\La$, the \emph{$(\D,\eta)$-ray} (or simply \emph{$\eta$-ray}) associated to $\al$ is a sequence $c^\al:\N_0\to\Ga$ defined by
\begin{align*}
c^\al_k=s_{\al(0)}s_{\al(1)}\cdots s_{\al(k)}.
\end{align*}
We set $\Ray_x(\D,\eta):=\{c^\al\mid\al\in\Code_x(\D,\eta)\}$ 
and interpret $c$ as a map
\[
c:\Code_x(\D,\eta)\to\Ray_x(\D,\eta),\quad\al\mapsto c^\al.
\] 
\end{definition}

\begin{remark}\phantomsection\label{rem:ray}
\begin{enumerate}[label=(\alph*),nosep,leftmargin=*]
\item
The initial point $c^\al_0= s_{\al(0)}$ of $c^\al$ is an element of $\vS$.
\item 
From the definition of codes, it is easy to check that $\rho(c^\al_k)(p_{k+1})=p_0=x$ for every $k\in \N_0$.
\item
Every ray $c^\al$ defines an edge-path in the Cayley graph $(\Ga,d_\vS)$ of $\Ga$ (with respect to the generating set $\vS$). Note that each word $s_{\al(0)}^{-1}c^\al_k=s_{\al(1)}\cdots s_{\al(k)}$ $(k>1)$ is reduced, since an appearance of $s^{-1}s$ for $s\in\vS$ would imply that the composite map $\rho(s^{-1})\rho(s)$, which is the identity, is $\la^2$-expanding on some non-empty open subset of $M$.
\item
If $(\vS_0,\de_0)=\D_0\prec\D_1=(\vS_1,\de_1)$ are expansion data of $\rho$ and $\eta\le\de_1\le\de_0$, then we have $\Ray_x(\D_0,\eta)\subset\Ray_x(\D_1,\eta)$ as in Remark~\ref{rem:code}(d).
\end{enumerate}
\end{remark}

Now we will see that, for every $x\in\La$, each $\eta$-code for $x$ gives rise to a nested sequence of neighborhoods of 
$x$ whose diameters tend to $0$ exponentially fast.

\begin{lemma}\label{lem:nest}
Let $\rho:\Ga\to\Homeo(M)$ be expanding at $\La$ with a datum $\D=(\I,\U,\vS,\allowbreak\de,L,\la)$ and let $0<\eta\le\de$. If $(\al,p)$ is a $(\D,\eta)$-code for $x\in\La$, then the sequence of neighborhoods of $x$
\begin{align*}
x\in\rho(c^\al_k)[B_\eta(p_{k+1})]\quad(k\in\N_0)
\end{align*}
is nested and exponentially shrinking. More precisely, we have
\begin{enumerate}[label=\textup{(\roman*)},parsep=0ex,itemsep=0.5ex]
\item
$\rho(c^\al_k)[B_\eta(p_{k+1})]\subset\rho(c^\al_{k-1})[B_\eta(p_k)]$ for $k\in\N$, and
\item
$\rho(c^\al_k)[B_\eta(p_{k+1})] \subset B_{L\eta/\la^{k}}(x)$ for all $k\in\N_0$.
\end{enumerate}
Consequently, we have the equality
\[
\{x\}=\bigcap_{k=0}^\infty\rho(c^\al_k)[B_\eta(p_{k+1})].
\]
The map $\rho(c^\al_k)^{-1}$ is $\la^k/L$-expanding on the neighborhood $\rho(c^\al_k)[B_\eta(p_{k+1})]$ of $x$ for all $k\in\N_0$, and the map
\[
\rho(c^\al_k)^{-1}\rho(c^\al_j)=\rho(s_{\al(k)}^{-1}s_{\al(k-1)}^{-1}\cdots s_{\al(j+1)}^{-1})
\]
is $\la^{k-j}$-expanding on $\rho(c^\al_j)^{-1}\big[\rho(c^\al_k)[B_{\eta}(p_{k+1})]\big]$ for all $0\le j< k$.
\end{lemma}

\begin{figure}[ht]
\labellist
\pinlabel {$p_{k}$} at 58 64
\pinlabel {$p_{k+1}$} at 252 64
\pinlabel {$\rho(s_{\al(k)}^{-1})$} at 130 86
\pinlabel {$\rho(s_{\al(k)})$} at 130 48
\pinlabel {$B_\eta(p_k)$} at 90 28
\pinlabel {$B_\eta(p_{k+1})$} at 240 30
\pinlabel {\scriptsize $\rho(s_{\al(k)})[B_\eta(p_{k+1})]$} at 50 44
\pinlabel {$\rho(s_{\al(k)}^{-1})[B_\eta(p_k)]$} at 304 6
\endlabellist
\centering
\includegraphics[width=0.8\textwidth]{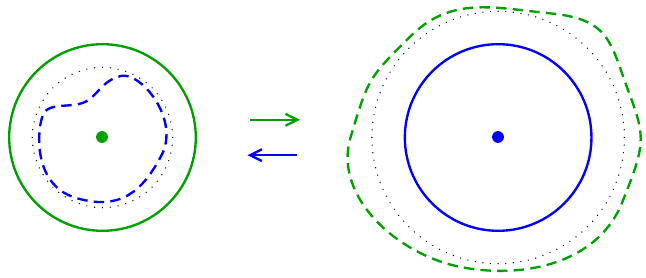}
\caption{Actions of $\rho(s_{\al(k)}^{-1})$ and $\rho(s_{\al(k)})$ for $k\in\N$.}
\label{fig:01}
\end{figure}

\begin{proof}
Suppose $(\al,p)$ is an $\eta$-code for $x\in\La$ and let $k\in\N$. Since $B_\eta(p_k)\subset U_{\al(k)}$ and the homeomorphism $\rho(s_{\al(k)}^{-1})$ is $(\la,U_{\al(k)};\de)$-expanding with $\rho(s_{\al(k)}^{-1})(p_k)=p_{k+1}$, we have by definition
\begin{align}\label{eqn:expand}
B_{\la\eta}(p_{k+1})
\subset \rho(s_{\al(k)}^{-1})[B_\eta(p_k)]
\subset \rho(s_{\al(k)}^{-1})[U_{\al(k)}].
\end{align}
On the other hand, the inverse map $\rho(s_{\al(k)})$ is $(\la^{-1},\rho(s_{\al(k)}^{-1})[U_{\al(k)}])$-contracting as we saw in Remark~\ref{rem:expansion}(d). Since $B_{\eta}(p_{k+1})\subset\rho(s_{\al(k)}^{-1})[U_{\al(k)}]$, we obtain
\begin{align}\label{eqn:contract}
\rho(s_{\al(k)})[B_\eta(p_{k+1})]
\subset B_{\eta/\la}(p_k)
\subset B_\eta(p_k).
\end{align}
See Figure~\ref{fig:01}. By a similar reasoning, we inductively obtain
\begin{align*}
\rho(s_{\al(k-1)}s_{\al(k)})[B_\eta(p_{k+1})]
&\subset B_{\eta/\la^2}(p_{k-1}),\\
&\;\;\vdots \\
\rho(s_{\al(1)}\cdots s_{\al(k-1)}s_{\al(k)})[B_\eta(p_{k+1})]
&\subset B_{\eta/\la^{k}}(p_1).
\end{align*}
Lastly, the map $\rho(s_{\al(0)})$ is $L$-Lipschitz on $B_{\eta/\la^{k}}(p_1)\subset N_\de(\La)$ by Definition~\ref{def:expansion}(i). Thus we see that
\begin{align}\label{eqn:shrink}
\rho(s_{\al(0)}\cdots s_{\al(k-1)}s_{\al(k)})[B_\eta(p_{k+1})]
\subset B_{L\eta/\la^{k}}(p_0).
\end{align}

Since $\rho(c^\al_k)(p_{k+1})=p_0=x$, the inclusion \eqref{eqn:shrink} can be written as
\begin{align*}
x\in\rho(c^\al_k)[B_\eta(p_{k+1})]
\subset B_{L\eta/\la^{k}}(x)
\end{align*}
for all $k\in\N_0$. Moreover, from \eqref{eqn:contract} we have
\begin{align}\label{eqn:nesting}
\begin{aligned}
\rho(c^\al_k)[B_\eta(p_{k+1})]
&=\rho(s_{\al(0)}\cdots s_{\al(k-1)} s_{\al(k)})[B_\eta(p_{k+1})] \\
&\subset \rho(s_{\al(0)}\cdots s_{\al(k-1)})[B_\eta(p_k)]\\
&= \rho(c^\al_{k-1})[B_\eta(p_k)]
\end{aligned}
\end{align}
for all $k\in\N$. From these follows the rest of the lemma.
\end{proof}

A number of corollaries will follow.

\medskip
We say an action $\rho:\Ga\to\mathrm{Homeo}(M)$ with an invariant subset $\La$ has an \emph{expansivity constant $c>0$} if for every distinct pair of points $x,y\in\La$ there exists an element $g\in\Ga$ such that $d(\rho(g)(x),\rho(g)(y))\ge c$. Compare \cite{CP}*{Proposition 2.2.4}.

\begin{corollary}\label{cor:expansivity}
If an action $\rho:\Ga\to\Homeo(M)$ is expanding at $\La$ with a datum $\D=(\de)$, then $\de$ is an expansivity constant of this action.
\end{corollary}

\begin{proof}
Let $x,y\in\La$ be distinct and consider a $\de$-code $(\al,p)$ for $x$. By Lemma~\ref{lem:nest}, the sequence of neighborhoods of $x$
\begin{align*}
\rho(c^\al_k)[B_\de(p_{k+1})]\quad (k\in\N_0) 
\end{align*}
is nested and exponentially shrinking. Thus there exists an $n\in\N_0$ such that
\begin{align*}
y&\notin\rho(c^\al_n)[B_\de(p_{n+1})], \textup{ that is, } \rho(c^\al_n)^{-1}(y) \notin B_\de(p_{n+1})=B_\de(\rho(c^\al_n)^{-1}(x)).
\end{align*}
Therefore, $d(\rho(c^\al_n)^{-1}(x),\rho(c^\al_n)^{-1}(y))\ge\de$ as desired.
\end{proof}

The following corollary will be crucial for Definition~\ref{def:pi}, which in turn plays an essential role when we discuss actions of hyperbolic groups in Section~\ref{sec:wordhyp}. The corollary is also the reason why we need the assumption that no point of $\La$ is isolated in $M$.

\begin{corollary}\label{cor:qg}
Let $\rho:\Ga\to\Homeo(M)$ be expanding at $\La$ with a datum $\D=(\vS,\de,L,\la)$. Let $\eta\in(0,\de]$ and $\al$ be an $\eta$-code for $x \in \La$. Then the ray $c^\al\in\Ray_x(\D,\eta)$ is a $(\frac{\log L}{\log \la},0)$-quasigeodesic ray in $(\Ga,d_\vS)$.
\end{corollary}

\begin{proof}
Let $(\al,p)$ be an $\eta$-code for $x\in\La$. Since no point of $\La$ is isolated in $M$, for each $k\in\N_0$ we can choose a point $y_k$ such that
\[
x\neq y_k\in \rho(c^\al_k)[B_{\eta}(p_{k+1})]\cap N_{\de/L^{k}}(\La).
\] 
Let $k,j\in\N_0$ be such that $0\le j<k$ and set $r_{kj}:=(c^\al_k)^{-1}c^\al_j$. By (the last statement of) Lemma~\ref{lem:nest}, the map $\rho(r_{kj})$ is $\la^{k-j}$-expanding on $\rho(c^\al_j)^{-1}\big[\rho(c^\al_k)[B_{\eta}(p_{k+1})]\big]\supset\{\rho(c^\al_j)^{-1}(y_k),\rho(c^\al_j)^{-1}(x)\}$, hence
\[
d(\rho(r_{kj})(y_{kj}),\rho(r_{kj})(x_{j}))
\ge\la^{k-j}d(y_{kj},x_{j}),
\]
where we set $y_{kj}:=\rho(c^\al_j)^{-1}(y_k)$ and $x_{j}:=\rho(c^\al_j)^{-1}(x)$. On the other hand, we have
\[
L^{|r_{kj}|_\vS}d(y_{kj},x_{j})\ge d(\rho(r_{kj})(y_{kj}),\rho(r_{kj})(x_{j}))
\]
from \eqref{eqn:lip}, since $x_{j}\in\La$ while $y_{kj}\in N_{\de/L^{k-j-1}}(\La)$ and $|r_{kj}|_\vS\le k-j$. From these two inequalities we obtain
\[
d_\vS(c^\al_k,c^\al_j)=|r_{kj}|_\vS\ge\frac{\log\la}{\log L}\cdot(k-j).
\]
Therefore, the $\eta$-ray $c^\al$ is a $(\frac{\log L}{\log \la},0)$-quasigeodesic ray.
\end{proof}

Another consequence of the encoding concerns the dynamics of the action of $\Ga$ on $\La$. The action of $\Ga$ on $\La$ need not be minimal in general even if $\Ga$ is a non-elementary hyperbolic group (see Example~\ref{ex:E1}). Nevertheless, the action of $\Ga$ on $\La$ has no \emph{wandering points}:

\begin{theorem}\label{thm:minimal}
Let $\rho:\Ga\to\Homeo(M)$ be expanding at $\La$ with a datum $\D=(\de)$. For $x\in\La$ and $\eta\in(0,\de]$, consider a ray $c^\al\in\Ray_x(\D,\eta)$ associated to an $\eta$-code $(\al,p)$ for $x$. Then
\begin{enumerate}[label=\textup{(\arabic*)},nosep,leftmargin=*]
\item
there exist a subsequence $(g_j)$ of $(c^\al_k)$ and a point $y\in\La$ such that $(\rho(g_j))$ converges to $x$ uniformly on $B_{\eta/2}(y)$;
\item
there exists an infinite sequence $(h_j)$ in $\Ga$ such that
\[
\lim_{j\to \infty} \rho(h_j)(x)=x.
\]
\end{enumerate}
\end{theorem}

\begin{proof}
(1) By the compactness of $\La$, the sequence $p=(p_k)$ contains a subsequence $(p_{k_j+1})$ converging to some point $y\in\La$. The point $y$ is then covered by infinitely many balls $B_{\eta/2}(p_{k_j+1})$, $j\in\N$. By Lemma~\ref{lem:nest} the corresponding elements $g_j:= c^\al_{k_j}\in\Ga$ will send $B_{\eta/2}(y)\subset B_{\eta}(p_{k_j+1})$ to a subset $\rho(g_j)[B_{\eta/2}(y)]$ of diameter at most $L\eta/\la^{k_j}$ containing $x$. From this we conclude that the sequence $(\rho(g_j))$ converges to $x$ uniformly on $B_{\eta/2}(y)$. 

(2) On the other hand, since, in particular, $p_{k_1+1}\in B_{\eta/2}(y)$, we obtain
\[
\lim_{j\to\infty} \rho( c^\al_{k_j} ) (p_{k_1+1})= x. 
\]
Since $\rho(c^\al_{k_1})^{-1}(x)=p_{k_1+1}$, it follows that for $h_j:= c^\al_{k_j} (c^\al_{k_1})^{-1}$ we have
\[
\lim_{j\to\infty} \rho(h_j)(x)=x.\qedhere
\]
\end{proof}

\begin{remark}
The idea of \emph{Markov coding} of limit points of actions of finitely generated groups $\Ga$ by sequences in $\Ga$ is rather standard in \emph{symbolic dynamics} and goes back to Nielsen, Hedlund and Morse; we refer the reader to the paper by Series \cite{Series} for references and historical discussion. In the setting of hyperbolic groups this was introduced in Gromov's paper \cite{Gromov}*{\S.8} and discussed in more detail in the book by Coornaert and Papadopoulos \cite{CP}. Section 8.5.Y of Gromov's paper discusses a relation to Sullivan's stability theorem.
\end{remark}

\subsection{The meandering hyperbolicity condition}\label{sec:hyperbolicity}

We continue the discussion from the previous section. In order to define the meandering hyperbolicity condition we need to introduce an equivalence relation $\sim^N_{(\D,\eta)}$ on the set $\Ray_x(\D,\eta)$ of rays in $\Ga$ associated to $(\D,\eta)$-codes for $x\in \La$; recall Definitions~\ref{def:code} and \ref{def:ray}.

\begin{definition}[$(\D,\eta;N)$-equivalence]\label{def:N-equiv}
Suppose $\rho:\Ga\to\Homeo(M)$ is expanding at $\La$ with a datum $\D=(\vS,\de)$ and let $0<\eta\le\de$.
\begin{enumerate}[label=(\alph*),nosep,leftmargin=*]
\item
For each integer $N\ge0$ we define the relation $\approx^N_{(\D,\eta)}$ on $\Ray_x(\D,\eta)$ by declaring that $c^\al \approx^N_{(\D,\eta)} c^\be$ if there exist infinite subsets $P,\,Q\subset \N_0$ such that the subsets
\[
c^\al(P)=\{c^\al_k \mid k\in P\}\quad\textup{and}\quad c^\be(Q)=\{c^\be_j \mid j\in Q\}
\]
of $(\Ga,d_\vS)$ are within Hausdorff distance $N$ from each other. The  \emph{$(\D,\eta;N)$-equivalence}, denoted  $\sim^N_{(\D,\eta)}$, is the equivalence relation on the set $\Ray_x(\D,\eta)$ generated by the relation $\approx^N_{(\D,\eta)}$. In other words, we write $c^\al\sim^N_{(\D,\eta)} c^\be$ and say $c^\al$ and $c^\be$ are $(\D,\eta;N)$-equivalent if there is a finite chain of ``interpolating" rays $c^\al=c^{\ga_1}, c^{\ga_2},\ldots, c^{\ga_n}=c^\be$ in $\Ray_x(\D,\eta)$ such that
\[
c^{\ga_1} \approx^N_{(\D,\eta)} c^{\ga_2} \approx^N_{(\D,\eta)} \cdots \approx^N_{(\D,\eta)} c^{\ga_n}. 
\]
\item
The $(\D,\eta)$-rays $c^\al$ and $c^\be$ are said to \emph{$(\D,\eta;N)$-fellow-travel} if their images $c^\al(\N_0)$ and $c^\be(\N_0)$ are within Hausdorff distance $N$ from each other.
\end{enumerate}
\end{definition} 

\begin{remark}\phantomsection\label{rem:N-equiv}
\begin{enumerate}[label=(\alph*),nosep,leftmargin=*]
\item
Observe that $(\D,\eta;N)$-fellow-traveling rays are $(\D,\eta;N)$-equivalent.
\item
The notion of $(\D,\eta;N)$-equivalence is much more complex than that of $(\D,\eta;N)$-fellow-traveling for a reason: while fellow-traveling is a natural condition in the context of hyperbolic groups where quasigeodesics with common end-points uniformly fellow-travel, our notion of equivalence is defined this way in order to incorporate non-hyperbolic groups, such as higher rank uniform lattices, into the theory of hyperbolic group actions. In all examples that we have, two rays are  $(\D,\eta;N)$-equivalent if and only if they admit an interpolation by at most one extra ray. However, in general, we see no reason for interpolation via one ray to define an equivalence relation on coding rays. A picture describing the behavior of an interpolating ray can be found in Figure~\ref{fig:int-ray} in Section~\ref{sec:ulattice3}: the red interpolating ray ``meanders'' between the two black rays. This explains the terminology \emph{meandering hyperbolicity}.
\item
If $(\vS_0,\de_0)=\D_0\prec\D=(\vS,\de)$ and $\eta\le\de<\de_0$, 
then we have the inclusion $\Ray_x(\D_0,\eta)\subset\Ray_x(\D,\eta)$ as in Remark~\ref{rem:ray}(d) and can talk about the $(\D,\eta;N)$-equivalence of $(\D_0,\eta)$-rays.
\end{enumerate}
\end{remark}

\medskip
We are now ready to define the meandering hyperbolicity condition. First, we recall the notion of S-hyperbolicity which Sullivan defined in \cite{Sul}.

\begin{definition}[S-hyperbolicity]\label{def:shyp}
Let $\rho:\Ga\to\Homeo(M)$ be expanding at $\La$ with a datum $\D=(\I,\U,\vS,\allowbreak\de,L,\la)$.
\begin{enumerate}[label=(\alph*),nosep,leftmargin=*]
\item
The action $\rho$ is said to be \emph{S-hyperbolic (at $\La$)} if there exists an integer $N\ge1$ such that, for every $x\in\La$, all rays in $\Ray_x(\D,\de)$ are $(\D,\de;N)$-fellow-traveling.
\item
The action $\rho$ is said to be \emph{uniformly S-hyperbolic (at $\La$)} if there is an integer $N \ge1$ such that, for every $\eta\in(0,\de]$ and every $x\in\La$, all rays in $\Ray_x(\D,\eta)$ are $(\D,\eta;N)$-fellow-traveling. 
\end{enumerate}
We shall refer to the pair $(\D;N)$ as a \emph{(uniform) S-hyperbolicity datum} of $\rho$. 
\end{definition}

By weakening the $(\D,\eta;N)$-fellow-traveling condition in the definition of S-hyperbolicity to the $(\D,\eta;N)$-equivalence, as well as considering a refined expansion datum (Definition~\ref{def:refine}), we define the notion of \emph{meandering hyperbolicity} as follows. 

\begin{definition}[Meandering hyperbolicity]\label{def:mhyp}
Let $\rho:\Ga\to\Homeo(M)$ be expanding at $\La$ with the expansion datum $\D_0=(\I_0,\U_0,\vS_0,\allowbreak\de_0,L_0,\la_0)$. 
\begin{enumerate}[label=(\alph*),nosep,leftmargin=*]
\item
The action $\rho$ is said to be \emph{meandering-hyperbolic (at $\La$)} if there exist a refinement $\D=(\I,\U,\vS,\allowbreak\de,L,\la)\succ\D_0$ and an integer $N\ge1$ such that, for every $x\in\La$, all rays in  $\Ray_x(\D_0,\de_0)$ are $(\D,\de;N)$-equivalent (as rays in $\Ray_x(\D,\de)$). 
\item
The action $\rho$ is said to be \emph{uniformly meandering-hyperbolic (at $\La$)} if there are a refinement $\D=(\I,\U,\vS,\allowbreak\de,L,\la)\succ\D_0$ and an integer $N\ge1$ such that for every $\eta\in(0,\de]$ and every $x\in \La$, all rays in $\Ray_x(\D_0,\eta)$ are $(\D,\eta;N)$-equivalent (as rays in $\Ray_x(\D,\eta)$). 
\end{enumerate}
We shall refer to the triple  $(\D_0\prec\D;N)$ as a \emph{(uniform) meandering hyperbolicity datum} of $\rho$.
\end{definition}

\begin{remark}\phantomsection\label{rem:m-hyp}
\begin{enumerate}[label=(\alph*),nosep,leftmargin=*]
\item
Obviously, uniform meandering hyperbolicity implies meandering hyperbolicity.
\item
If an action is S-hyperbolic with a datum $(\D_0;N)$ then, for any trivial refinement $\D$ of $\D_0$ (see Definition~\ref{def:t-refine}), it is meandering-hyperbolic with the datum $(\D_0\prec\D;N)$, since $(\D_0,\eta;N)$-fellow-traveling rays are $(\D,\eta;N)$-equivalent. 
\item
A prime example of a meandering-hyperbolic action which is not S-hyperbolic is given by uniform lattices in higher rank semisimple Lie groups (see Section~\ref{sec:ulattice}). Every other examples of meandering-hyperbolic actions we present in this paper are actually S-hyperbolic (see Sections~\ref{sec:toy}, \ref{sec:wordhyp} and \ref{sec:examples}).
\item
We exhibit expanding actions which fail to be meandering-hyperbolic in Section~\ref{sec:e/=>h}.
\end{enumerate}
\end{remark}

\section{The structural stability theorem}\label{sec:proof}

In Section~\ref{sec:hyp->stability} we present the statement of our structural stability theorem for meandering-hyperbolic actions. Then we devote the rest of the section to its proof.

\subsection{Statement of the theorem: meandering hyperbolicity implies stability}\label{sec:hyp->stability}

Suppose $\rho:\Ga\to\Homeo(M)$ is expanding at $\La$. Following Sullivan's remark (at the very end of his paper \cite{Sul}) we would like to talk about small perturbations of $\rho$ which are still expanding. For this, we only need information on a neighborhood of the compact subset $\La$.

We equip $\Homeo(M)$ with what we call the \emph{compact-open Lipschitz topology}. Given a compact subset $K\subset M$, a positive number $\ep>0$ and $f\in\Homeo(M)$, let $U(f;K,\ep)$ denote the set of all $g\in\Homeo(M)$ that are \emph{$\ep$-close to $f$ on $K$} with respect to the metric $d_{\Lip,K}$ defined by
\begin{align}\label{eqn:dLip}
d_{\Lip,K}(f,g):=
\sup_{x\in K} d(f(x),g(x))+
\sup_{\substack{x,y\in K \\ x\neq y}} \left|\frac{d(f(x),f(y))}{d(x,y)}-\frac{d(g(x),g(y))}{d(x,y)}\right|
<\ep.
\end{align}
We define the compact-open Lipschitz topology on $\Homeo(M)$ as the topology generated by the collection of all such $U(f;K,\ep)$. This topology enables us to control the image as well as the Lipschitz constant on a neighborhood of the compact subset $\La$.

\begin{remark}\label{rem:diff}
As we mentioned in the introduction, Sullivan \cite{Sul}*{\S 9} actually considers the case when $M$ is a Riemannian manifold with the Riemannian distance function $d$ and the actions are by $C^1$-diffeomorphisms. In this case the compact-open Lipschitz topology on $\Diff^1(M)$ is weaker than the compact-open $C^1$-topology.
\end{remark}

Suppose $\rho:\Ga\to\Homeo(M)$ is expanding at $\La$ with a datum $\D=(\I,\vS)$  (Definition~\ref{def:expansion}). Given a compact subset $K\subset M$ and $\ep>0$, we say that a homomorphism $\rho':\Ga\to\Homeo(M)$ is a \emph{$(K,\ep)$-perturbation of $\rho$} if $\rho'(s_i)$ is $\ep$-close to $\rho(s_i)$ on $K$ for all $i\in\I$. The set of all $(K,\ep)$-perturbations of $\rho$ will be denoted by
\begin{align}\label{eqn:pert}
U(\rho;K,\ep).
\end{align}
Accordingly, we topologize $\mathrm{Hom}(\Ga,\Homeo(M))$ via the topology of ``algebraic convergence" by identifying it, via the map $\rho\mapsto (\rho(s_i))_{i\in\I}$, with a subset of $[\Homeo(M)]^\I$ equipped with the subspace topology. Then the subset $U(\rho;K,\ep)$ of $\mathrm{Hom}(\Ga,\Homeo(M))$ is an open neighborhood of $\rho$. Note that, when the ambient space $M$ itself is compact, we can set $K=M$ and simply talk about $\ep$-perturbations.

Now we are able to state our structural stability theorem for meandering-hyperbolic actions.

\begin{theorem}\label{thm:main}
If an action $\rho:\Ga\to\Homeo(M)$ is meandering-hyperbolic at a compact invariant subset $\La\subset M$ with a datum $(\D_0\prec\D;N)$, then the following hold.
\begin{enumerate}[label=\textup{(\arabic*)},leftmargin=*,parsep=0ex,itemsep=0.5ex]
\item
{
\setlength{\abovedisplayskip}{0.5ex}
\setlength{\belowdisplayskip}{0.5ex}
\setlength{\abovedisplayshortskip}{0ex}
\setlength{\belowdisplayshortskip}{0ex}
The action $\rho$ is structurally stable in the sense of Lipschitz dynamics. In other words, there exist a compact set $K\supset\La$ and a constant $\ep=\ep(\D,N)>0$ such that for every
\[
\rho'\in U(\rho;K,\ep)
\]
there exist a $\rho'$-invariant compact subset $\La'\subset M$ and an equivariant homeomorphism
\[
\phi:\La\to\La',
\]
that is, $\rho'(g)\circ\phi=\phi\circ\rho(g)$ on $\La$ for all $g\in\Ga$.
}
\item
The map $U(\rho;K,\ep)\to C^0(\La, M)$, $\rho'\mapsto\phi$ is continuous at $\rho$.
\item
Every action $\rho'\in U(\rho;K,\ep)$ is expanding at $\La'$.
\end{enumerate}
If an action $\rho:\Ga\to\Homeo(M)$ is uniformly meandering-hyperbolic then, in addition to the preceding statements, the following is true as well.
\begin{enumerate}[label=\textup{(\arabic*)},start=4,leftmargin=*,parsep=0pt,itemsep=0pt]
\item Every action $\rho'\in U(\rho;K,\ep)$ is again uniformly meandering-hyperbolic.
\end{enumerate}
\end{theorem}

As an immediate consequence of Theorem~\ref{thm:main}(1), we have:

\begin{corollary}\label{cor:alg-stab}
A meandering-hyperbolic action $\rho:\Ga\to\Homeo(M)$ at a compact invariant subset $\La\subset M$ is algebraically stable in the following sense: for every $\rho'\in U(\rho;K,\ep)$, the kernel of the $\rho'$-action on $\La'$ equals the kernel of the $\rho$-action on $\La$.
\end{corollary}

\noindent We note, however, that faithfulness of the $\rho$-action on $M$ does not imply faithfulness of nearby actions: see Example~\ref{ex:H^4}.
 
In view of Remarks~\ref{rem:expansion}(c), \ref{rem:m-hyp}(b) and \ref{rem:diff}, we obtain Sullivan's structural stability theorem \cite{Sul}*{Theorem~II} in the $C^1$-setting as a corollary of Theorem~\ref{thm:main}(1):

\begin{corollary}
Consider a group action $\Ga\to\Diff^1(M)$ on a Riemannian manifold $M$ with a compact invariant subset $\La\subset M$. If the action is S-hyperbolic, then it is structurally stable in the sense of $C^1$-dynamics.
\end{corollary}

\begin{remark}\label{rem:sul}
Sullivan's structural stability theorem stated above is not to be confused with the main theorem of \cite{Sul}, which is widely known in the theory of Kleinian groups. In order to guide the reader we summarize the contents of Sullivan's paper \cite{Sul} as follows.

Let $\Ga<\PSL(2,\C)$ be a finitely generated, non-solvable, non-rigid, non-relatively-compact and torsion-free group of conformal transformations of the Riemann sphere $\p^1(\C)$. Sullivan showed that the following are equivalent:
\begin{enumerate}[label=(\arabic*),parsep=0pt,itemsep=0pt]
\item the subgroup $\Ga<\PSL(2,\C)$ is convex-cocompact;
\item the $\Ga$-action on the limit set $\La\subset\p^1(\C)$ satisfies the expansion-hyperbolicity axioms;
\item this action is structurally stable in the sense of $C^1$-dynamics;
\item the subgroup $\Ga<\PSL(2,\C)$ is algebraically stable.
\end{enumerate}
Here $\Ga<\PSL(2,\C)$ is said to be \emph{algebraically stable} if all representations $\Ga\to\PSL(2,\C)$ sufficiently close to the identity embedding are injective.

The implication (4~$\Rightarrow$~1) is the main result (Theorem~A) of the paper \cite{Sul}, in the proof of which his quasiconformal stability theorem (Theorem~C) is obtained. The implication (1~$\Rightarrow$~2) is his Theorem~I, (2~$\Rightarrow$~3) is a special case of Theorem~II, and (3~$\Rightarrow$~4) is immediate. For groups with torsion the implication (4~$\Rightarrow$~1) is false (see Example~\ref{ex:VD}) but the other implications (1~$\Rightarrow$~2~$\Rightarrow$~3~$\Rightarrow$~4) still hold.
\end{remark}

\subsection{Specifying small perturbations \texorpdfstring{$\rho'$}{rho'}}\label{sec:specify}

We now proceed to prove Theorem~\ref{thm:main}. The assertion (1) will be proved in Sections~\ref{sec:specify}-\ref{sec:injective}, and the assertions (2)-(4) in Sections~\ref{sec:cont'}-\ref{sec:uhyp}, respectively.

\medskip
Let $\rho:\Ga\to\Homeo(M)$ be a meandering-hyperbolic action (Definition~\ref{def:mhyp}) at a compact invariant subset $\La\subset M$ with a datum $(\D_0\prec\D;N)$, where
\begin{align*}
(\I_0,\U_0,\vS_0,\de_0,L_0,\la_0)=\D_0\prec\D=(\I,\U,\vS,\de,L,\la),\\
\textup{that is, }
\I_0\subset\I,\; \U_0\prec\U,\; \vS_0\prec\vS,\; \de_0>\de,\; L_0\le L,\; \la_0\ge\la.
\end{align*} 
Recall, in particular, that $\rho$ is expanding with respect to the refined datum $\D$. 

For convenience, we summarize all the important constants for the proof as follows:
\vspace{2mm}

\noindent \framebox[16.6cm][c]{$K :=\ov{N}_{\de}(\La)$, \ $\kappa:= {\min\left\{\bd,\ 1\right\}}$, \ $\displaystyle \ep := \frac{\la-1}{2} \kappa$,  \ 
$L':=L+\ep$, \ $\la':=\la-\ep$, \ $L \geq \la> 1$ }
\vspace{2mm}

The compact set $K\supset\La$ and the constant $\ep=\ep(\D,N)>0$ specify the open set $U(\rho;K,\ep)$ of all $(K,\ep)$-perturbations $\rho'$ of $\rho$. 
Note that $K$ is compact since $(M,d)$ is assumed to be proper. Since $\de\le\de_\U$ is a Lebesgue number of $\U$, 
it follows that $\U$ is a cover of $N_\de(\La)$.

Now, we suppose
\[
\rho'\in U(\rho;K,\ep).
\]
Then, by the definitions \eqref{eqn:dLip} and \eqref{eqn:pert}, we have that for all $s\in\vS$,
\begin{align}\label{eqn:pert0}
d_{\Lip,K}(\rho(s),\rho'(s))<\ep.
\end{align}

We first observe that for each $i\in\I$,
\begin{align}\label{eqn:la'}
\rho'(s_i^{-1})\textup{ is }\la'\textup{-expanding on }U_i\cap K
\end{align}
Indeed, since $\rho(s_i^{-1})$ is $\la$-expanding on $U_i$, we see from \eqref{eqn:pert0} that for all distinct $x,y\in U_i\cap K$,
\begin{align*}
\frac{d(\rho'(s_i^{-1})(x),\rho'(s_i^{-1})(y))}{d(x,y)}
>\frac{d(\rho(s_i^{-1})(x),\rho(s_i^{-1})(y))}{d(x,y)}-\ep
\ge \la-\ep
=\la'.
\end{align*}

Moreover, we also note that
\begin{align}\label{eqn:L'}
\rho'(s_i)\textup{ is }L'\textup{-Lipschitz on }N_\de(\La)
\end{align}
for every $i\in\I$. To see this recall that the maps $\rho(s_i)$ are $L$-Lipschitz on $N_\de(\La)$ by Definition~\ref{def:expansion}(i), thus by \eqref{eqn:pert0} again
\begin{align*}
\frac{d(\rho'(s_i)(x),\rho'(s_i)(y))}{d(x,y)}
<\frac{d(\rho(s_i)(x),\rho(s_i)(y))}{d(x,y)}+\ep
\le L+\ep=L'
\end{align*}
for all distinct $x,y\in N_\de(\La)\subset K$. Note that $L'=L+\ep\ge\la+\ep>\la-\ep=\la'$.

For later use, we prove the following lemma.

\begin{lemma}\label{lem:dbw}
For every $w\in \Ga$ such that $|w|_\vS=n\le N$, we have
\[
d(\rho(w)(y),\rho'(w)(y))<nL^{n-1}\ep
\]
for all $y\in N_\kappa (\La)$.
\end{lemma}

\begin{proof}
We prove this by induction on $n\leq N$. If $|w|_\vS=1$, the claim is true by \eqref{eqn:pert0} since $N_\kappa(\La)\subset N_{\de}(\La)$. Suppose $w=st$ with $s\in\vS$, $|t|_\vS=n-1$ and $|w|_\vS=n$. Since $\rho(t)$ is $L^{n-1}$-Lipschitz on $N_{\de/L^{n-2}}(\La)\supset N_\kappa(\La)$ by \eqref{eqn:lip}, we have $\rho(t)(y)\in N_{L^{n-1}\kappa}(\La)\subset N_\de(\La)$. By the induction hypothesis $d(\rho(t)(y),\rho'(t)(y))<(n-1)L^{n-2}\ep$, we then have $\rho'(t)(y)\in N_{\de}(\La)$ as well, since $\ep=\frac{\la-1}{2}\kappa<L\kappa$ and thus
\[
L^{n-1}\kappa+(n-1)L^{n-2}\ep
<L^{n-1}\kappa+(n-1)L^{n-2}L\kappa
=nL^{n-1}\kappa<\de.
\]
As $\rho(s)$ is $L$-Lipschitz on $N_\de(\La)$, we obtain
\[
d\big(\rho(s)[\rho(t)(y)],\rho(s)[\rho'(t)(y)]\big) < L\cdot(n-1)L^{n-2}\ep.
\]
Furthermore, we have from \eqref{eqn:pert0}
\[
d\big(\rho(s)[\rho'(t)(y)],\rho'(s)[\rho'(t)(y)]\big)<\ep
\]
since $\rho'(t)(y)\in N_\de(\La)$. Thus
\begin{align*}
d(\rho(w)(y),\rho'(w)(y))
&\le d\big(\rho(s)[\rho(t)(y)],\rho(s)[\rho'(t)(y)]\big)
+d\big(\rho(s)[\rho'(t)(y)],\rho'(s)[\rho'(t)(y)]\big)\\
&< (n-1)L^{n-1}\ep + \ep\\
&< (n-1)L^{n-1}\ep + L^{n-1}\ep = nL^{n-1}\ep
\end{align*}
and the claim is proved.
\end{proof}

\subsection{Definition of \texorpdfstring{$\phi$}{phi}}

We first construct a map $\phi:\La\to M$.

Let $x\in\La$. In order to define $\phi(x)$, choose a  $(\D_0,\de_0)$-code $(\al,p)$ for $x$ as in Section~\ref{sec:encoding}. Since $\de<\de_0$, we may regard $\al$ as a $(\D,\de)$-code as in Remark~\ref{rem:code}(c)(d). Then, as we know from Lemma~\ref{lem:nest}, the point $x$ has an exponentially shrinking nested sequence of neighborhoods $\rho(c^\al_k)[B_\de(p_{k+1})]$ ($k\in\N_0$) such that 
\[
\{x\}=\bigcap_{k=0}^\infty\rho(c^\al_k)[B_\de(p_{k+1})].
\]

Now, consider a perturbation $\rho'\in U(\rho;K,\ep)$ of $\rho$ as specified in Section~\ref{sec:specify}. We claim that the sequence of perturbed subsets $\rho'(c^\al_k)[B_\de(p_{k})]$ ($k\in\N_0$) is also nested and exponentially shrinking. Since $M$ is complete, the intersection of this collection of subsets is a singleton in $M$ and we can \emph{define} $\phi_\al(x)$ by the formula
\begin{align*}
\{\phi_\al(x)\}
=\bigcap_{k=0}^\infty\rho'(c^\al_k)[B_\de(p_{k+1})]. 
\end{align*}

The claim will be proved in the following lemma, where, in fact, we shall find the values of $\eta\le\de$ for which the sequence $\rho'(c^\al_k)[B_\eta(p_{k})]$ is nested and exponentially shrinking. This lemma will later be used often, for example, when we show that $\phi$ is well-defined (Section~\ref{sec:well}) and is continuous (Section~\ref{sec:cont}).

\begin{lemma}\label{lem:nest'}
Let $\rho:\Ga\to\Homeo(M)$ be a meandering-hyperbolic action with the datum $(\D_0\prec\D;N)$. Consider a number $0<t\le\kappa=\min\{\bd,1\}$ (so that $\frac{\la-1}{2}t\le\ep$) and let
\[
\rho'\in U(\rho;K,\textstyle{\frac{\la-1}{2}t})\subset U(\rho;K,\ep).
\]
If $(\al,p)$ is a $(\D,\de)$-code for $x\in\La$, then for every $\eta\in[t,\de]$ the sequence of subsets
\[
\rho'(c^\al_k)[B_\eta(p_{k+1})]\quad(k\in\N_0)
\]
is nested and exponentially shrinking. Consequently, we have
\begin{align*}
\{\phi_\al(x)\}
=\bigcap_{k=0}^\infty\rho'(c^\al_k)[B_\eta(p_{k+1})]
=\bigcap_{k=0}^\infty\rho'(c^\al_k)[B_\de(p_{k+1})]
\end{align*}
and thus, for every $k\in\N_0$,
\[
\rho'(c^\al_k)^{-1}(\phi_\al(x))
\in B_{\eta}(p_{k+1}).
\]
If the $\de$-code $(\al,p)$ is special, then
\[
d(x,\phi_\al(x))<t.
\]
\end{lemma}

\begin{proof}
Since $t\le\eta$, we see from the assumption $\rho'\in U(\rho;K,\frac{\la-1}{2}t)$ that
\begin{align}\label{eqn:perturb}
\sup_{x\in K} d(\rho(s_i^{-1})(x),\rho'(s_i^{-1})(x))
\le d_{\Lip,K}(\rho(s_i^{-1}),\rho'(s_i^{-1}))
< \frac{\la-1}{2}t \le \frac{1}{2}(\la\eta-\eta)
\end{align}
for all $i\in\I$. Since $\eta\le\de$, we have by Definition~\ref{def:expansion}(ii) the inclusion \eqref{eqn:expand} for $\rho$
\begin{align}\label{eqn:inclusions}
B_{\eta}(p_{k+1})
\subset B_{\la\eta}(p_{k+1})
\subset \rho(s_{\al(k)}^{-1})[B_\eta(p_k)]
\subset \rho(s_{\al(k)}^{-1})[U_{\al(k)}]
\end{align}
for all $k\in\N$, while $B_\eta(p_k)\subset B_\de(p_k)\subset U_{\al(k)}\cap K$ as $\al$ is a $\de$-code. 

In particular, we see from \eqref{eqn:inclusions} that the boundary of $\rho(s_{\al(k)}^{-1})[B_\eta(p_k)]$ lies outside $B_{\la\eta}(p_{k+1})$, and from \eqref{eqn:perturb} that the boundary of $\rho'(s_{\al(k)}^{-1})[B_\eta(p_k)]$ lies outside the ball centered at $p_{k+1}$ of radius $\la\eta-\frac{1}{2}(\la\eta-\eta)=\frac{1}{2}\eta(\la+1)>\eta$. Moreover, $\rho'(s_{\al(k)}^{-1})(p_k)$ lies in the ball centered at $p_{k+1}$ of radius $\frac{1}{2}\eta(\la-1)$, which implies $\rho'(s_{\al(k)}^{-1})[B_\eta(p_k)]\cap B_{\frac{1}{2}\eta(\la+1)}(p_{k+1})\neq\emptyset$. Note that every open ball is path-connected in a geodesic metric space $M$. Thus we deduce that $\rho'(s_{\al(k)}^{-1})[B_\eta(p_k)]$ is a connected component of $M\setminus \partial \rho'(s_{\al(k)}^{-1})[B_\eta(p_k)]$ and hence $B_{\frac{1}{2}\eta(\la+1)}(p_{k+1})\subset \rho'(s_{\al(k)}^{-1})[B_\eta(p_k)]$. Therefore, we conclude that
\begin{align}\label{eqn:expand'}
B_{\eta}(p_{k+1})
\subset B_{\frac{1}{2}\eta(\la+1)}(p_{k+1})
\subset \rho'(s_{\al(k)}^{-1})[B_\eta(p_k)]
\subset \rho'(s_{\al(k)}^{-1})[U_{\al(k)}\cap K]. 
\end{align}
Thus
\begin{align}\label{eqn:contract'}
 \rho'(s_{\al(k)})[B_\eta(p_{k+1})]\subset B_\eta(p_k) 
\end{align}
for all $k\in\N$, and we check as in \eqref{eqn:nesting} the nesting property
\[
\rho'(c^\al_k)[B_\eta(p_{k+1})]\subset\rho'(c^\al_{k-1})[B_\eta(p_k)]
\]
for all $k\in\N$.

Furthermore, the diameter of $ \rho'(c^\al_k)[B_\eta(p_{k+1})]
=\rho'(s_{\al(0)}s_{\al(1)}\cdots s_{\al(k)})[B_\eta(p_{k+1})]$ is at most $2\eta(L+\ep)/(\la')^k$, because we have \eqref{eqn:expand'} and each $\rho'(s_{\al(j)})$ $(1\le j\le k)$ is $(1/\la')$-contracting on $\rho'(s_{\al(j)}^{-1})[U_{\al(j)}\cap K]$ by \eqref{eqn:la'}, and the last map $\rho'(s_{\al(0)})$ is $(L+\ep)$-Lipschitz on $B_\eta(p_1)\subset N_\de(\La)$ by \eqref{eqn:L'}. Hence the exponentially shrinking property also holds.

If the $\de$-code $(\al,p)$ is special, then the inclusion \eqref{eqn:expand'} as well as \eqref{eqn:contract'} hold for $k=0$. Thus
\[
\phi_\al(x)
\in\rho'(c^\al_0)[B_\eta(p_{1})]
=\rho'(s_{\al(0)})[B_\eta(p_{1})]
\subset B_\eta(x)
\]
for all $\eta\in[t,\de]$. It follows that $d(x,\phi_\al(x))<t$.
\end{proof}

\subsection{\texorpdfstring{$\phi$}{phi} is well-defined}\label{sec:well}

To show that $\phi$ is well-defined, we need to show that $\phi_\al(x)=\phi_\be(x)$ for any two $(\D_0,\de_0)$-codes $(\al,p)$ and $(\be,q)$. Since $\rho$ is meandering-hyperbolic with the datum $(\D_0\prec\D;N)$, the corresponding $(\D_0,\de_0)$-rays $c^\al$ and $c^\be$ are $(\D,\de;N)$-equivalent, that is, $c^\al\sim^N_{(\D,\de)} c^\be$ in $\Ray_x(\D,\de)$. By definition, the equivalence relation $\sim^N_{(\D,\de)}$ is generated by the relation $\approx^N_{(\D,\de)}$ (see Definition~\ref{def:N-equiv}(a)). Thus, it suffices to show the equality $\phi_\al(x)=\phi_\be(x)$ when  $c^\al$ and $c^\be$ are $(\D,\de)$-rays and satisfy $c^\al\approx^N_{(\D,\de)} c^\be$, that is, there exist infinite subsets $P,\,Q\subset \N_0$ such that the subsets $c^\al(P)$ and $c^\be(Q)$ are within Hausdorff distance $N$ from each other in $(\Ga,d_\vS)$.

Suppose to the contrary that
\[
\bigcap_{k\in P} \rho'(c^\al_k)[B_\de(p_{k+1})] = \{\phi_\al(x)\}\neq\{\phi_\be(x)\}=\bigcap_{j\in Q} \rho'(c^\be_j)[B_\de(q_{j+1})].
\]
Since the open sets $\rho'(c^\be_j)[B_\de(q_{j+1})]$ shrink to $\phi_\be(x)$, there exists an integer $n\in Q$ such that $\phi_\al(x)\notin\rho'(c^\be_n)[B_\de(q_{n+1})]$, that is,
\begin{align}\label{eqn:well}
\rho'(c^\be_{n})^{-1}(\phi_\al(x))\notin B_\de(q_{n+1}).
\end{align}
Since the Hausdorff distance between $\{c^\al_k\}_{k\in P}$ and $\{c^\be_j\}_{j\in Q}$ is at most $N$, there is an integer $m\in P$ such that $d_\vS(c^\be_n,c^\al_m)\le N$. Set
\[
r=(c^\be_{n})^{-1}c^\al_{m},
\]
so that $|r|_\vS\le N$. Note from Remark~\ref{rem:ray}(b) that $\rho(r)$ maps $p_{m+1}$ to $q_{n+1}$:
\[
\rho(r)(p_{m+1})=\rho(r)[\rho(c^\al_{m})^{-1}(x)]=\rho(c^\be_{n})^{-1}(x)=q_{n+1}.
\] 
See Figure~\ref{fig:02}.

\begin{figure}[ht]
\labellist
\pinlabel {$\phi_\al(x)$} at 26 148
\pinlabel {$\phi_\be(x)$} at 30 66
\pinlabel {$p_{m+1}$} at 261 160
\pinlabel {$B_\kappa(p_{m+1})$} at 290 180
\pinlabel {$q_{n+1}$} at 261 44
\pinlabel {$B_\de(q_{n+1})$} at 304 20
\pinlabel {$\rho(r)=\rho((c^\be_{n})^{-1}c^\al_{m})$} at 292 110
\pinlabel {$\rho'(c^\al_{m})^{-1}$} at 140 190
\pinlabel {$\rho'(c^\be_{n})^{-1}$} at 140 14
\pinlabel {$\rho'(c^\be_{n})^{-1}$} at 110 74
\endlabellist
\centering
\includegraphics[width=0.7\textwidth]{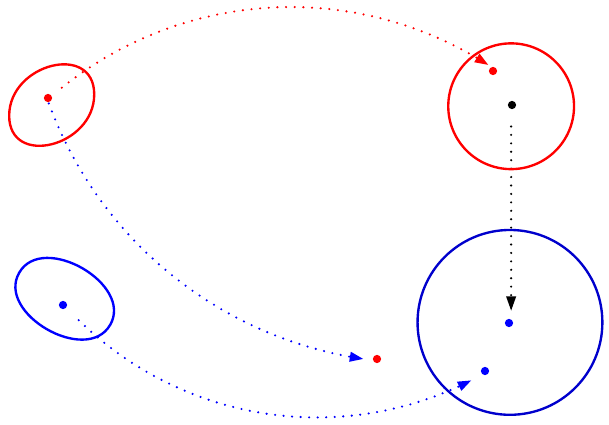}
\caption{The points $\phi_\al(x)$ and $\phi_\be(x)$.}
\label{fig:02}
\end{figure}

In view of Lemma~\ref{lem:nest'} (with $t=\eta=\kappa=\min\{\bd,1\}$), we may assume that we used $\kappa$-balls in the definition of $\phi_\al(x)$, so that
\[
\rho'(c^\al_{m})^{-1}(\phi_\al(x))\in B_{\kappa}(p_{m+1}).
\]
Now, if we show that $\rho'(r)$ maps $B_{\kappa}(p_{m+1})$ into $B_\de(q_{n+1})$ and hence
\begin{align*}
\rho'(c^\be_{n})^{-1}(\phi_\al(x))
&=\rho'(r)[\rho'(c^\al_{m})^{-1}(\phi_\al(x))]\\
&\in\rho'(r)[B_{\kappa}(p_{m+1})]\subset B_\de(q_{n+1}),
\end{align*}
then we are done, since we are in contradiction with \eqref{eqn:well}.

To show $\rho'(r)[B_{\kappa}(p_{m+1})]\subset B_\de(q_{n+1})$, let $l=|r|_\vS\le N$. If $y\in B_{\kappa}(p_{m+1})\subset N_\kappa(\La)$ then
\begin{align*}
d(\rho'(r)(y),q_{n+1})
&= d(\rho'(r)(y),\rho(r)(p_{m+1}))\\
&\le d(\rho'(r)(y),\rho(r)(y)) +d(\rho(r)(y),\rho(r)(p_{m+1}))\\
&< k L^{l-1}\ep +L^l\kappa\\
&< N L^{N-1}L\kappa +L^N\kappa\le \de,
\end{align*}
where the second inequality holds by Lemma~\ref{lem:dbw} and since $\rho(r)$ is $L^l$-Lipschitz on $N_{\de/L^{l-1}}(\La)\supset B_{\kappa}(p_{m+1})$ by \eqref{eqn:lip}, and the third inequality holds since $\ep=\frac{\la-1}{2}\kappa<L\kappa$.

This completes the proof of the equality $\phi_\al(x)=\phi_\be(x)$.

\medskip
From now on we may write $\phi(x)$ for $x\in\La$ without ambiguity. An immediate consequence of this is that we are henceforth free to choose a \emph{special} $\de$-code for $x$. Without loss of generality, we may assume that $N\geq 4$. Then from Lemma~\ref{lem:nest'} (with $t=\kappa=\min\{\bd,1\}$) we conclude that
\begin{align}\label{eqn:bd}
d(x,\phi(x))<\kappa\le\bd<\frac{\de}{5}
\end{align}
for all $x\in\La$.

\subsection{\texorpdfstring{$\phi$}{phi} is equivariant} \label{sec:equiv}

To show the equivariance of $\phi$, it suffices to check it on the generating set $\vS$ of $\Ga$.

Given $x\in\La$ and $s\in\vS$, set $y=\rho(s^{-1})(x)$. Let $(\be,q)$ be a \emph{special} $\de$-code for $y$, so that $B_\de(y)\subset U_{\be(0)}$ (see Definition~\ref{def:code}). Then we consider a $\de$-code $(\al,p)$ for $x$ defined by $s_{\al(0)}=s$ and
\begin{align*}
{\al(k)}&={\be(k-1)},\\
p_k&=q_{k-1}
\end{align*}
for $k\in\N$: indeed, we verify the requirement that
\[
B_\de(p_{k})=B_\de(q_{k-1}) \subset U_{\be(k-1)}=U_{\al(k)}
\]
for $k\in\N$. The associated rays $\{c^\al_k\}_{k\in\N_0}$ and $\{c^\be_k\}_{k\in\N_0}$ (see Definition~\ref{def:ray}) are related by
\[
c^\al_k
=s_{\al(0)}(s_{\al(1)}\cdots s_{\al(k)})
=s(s_{\be(0)}\cdots s_{\be(k-1)})
=s\,c^\be_{k-1}
\]
for $k\in\N$. Therefore, we have
\begin{align*}
\{\phi[\rho(s)(y)]\}=\{\phi(x)\}
&=\bigcap_{k=0}^\infty \rho'(c^\al_k)[B_\de(p_{k+1})] \\
&=\bigcap_{k=1}^\infty \rho'(c^\al_k)[B_\de(p_{k+1})] \\
&=\bigcap_{k=1}^\infty \rho'(s)\rho'(c^\be_{k-1})[B_\de(q_k)] \\
&=\rho'(s)\left[\bigcap_{k=0}^\infty \rho'(c^\be_{k})[B_\de(q_{k+1})] \right]
=\rho'(s)\{\phi(y)\},
\end{align*}
which implies the equivariance of $\phi$.

\subsection{\texorpdfstring{$\phi$}{phi} is continuous}\label{sec:cont}

Let $\zeta>0$ be given. In order to show that $\phi$ is continuous at $x\in\La$, assign a $(\D,\de)$-code $(\al,p)$ for $x$ which comes from a $(\D,\de_0)$-code for $x$. (Recall $\de_0>\de$ and Remark~\ref{rem:code}(c). Also note that Lemma~\ref{lem:nest'} applies to the $(\D,\de)$-code $(\al,p)$ with $\eta=\de$). Choose an integer $j\in\N_0$ such that $2\de_0(L+\ep)/(\la')^{j}<\zeta$, where the constants $\la'$ and $\ep$ are from Section \ref{sec:specify}. Since $\rho(c^\al_j)^{-1}$ maps $x$ to $p_{j+1}$ and is continuous, there exists $\zeta'>0$ such that $\rho(c^\al_j)^{-1}[B_{\zeta'}(x)]\subset B_{\de_0-\de}(p_{j+1})$. Below we will show that if $y\in\La$ satisfies $d(x,y)<\zeta'$ then $d(\phi(x),\phi(y))<\zeta$ thereby proving that $\phi$ is continuous at $x$.

Let $y\in\La$ be such that $d(x,y)<\zeta'$. Then $\rho(c^\al_j)^{-1}(y)\in B_{\de_0-\de}(p_{j+1})$, hence
\begin{align*}
y\in
\rho(c^\al_j)[B_{\de_0-\de}(p_{j+1})]\subset
\rho(c^\al_{j-1})[B_{\de_1-\de}(p_j)]\subset\cdots\subset
\rho(c^\al_0)[B_{\de_0-\de}(p_1)]
\end{align*}
by Lemma~\ref{lem:nest}(i). In other words, for $0\leq k\leq j$, we have
\begin{align*}
\rho(c^\al_k)^{-1}(y) &\in B_{\de_0-\de}(p_{k+1}), \\
\textup{hence}\quad B_\de(\rho(c^\al_k)^{-1}(y)) &\subset B_{\de_0}(p_{k+1}) \subset U_{\al(k+1)}.
\end{align*}
This means that there is a $\de$-code $(\be,q)$ for $y\in\La$ with a property that
\[
\be(k)=\al(k)\quad\textup{for }0\le k\le j.
\]
In particular, $c^\be_j=c^\al_j$ and hence $q_{j+1}=\rho(c^\be_j)^{-1}(y)=\rho(c^\al_j)^{-1}(y)\in B_{\de_0-\de}(p_{j+1})$. Consequently, we have $B_{\de}(q_{j+1})\subset B_{\de_0}(p_{j+1})$ and
\begin{align*}
\{\phi(y)\} 
&=\bigcap_{k=0}^\infty \rho'(c^\be_k)[B_{\de}(q_{k+1})]\\
&\subset \rho'(c^\be_j)[B_{\de}(q_{j+1})]\\
&=\rho'(c^\al_j)[B_{\de}(q_{j+1})]\\
&\subset\rho'(c^\al_j)[B_{\de_0}(p_{j+1})],
\end{align*}
By (the proof of) Lemma~\ref{lem:nest'}, the diameter of the last set $\rho'(c^\al_j)[B_{\de_0}(p_{j+1})]$ is at most $2{\de_0}(L+\ep)/(\la')^{j}<\zeta$. Since $\phi(x)\in\rho'(c^\al_j)[B_{\de_0}(p_{j+1})]$ as well, we showed
\[
d(\phi(x),\phi(y))\le2{\de_0}(L+\ep)/(\la')^{j}<\zeta
\]
as desired.

\subsection{\texorpdfstring{$\phi$}{phi} is injective}\label{sec:injective}

Suppose, to the contrary, that $\phi(x)=\phi(y)$ but $x\neq y$. Since $\phi$ is equivariant, we then have $\phi[\rho(g)(x)]=\phi[\rho(g)(y)]$ for any $g\in\Ga$, hence, by \eqref{eqn:bd},
\[
d(\rho(g)(x),\rho(g)(y))<d(\rho(g)(x),\phi[\rho(g)(x)])+d(\phi[\rho(g)(y)],\rho(g)(y))<\de/5+\de/5=2\de/5
\]
for all $g\in\Ga$. But this contradicts Corollary~\ref{cor:expansivity}, since $\rho$ has an expansivity constant $2\de/5$. Therefore, $\phi$ is injective.

\medskip
So far, we have proved the claim that $\phi$ is an equivariant homeomorphism. This completes the proof of Theorem~\ref{thm:main}(1).

\subsection{\texorpdfstring{$\phi$}{phi} depends continuously on \texorpdfstring{$\rho'$}{rho'}}\label{sec:cont'} 

We show that the map
\begin{align*}
\mathrm{Hom}(\Ga,\Homeo(M))
\supset U(\rho;K,\ep)&\to C^0(\La,(M,d))\\
\rho'&\mapsto \phi 
\end{align*}
is continuous at $\rho$, where we equip $\mathrm{Hom}(\Ga, \Homeo(M))$ with the topology of ``algebraic convergence" as in Section~\ref{sec:hyp->stability} and $C^0(\La,(M,d))$ with the uniform topology. This will imply that if $\rho'$ is close to $\rho$ then $\phi$ is close to the identity map and, as for $\La$ and $\La'=\phi(\La)$, that the map $\rho'\mapsto d_{\mathrm{Haus}}(\La,\La')$ is continuous at $\rho$, where $d_{\mathrm{Haus}}$ stands for the Hausdorff distance.

To prove the claim, suppose a sufficiently small constant $t>0$ is given. Of course, we may assume $t\le\min\{\bd,1\}$. Then we have to find a neighborhood $U(\rho;K',\ep')$ of $\rho$ such that $\sup\{d(x,\phi(x))\mid x\in\La\}<t$ for every $\rho'\in U(\rho;K',\ep')$. So, we let
\[
K'=K\quad\textup{and}\quad\ep'=\frac{\la-1}{2}t.
\]
Suppose $\rho'\in U(\rho;K',\ep')$. Let $x\in\La$ and choose a \emph{special} $\de$-code $(\al,p)$ for $x$. Then by Lemma~\ref{lem:nest'} we have $d(x,\phi(x))=d(x,\phi_\al(x))<t$. Since $x\in\La$ is arbitrary, the proof is complete.

\subsection{\texorpdfstring{$\rho'$}{rho'} is expanding}\label{sec:expand}

We now show how, given an expansion datum $\D$ of the action $\rho$, to get a new expansion datum $\D'=(\I',\U',\vS',\allowbreak\de',L',\la')$ of the action $\rho'$.

First take the same data $\I$ and $\vS$ of $\D$ i.e., set $\I'=\I$ and $\vS'=\vS$, and set $L'=L+\ep$ and $\la'=\la-\ep$ from \eqref{eqn:L'} and \eqref{eqn:la'}. Then every $\rho'(s_i^{-1})$ is $\la'$-expanding on $U_i\cap N_\de(\La)$ and  $L'$-Lipschitz on $N_\de(\La)$.
For a positive number $r<4\de/5$, if we set
\[
\U'=\{U_i':=\intr (U_i\cap N_{\de}(\La) )^r=\intr U_i^{r}\cap N_{\de-r}(\La) \mid i\in \I\}
\]
(recall \eqref{eqn:ur} for the definition of $U^{r}$), then by Remark~\ref{rem:expansion}(e) there is a constant $\de'=\de'(r)>0$ such that every $\rho'(s_i^{-1})$ is $(\la',U_\al';\de')$-expanding. Note that $\de'$ can be chosen as small as we want and thus we may assume that $\de'<4\de/5$.
Let $x\in\La$. Then, since $d(x,\phi(x))<\de/5$ by \eqref{eqn:bd} and $0<r<4\de/5$, it follows that
\[
\ov{B}_{r}(\phi(x))\subset B_{r+\de/5}(x) \subset B_{\de}(x) \subset U_i\cap N_\de(\La)
\]
for some $i\in\I$, and hence $\phi(x)\in U'_i$. This implies that $\U'$ covers $\La'$ and $N_{\de'}(\La')\subset N_\de(\La)$. 
Thus every $\rho'(s_i^{-1})$ is $L'$-Lipschitz on $N_{\de'}(\La')$. If we choose $\de'$ so that $\de'$ is smaller than a Lebesgue number of the open covering $\U'$ of $\La'$, 
the properties (i) and (ii) of Definition~\ref{def:expansion} are verified for $\rho'$ with the datum $\D'=(\I',\U',\vS',\allowbreak\de',L',\la')$. Therefore, $\rho'$ is expanding.

\subsection{\texorpdfstring{$\rho'$}{rho'} is again uniformly meandering-hyperbolic}\label{sec:uhyp}

Assume that the action $\rho:\Ga\to\Homeo(M)$ is uniformly meandering-hyperbolic with a datum $(\D_0\prec\D;N)$ (Definition~\ref{def:mhyp}). Then, for every $\eta\in(0,\de]$ and every $x\in\La$, all rays in $\Ray_x(\D_0,\eta)$ are $(\D,\eta;N)$-equivalent.  In order to show that every $\rho'\in U(\rho;K,\ep)$ is uniformly meandering-hyperbolic, we first let $\D'$ be the expansion data of $\rho'$ obtained from $\D$ via the way described in Section~\ref{sec:expand} by taking $r=\de/5$, so that, in particular,
\[
\U'=\{U_i':=\intr (U_i\cap N_{\de}(\La) )^{\de/5}=\intr U_i^{\de/5}\cap N_{4\de/5}(\La) \mid i \in \I\}.
\]
Then, we will define an expansion data $\D_0'\prec \D'$ of $\rho'$ and show that $(\D_0'\prec\D';N)$ is a uniform meandering hyperbolicity datum of $\rho'$.

Since $\D_0\prec D$, one can easily derive an expansion datum $\D_0'=(\I'_0,\U'_0,\vS'_0,\allowbreak\de'_0,L'_0,\la'_0)$ of $\rho'$ from $\D_0$ as follows. Set $\I_0'=\I_0$, $\vS'_0=\vS_0$, $L'_0=L'$ and $\la'_0=\la'$. Let $\U_0=\{U_{0,i} \mid i\in\I_0\}$. From $\U_0\prec\U$, it immediately follows that $\rho'(s_i^{-1})$ is $\la'$-expanding on $U_{0,i}\cap N_\de(\La)$ and  $L'$-Lipschitz on $N_\de(\La)$ for all $i \in \I_0$. If we set
\[
\U'_0=\{U_{0,i}':=\intr (U_{0,i}\cap N_{\de}(\La) )^{3\de/5}=\intr U_{0,i}^{3\de/5}\cap N_{2\de/5}(\La) \mid i\in \I_0\},
\]
then Remark~\ref{rem:expansion}(e) implies that there is $\de_0'>0$ such that $\rho'(s_i^{-1})$ is $(\la',U_{0,i}';\de_0')$-expanding for all $i\in\I_0$.
For any $x\in \La$, since $d(x,\phi(x))<\de/5$, we have
\[
\ov{B}_{3\de/5}(\phi(x))\subset B_{4\de/5}(x) \subset B_{\de}(x) \subset U_{0,i}\cap N_\de(\La)
\]
for some $i\in\I_0$. We may choose $\de_0'$ and $\de'$ so that $\de'<\de_0'<2\de/5$. Then $\D_0'$ is another expansion datum of $\rho'$ such that $\D_0'\prec\D'$.

In order to check that $(\D_0'\prec\D';N)$ is a uniform meandering hyperbolicity datum of $\rho'$, let $0<\eta\leq \de'<2\de/5$ and $x\in\La$. We need to show that all rays in $\Ray_x(\D_0',\eta)$ are $(\D',\eta;N)$-equivalent.

Let $x\in \La$ and $(\al,p')$ be a $(\D_0',\eta)$-code for $\phi(x)$.
Recall Definition~\ref{def:code}. Since $(\al,p')$ is an $(\D'_0,\eta)$-code for $\phi(x)\in\La'$, we have
\begin{align*}
p'_0&=\phi(x),\\
p'_{k+1}&=\rho'(s_{\al(k)}^{-1})(p'_k),\\
B_{\eta}(p'_k)&\subset U'_{0,\al(k)} 
\end{align*}
for all $k\in\N_0$. By the definition of $U_{0,i}':=\intr U_{0,i}^{3\de/5}\cap N_{2\de/5}(\La)$, it follows that $B_{\eta+3\de/5}(p_k')\subset U_{0,\al(k)}$ for all $k\in \N_0$.
Let $p_k=\phi^{-1}(p_k')$ for $k\in \N_0$ and $p : \N_0 \to \La$ be the sequence corresponding to $\{p_k\}_{k\in \N_0}$.
We first check
\[
(\phi^{-1}\circ p')(0)=\phi^{-1}(\phi(x))=x
\]
and then, for $k\in\N$, check
\[
p_{k+1}=(\phi^{-1}\circ p')(k+1)
=\phi^{-1}[\rho'(s_{\al(k)}^{-1})(p'_k)]
=\rho(s_{\al(k)}^{-1})[(\phi^{-1}\circ p')(k)]=\rho(s_{\al(k)}^{-1})(p_k),
\]
where the second equality is due to the equivariance of $\phi$.
Since $d(p_k,p_k')<\de/5$, we have that for all $k\in \N_0$,
\[
B_{\eta+2\de/5}(p_k) \subset B_{\eta+3\de/5}(p_k')\subset U_{0,\al(k)},
\]
which implies that $(\al, p)$ is a $(\D_0,\eta+2\de/5)$-code for $x\in \La$.

By the uniform meandering hyperbolicity of $\rho$, there is a finite chain of interpolating rays $c^\al=c^{\ga_1}, c^{\ga_2},\ldots, c^{\ga_n}=c^\be$ in $\Ray_x(\D,\eta+2\de/5)$ such that
\[
c^{\ga_1} \approx^N_{(\D,\eta+2\de/5)} c^{\ga_2} \approx^N_{(\D,\eta+2\de/5)} \cdots \approx^N_{(\D,\eta+2\de/5)} c^{\ga_n}.
\]
\begin{claim*}
$\Ray_x(\D,\eta+2\de/5)\subset \Ray_{\phi(x)}(\D', \eta)$.
\end{claim*}
If the claim holds, we have
$ c^{\ga_1} \approx^N_{(\D',\eta)} c^{\ga_2} \approx^N_{(\D',\eta)} \cdots \approx^N_{(\D',\eta)} c^{\ga_n}$
and, therefore, we conclude that $\rho'$ is uniformly meandering-hyperbolic with a datum $(\D_0'\prec \D';N)$.

Now it only remains to prove the claim. Let $(\ga, q)$ be a $(\D,\eta+2\de/5)$-code for $x\in \La$. Then for all $k\in \N_0$,
\[
B_{\eta+\de/5}(\phi(q_k)) \subset B_{\eta+2\de/5}(q_k)\subset U_{\ga(k)}\cap N_{4\de/5}(\La),
\]
which implies that $B_\eta(\phi(q_k))\subset U_{\ga(k)}'$ and thus $(\ga, q')$ is a $(\D',\eta)$-code for $\phi(x)$ as desired, where $q'$ is the sequence defined by $q'_k=\phi(q_k)$ for $k\in \N_0$. The claim is proved.

\section{Uniform lattices in semisimple Lie groups  are structurally stable}\label{sec:ulattice} 

In this section we shall prove the following theorem that, for every uniform lattice $\Ga$ in a semisimple Lie group $G$, the $\Ga$-action on each flag manifold $G/P$ is meandering-hyperbolic.

\begin{theorem}\label{thm:ulattice}
Let $\Ga< G$ be a uniform lattice in a semisimple  Lie group $G$ (with finitely many connected components and finite center). Then, for each face $\tad\subset\sid$, the $\Ga$-action on the flag manifold $\Flag(\tad)$ is uniformly meandering-hyperbolic.
\end{theorem}

From the structural stability theorem (Theorem~\ref{thm:main}) we then conclude:

\begin{corollary}\label{cor:ulattice}
The actions on flag manifolds of uniform lattices in semisimple Lie groups are structurally stable in the sense of Lipschitz dynamics.
\end{corollary}

\subsection{Sketch of proof}

Let $X$ be the symmetric space associated with the semisimple Lie group $G$.  Our proof will use definitions and results regarding asymptotic geometry of $X$ established 
in a series of papers of Kapovich, Leeb and Porti \cites{KL18b, KLP17, KLP18}, to which we refer the reader for details; surveys of these notions and results can be found in \cites{KLP16, KL18a}.

The key notion used in the proof is the one of {\em Morse quasigeodesics} in $X$, which was first introduced in \cite{KLP14} in order to give one of many alternative interpretations of Anosov subgroups. While uniform lattices in higher rank Lie groups are clearly non-Anosov, some key geometric results developed for the purpose of analyzing Anosov subgroups are still very useful when dealing with higher rank lattices, as we will see below.

In the remainder of the section we prove Theorem~\ref{thm:ulattice} in the following steps:
\begin{enumerate}[label=(\arabic*),leftmargin=35pt,parsep=0pt,itemsep=0pt]
\item[(\S\ref{sec:ulattice1})] we find an expansion datum $\D$ for the action;
\item[(\S\ref{sec:ulattice2})] we show that coding rays are Morse quasigeodesics, and
\item[(\S\ref{sec:ulattice3})] that asymptotic Morse quasigeodesics can be interpolated;
\item[(\S\ref{sec:ulattice4})] for each $\eta \in (0,\de]$, we find a refinement $\D_\eta$ of the expansion datum $\D$;
\item[(\S\ref{sec:ulattice5})] we verify the uniform meandering hyperbolicity by using Proposition~\ref{prop:umhp} below.
\end{enumerate}

\begin{proposition}\label{prop:umhp}
Let $\rho:\Ga \to \Homeo(M)$ be expanding at $\La$ with a datum $\D=(\I,\U,\vS,\allowbreak\de,L,\la)$. Suppose that, for each $\eta\in (0,\de]$, there exist a positive constant $r_\eta<\eta$, an integer $N_\eta>0$ and an expansion datum $\D_\eta=(\I_\eta,\U_\eta,\vS_\eta,\allowbreak r_\eta,L_\eta,\la_\eta)\succ \D$ such that for every $x\in \La$, all rays in $\Ray_x(\D,\eta)$ are $(\D_\eta,r_\eta;N_\eta)$-equivalent as rays in $\Ray_x(\D_\eta,r_\eta)$. Then $\rho$ is uniformly meandering-hyperbolic.
\end{proposition}

\begin{proof} 
To define a uniform meandering hyperbolicity datum of $\rho$, we set $\de_0=r_{\de/2}+\de/2<\de$, $\U_0=\{\intr U_{i}^{\de/2} \mid i\in\I\}$ (recall \eqref{eqn:ur}) and $\D_0=(\I,\U_0,\vS,\allowbreak \de_0,L,\la)$. Obviously, $\D_0$ is  an expansion datum of $\rho$ since $\U_0\prec\U$ and $\de_0 <\de$. We take $\D_0$ as the initial expansion datum of $\rho$ and claim that $(\D_0\prec\D_{\de/2};N_{\de/2})$ is a uniform meandering hyperbolicity datum of $\rho$.

For $\eta\in(0,r_{\de/2})$, let $\al$ and $\be$ be $(\D_0,\eta)$-codes for $x\in \La$. 
Then, by the definition of $\U_0$, it can be easily seen that $\al$ and $\be$ are $(\D,\eta+\de/2)$-codes for $x$ and hence $(\D,\de/2)$-codes for $x$. In other words, all $(\D_0,\eta)$-codes are regarded as $(\D,\de/2)$-codes for any $\eta\in(0,r_{\de/2})$. By the hypothesis, there is a finite chain of interpolating rays $c^\al=c^{\ga_1}, c^{\ga_2},\ldots, c^{\ga_n}=c^\be$ in $\Ray_x(\D_{\de/2},r_{\de/2})$ such that
\[
c^{\ga_1} \approx^{N_{\de/2}}_{(\D_{\de/2},r_{\de/2})} c^{\ga_2} \approx^{N_{\de/2}}_{(\D_{\de/2},r_{\de/2})} \cdots \approx^{N_{\de/2}}_{(\D_{\de/2},r_{\de/2})} c^{\ga_n}.
\]
Since $\eta<r_{\de/2}$, all the $(\D_{\de/2},r_{\de/2})$-codes $\ga_i$ are $(\D_{\de/2},\eta)$-codes. Hence all rays in $\Ray_x(\D_0,\eta)$ are $(\D_{\de/2},\eta;N_{\de/2})$-equivalent. Furthermore, one can easily check that $\D_0 \prec \D_{\de/2}$. Therefore, we conclude that $(\D_0\prec\D_{\de/2};N_{\de/2})$ is a uniform meandering hyperbolicity datum of $\rho$.
\end{proof}

\subsection{Expansion}\label{sec:ulattice1}

In order to obtain an expansion data for the $\Ga$-action on the flag manifold $\Flag(\tad)$, we proceed as in Remark~\ref{rem:expansion}(c). Roughly speaking, we will construct a finite open cover of $\Flag(\tad)$ whose elements are expanding subsets for some elements of $\Ga$.

First fix a number $\la>1$ and a base point $x\in X$.
Let $D>0$ be the diameter of a compact fundamental domain for the $\Ga$-action on $X$. Then the orbit $\Ga x \subset X$ is $D$-dense in $X$, that is, $B_D(y)\cap \Ga x \neq \emptyset$ for any $y\in X$. In view of \cite{KLP17}*{Theorem~2.41}, the $D$-density of $\Ga x$ implies that for each $\tau\in\Flag(\tad)$, there is an element $g_\tau\in\Ga$ such that $\ef(g_\tau^{-1}, \tau)>\lambda$. Define an open subset $U_\tau\subset\Flag(\tad)$ by
\[
U_\tau=\{\tau'\in\Flag(\tad) \mid \ef(g_\tau^{-1},\tau')>\la\},
\]
so that $\tau\in U_\tau$ and $g^{-1}_\tau$ is $(\la,U_\tau)$-expanding. Then we have an open cover $\{U_\tau\}_{\tau\in \Flag(\tad)}$ of $\Flag(\tad)$. Since $\Flag(\tad)$ is compact, there is a finite subcover $\{U_{\tau_1},\ldots, U_{\tau_k}\}$. For simplicity, let us abbreviate $U_{\tau_i}$ and $g_{\tau_i}$ to $U_i$ and $g_i$, respectively, for $i=1,\ldots,k$.

By adding extra generators to $\{g_1,\ldots,g_k\}$ with empty expanding subsets, we obtain a symmetric generating set $\vS$ of $\Ga$. Let $\I$ be the index set for $\vS$ and $\U$ the collection of all $(\la,g_i^{-1})$-expanding subsets $U_i$ where $g_i\in\vS$. Since $\Flag(\tad)$ is compact, $\vS$ is finite and the $\Ga$-action is by $C^1$-diffeomorphisms, there exist a constant $L>1$ such that every $g_i\in\vS$ is $L$-Lipschitz on $\Flag(\tad)$. Let $\de<\de_{\U}$ be a Lebesgue number of $\U$. Then $\D=(\I,\U,\vS,\allowbreak\de,L,\la)$ is an expansion datum of the $\Ga$-action on $\Flag(\tad)$.

\subsection{Coding rays are uniform Morse quasigeodesics}\label{sec:ulattice2}

In the proof of the next lemma we will be using the notions of $\tad$-regular (and uniformly regular) sequences and $\tad$-convergence property for such sequences; the reader will find a detailed treatment of these notions in \cite{KLP17}*{Section 4}. 

\begin{lemma}\label{lem:CodingMorse}
Let $\eta\in(0,\de]$ and $\al$ be a $(\D,\eta)$-code for $\tau\in \Flag(\tad)$ and $c^\al:\N_0\to\Ga<G$ be the ray associated to $\al$. Then the sequence $c^\al x$ in $X$ is a $(\Theta,R)$-Morse quasigeodesic with endpoint $\tau$ for some data $(\Theta,R)$ depending only on the pair $(\D,\eta)$.
\end{lemma}

\begin{proof}
Recall from Corollary~\ref{cor:qg} that the ray $c^\al$ is a uniform quasigeodesic in $\Ga$ which is quasi-isometric to $X$. Thus, in order to prove the Morse property, it suffices to establish the uniform $\tad$-regularity of the sequence $c^\al$; see \cite{KLP18}*{Theorem~1.3}.  

First of all, there exists a closed ball $B\subset\Flag(\tad)$ such that $c^\al$ restricted to $B$ subconverges uniformly to $\tau$. Moreover, the limit set of a subsequence of $(c^\al)^{-1}(\tau)$ lies in $B$ by Theorem~\ref{thm:minimal}. Since $B$ is Zariski dense in $\Flag(\tad)$, the sequence $c^\al$ is $\tad$-regular; see the proof of \cite{KL18b}*{Theorem~9.6}.

Now we show that the sequence $c^\al$ is uniformly $\tad$-regular.
We first verify that the sequence $c^\al$ conically converges to $\tau$. Since it suffices to prove this for all subsequences in $c^\al$, we will freely pass to such subsequences. Note that the inverse sequence $(c^\al)^{-1}$ is also $\iota\tad$-regular and, hence, flag-subconverges to a simplex $\tau_-\in \Flag(\iota\tad)$. Furthermore, the sequence $c^\al$ subconverges to $\tau$ uniformly on compacts in the open Schubert cell $C(\tau_-)\subset  \Flag(\tad)$ of $\tau_-$ and, whenever $B'$ is a closed ball in  $\Flag(\tad)$ which is \emph{not}  contained in $C(\tau_-)$, the sequence $c^\al$ cannot subconverge to $\tau$ uniformly on $B'$. Thus the ball $B$, and hence its center $\nu$ (which is the limit of a subsequence in $(c^{\al})^{-1}(\tau)$), is contained in $C(\tau_-)$. Since $\nu\in C(\tau_-)$, according to \cite{KLP17}*{Proposition~5.31}, the convergence $c^\al\to \tau$ is conical. 

Lastly, by \cite{KLP17}*{Theorem~2.41}, the uniform exponential expansion of the sequence $(c^{\al})^{-1}$ at $\tau$ (guaranteed by Lemma~\ref{lem:nest}) implies (uniform) $\Theta$-regularity of the sequence $c^\al$, where $\Theta$ is a certain fixed $W_\tad$-convex compact subset of the open star $\mathrm{ost}(\tad)$ of $\tad$ in $\sid$.  
\end{proof}

\subsection{Morse interpolation}\label{sec:ulattice3}
 
As in Definition~\ref{def:N-equiv}, given two maps $p,q:\R_+\to X$ and an integer $N\ge0$, we define the relation $p\approx^N q$ by the condition that there exist unbounded monotonic sequences $(P_k)$ and $(Q_j)$ in $\R_+$ such that the subsets $\{p(P_k)\mid k\in\N_0\}$ and $\{q(Q_j)\mid j\in\N_0\}$ in $X$ are within Hausdorff distance $N$ from each other.

\begin{lemma}\label{lem:interpolation}
Let $p$ and $q$ be $(\Theta,R)$-Morse quasigeodesic rays in $X$ with $p(\infty)=q(\infty)=\tau$. Let $\Theta'\subset \mathrm{ost}(\tad)$ be a $W_\tad$-convex compact subset such that $\Theta\subset\mathrm{int}(\Theta')$. Given a constant $l>0$, there is a $(\Theta',R)$-Morse quasigeodesic ray $r$ with endpoint $r(\infty)=\tau$ in $X$ such that 
\begin{enumerate}[label=(\roman*),parsep=0pt,itemsep=0pt]
\item $r(t_n) \in p(\R_+)$ if $n$ is odd and $r(t_n) \in q(\R_+)$ if $n$ is even, and hence  $p\approx^0  r\approx^0 q$,
\item  the restriction of $r$ to each interval $[t_n, t_{n+1}]$ is a $\Theta'$-regular geodesic,
\item $d_X(r(t_n),r(t_{n+1}))\geq l$ for all integers $n\geq 0$.
\end{enumerate}
Moreover, if $p(0)=q(0)=x\in X$, then $r$ can be chosen such that $r(0)=x$. 
\end{lemma}

\begin{figure}[ht]
\labellist
\pinlabel {$x$} at -6 70
\pinlabel {$r(1)$} at 80 78
\pinlabel {$r(2)$} at 146 74
\pinlabel {$r(3)$} at 224 82
\pinlabel {$r(4)$} at 340 56
\pinlabel {$r(5)$} at 430 100
\pinlabel {$p$} at 496 154
\pinlabel {$q$} at 474 0
\pinlabel {$r$} at 496 84
\endlabellist
\centering
\includegraphics[width=0.9\textwidth]{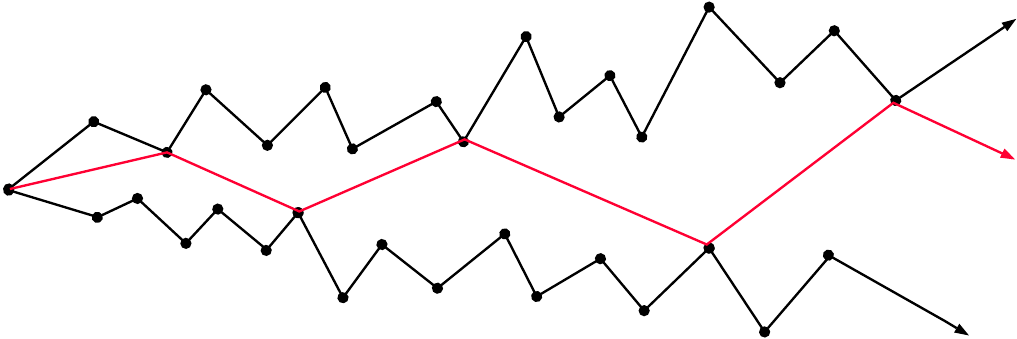}
\caption{Interpolating ray.}
\label{fig:int-ray}
\end{figure}

\begin{proof}
According to \cite{KLP17}*{Lemma 5.52}, both $p$ and $q$ are contained in the $R$-neighborhood of a cone $V=V(x,\st(\tau))$. We define the points $r(t_n)$ inductively as follows (see Figure~\ref{fig:int-ray}). Set$t_0=0$ and $r(0)=q(0)$. Suppose that $n$ is odd and $r(t_n)\in p(\R_+)$ is defined. Since the Morse quasigeodesic ray $q$ is $\Theta$-regular, the limit (as $t\to \infty$) of directions of oriented segments  
\[
r(t_n) q(t) 
\]
belongs to $\st_{\Theta}(\tau)$. Therefore, for a sufficiently large $s_n>0$, the directions of both geodesic segments $r(t_n)q(s_n)$ and $\bar r(t_n)\bar q(s_n)$ are in the interior of $\st_{\Theta'}(\tau)$ and $d_X(r(t_n),q(s_n))>l$ where $\bar r(t)$ and $\bar q(t)$ are the nearest projections of $r(t)$ and $q(t)$ onto $V$, respectively. Then we set $t_{n+1}=t_n+d_X(r(t_n),q(s_n))$ and $r(t_{n+1})=q(t_{n+1}$). When $n$ is even and $r(t_n)\in q(\R_+)$, the point $r(t_{n+1})$ is defined in the same way.

Since the cone $V$ is convex in $X$, the piecewise geodesic ray $\bar r :[0,\infty)\to V$ with vertices $(\bar r(t_n))_{n\in \N_0}$ is a geodesic ray whose restriction to each interval $[t_n,t_{n+1}]$ is $\Theta'$-regular and longitudinal. By the nestedness property of $\Theta'$-cones (see \cite{KLP17}*{Section 2.12}), the entire ray {$\bar r$} is a $\Theta'$-Finsler geodesic ray in $V$. Then the piecewise geodesic ray $r$ with vertices $(r(t_n))_{n\in \N_0}$ is $R$-Hausdorff close to the $\Theta'$-Finsler geodesic ray $\bar r$. Therefore, the ray $r$ is a $(\Theta',R)$-Morse quasigeodesic ray. From the above construction, the ray $r$ clearly satisfies $r(\infty)=\tau$ as well as (i), (ii) and (iii) of the lemma.
\end{proof}

\subsection{Refined expansion data}\label{sec:ulattice4}

Now we find a refinement $\D_\eta$ of the expansion datum $\D$ 
for each $\eta\in (0,\de]$. 

Let $\al$ and $\be$ be $(\D,\eta)$-codes for $\tau \in \Flag(\tad)$. By Lemma~\ref{lem:CodingMorse}, the associated rays $c^\al$ and $c^\be$ project in $X$ to $(\Theta,R)$-Morse quasigeodesic rays with endpoint $\tau$, where the datum $(\Theta,R)$ depends only on $\D$ and $\eta$. Hence $c^\al(n)x$ and $c^\be(n)x$ are contained in the $R$-neighborhood of $V(x, \text{st}_\Theta(\tau))$ for any integer $n\geq 0$. We fix compact $W_\tad$-convex subsets $\Theta', \Theta''\subset\text{ost}(\tad)$ in $\sid$ such that
\[
\Theta \subset \text{int}(\Theta') \subset \Theta' \subset \text{int}(\Theta'')\subset \Theta''.
\]

We first choose a constant $l$ such that
\begin{align}\label{eqn:l}
l>\max\{l_1,l_2,l_3\}+4(D+R),
\end{align}
where the constants $l_1$, $l_2$ and $l_3$ are defined as follows. First, choose a constant $l_1=l_1(\Theta',\Theta'',\allowbreak D,R)>0$ so that if $d_\De(x,y)\in V(0, W_{\tad}\Theta')\cap\De$ and $d_X(x,y)>l_1$, then 
\[
d_\De(x',y')\in V(0, W_{\tad}\Theta'')\cap\De
\]
for any $x'\in B_{D+R}(x)$ and any $y'\in B_{D+R}(y)$. This is possible due to the following inequality
\[
\|d_\De(x,y)-d_\De(x',y')\| \le d_X(x,x')+d_X(y,y') \le 2(D+R).
\]

For the second constant $l_2>0$, we recall \cite{KLP17}*{Theorem~2.41} again that  there are constants $C, A>0$ depending only on $x, R, D$ and the chosen Riemannian metric on $\Flag(\tad)$ such that if $d_X(gx,V(x,\st(\tau)))\le2(D+R)$, then 
\begin{align}\label{eqn:l2}
C^{-1}\cdot d_X(gx,\partial V(x,\st(\tau)))-A
\leq \log\ef(g^{-1},\tau)
\leq C\cdot d_X(gx,\partial V(x,\st(\tau)))+A,
\end{align}
where $\partial V(x,\st(\tau)))$ is the boundary of $V(x,\st(\tau)))$ in $X$. (Recall from \eqref{eqn:ep} that $\ef(g^{-1},\tau)$  is the expansion factor of $g^{-1}$ at $\tau\in \Flag(\tad)$). Given $\la>1$, the inequality (\ref{eqn:l2}) makes it possible to choose a constant $l_2=l_2(x,R,D,\la)>0$ such that if $d_X(gx, \partial V(x,\st(\tau)))\geq l_2$ and $d_X(gx, V(x, \text{st}(\tau)))\leq 2(D+R)$, then $\ef(g^{-1},\tau)\geq 2\la$. 

Finally, we choose a constant $l_3>0$ as follows: if $y\in V(x,\st_{\Theta''}(\tau))$, the distance of $y$ from $\partial V(x,\st(\tau))$ grows linearly with the distance of $x$ and $y$. Its linear growth rate only depends on $\Theta''$ and hence there is a constant $l_3=l_3(\Theta'',l_2)>0$ such that if $y\in V(x,\st_{\Theta''}(\tau))$ and $d_X(x,y)> l_3$, then $d_X(y, \partial V(x,\st(\tau)))> l_2+2(D+R)$ for any $x,y\in X$ and any $\tau \in \Flag(\tad)$.

Now we are ready to add new generators to $\vS$. Let $\vS_0$ be the set of all elements $g\in\Ga$ such that 
\begin{align}\label{eqn:newgen}
d_\De(x,g x) \in V(0,W_{\tad}\Theta'')\cap\De
\quad\textup{and}\quad
l-2D<d_X(x,g x) < 2l+2D.
\end{align}
Since $\Ga$ acts properly discontinuously on $X$, $\vS_0$ is finite. To each $g\in\vS_0$, we associate its expansion domain 
\[
U_g^\la=\{\tau\in\Flag(\tad) \mid \ef(g^{-1},\tau)>\la \}.
\]
Add all elements $g\in\vS_0$ together with $\la$-expanding domains $U^\la_g$ to $\vS$ and $\U$, respectively. Furthermore, if necessary, add the inverses of $\vS_0$ together with their (possibly empty) $\la$-expanding domains. Then we obtain a symmetric generating finite set $\vS_\eta$ and the collection $\U_\eta$ of all $(\la,g)$-expanding subsets for $g\in \vS_\eta$. The index set $\I_\eta$ for $\vS_\eta$, and the Lipschitz constant $L_\eta \ge L$ are determined in the obvious way. Lastly, by setting
\begin{align}\label{eqn:deeta}
\de_\eta:=\min\left\{\frac{99}{100}\eta,\;\min_{\substack{g\in\vS_0 \\ U^{2\la}_g \neq \emptyset}} d_X(U^{2\la}_g,\partial U^\la_g) \right\}<\eta\le\de,
\end{align}
we obtain a refined expansion datum
\[
\D_\eta=(\I_\eta,\U_\eta,\vS_\eta,\de_\eta,L_\eta,\la)
\succ(\I,\U,\vS,\allowbreak\de,L,\la)=\D
\]
for the $\Ga$-action on $\Flag(\tad)$. Note that by the definition of $\de_\eta$, if $\tau\in U_g^{2\la}$ for some $g\in\vS_0$, then $B_{\de_\eta}(\tau)\subset U^{\la}_g$.

\subsection{Uniform meandering hyperbolicity}\label{sec:ulattice5}

Lastly, we are left with checking the uniform meandering hyperbolicity condition.

Fix $\tau\in\Flag(\tad)$. Let $\al$ and $\be$ be $(\D,\eta)$-codes for $\tau$. In view of Proposition~\ref{prop:umhp}, it suffices to show that the associated rays $c^\al$ and $c^\be$ are $(\D_\eta,\de_\eta;N)$-equivalent for some integer $N\ge1$. But we claim, as a matter of fact, that there exists a $(\D_\eta,\de_\eta)$-code $\ga$ for $\tau$ such that 
\begin{align}\label{eqn:(D,de,0)}
c^\al\approx^0_{(\D_\eta,\de_\eta)} c^\ga\approx^0_{(\D_\eta,\de_\eta)} c^\be.
\end{align}
That is to say, the rays $c^\al$ and $c^\be$ are $(\D_\eta,\de_\eta;0)$-equivalent, and hence $(\D_\eta,\de_\eta;1)$-equivalent.

Let us prove the claim. By Lemma~\ref{lem:CodingMorse}, the sequences $c^\al x$ and $c^\be x$ are $(\Theta,R)$-Morse quasigeodesic rays in $\Ga x$ with endpoint $\tau$ for a uniform datum $(\Theta,R)$ depending only on $(\D,\eta)$. Then both $c^\al x$ and $c^\be x$ are contained in the $R$-neighborhood of a cone $V(x,\st(\tau))$. We fix compact $W_\tad$-convex subsets $\Theta', \Theta''\subset\text{ost}(\tad)\subset \sid$ such that $\Theta \subset \text{int}(\Theta') \subset \Theta' \subset \mathrm{int}(\Theta'')$. Applying Lemma~\ref{lem:interpolation} with the chosen constant $l>0$ in \eqref{eqn:l}, there is a $(\Theta',R)$-Morse quasigeodesic $r$ in $\Ga x$ with $r(0)=x$ and $r(\infty)=\tau$ and an unbounded increasing sequence $t_n\in \R_+$ such that the properties (i), (ii) and (iii) of the lemma hold. According to the proof of Lemma~\ref{lem:interpolation}, we may assume that $r(t_n)\in c^\al(\N_0)x$ if $n$ is odd and $r(t_n)\in c^\be(\N_0)x$ if $n$ is even, and the piecewise geodesic ray with vertices $(\bar r(t_n))_{n\in\N_0}$ is a $\Theta'$-Finsler geodesic ray in $V(x, st(\tau))$.

The properties (ii) and (iii) of Lemma~\ref{lem:interpolation} imply that $t_{n+1}-t_n>l$ for every $n\in \N_0$.
For each $n\in \N_0$, set \[  k_n=\left \lfloor \frac{t_{n+1}-t_n}{l} \right \rfloor -1. \]
Let $(u_k)_{k\in \N_0}$ be an increasing sequence consisting of $\{t_n+jl \mid n\in \N_0, \ j=0,\ldots,k_n \}$. Then it is easy to see that $(t_k)$ is a subsequence of $(u_k)$ and for every $k\in \N_0$, 
\begin{align}\label{eqn:duks}
l\leq u_{k+1}-u_k <2l, \textup{ that is, } l\leq d_X(r(u_{k+1}), r(u_{k}))<2l.
\end{align}
For a sufficiently large $l$, we may assume that the piecewise geodesic ray with vertices $(\bar r(u_k))_{k\in\N_0}$ is a $\Theta''$-Finsler geodesic ray. This is possible since $r$ is a $(\Theta',R)$-Morse quasigeodesic. For the proof, we refer the reader to \cite{KLP17}*{Theorem 5.53}.

Due to the $D$-density of $\Ga x$ in $X$, for each $r(u_k)$ there exists an element $g_k\in \Ga$ such that $g_k x \in B_D(r(u_k))$. 
Here we require that if $r(u_k)\in c^\al (\N_0)x$ (resp. $r(u_k)\in c^\be (\N_0)x$), then $g_k \in c^\al(\N_0)$ (resp. $g_k\in c^\be (\N_0)$) and $g_0=e$.
This gives us that $g_{k}\in c^\al(\N_0)$ if $u_k=t_n$ for an odd number $n$ and  $g_{k}\in c^\be(\N_0)$ if $u_k=t_n$ for a even number $n$. Thus $\{g_k\mid k\in\N_0\}$ has infinitely many elements of both $c^\al(\N_0)$ and $c^\be(\N_0)$.

We will show that there is a $(\D_\eta,\de_\eta)$-code $\ga$ for $\tau$ such that $c^\gamma(k)=g_{k+1}$ for $k\in \N_0$. Then (\ref{eqn:(D,de,0)}) will immediately follow.

\begin{lemma}\label{lemma:gk+1gk}
Let $h_k=g_k^{-1}g_{k+1}$ for $k\in \N_0$. Then $h_k \in \vS_\eta$ and $\ef(h_k^{-1}, h_{k-1}^{-1}\cdots h_0^{-1}\tau)>2\lambda$ for every $k\in \N_0$.
\end{lemma}

\begin{proof}
To show that $h_k =g_k^{-1}g_{k+1} \in \vS_\eta$, we need to verify (\ref{eqn:newgen}).
From the facts that $g_k x \in B_D(r(u_k))$, $\bar r(u_k) \in B_R(r(u_k))$ and (\ref{eqn:duks}), it easily follows that
\begin{align}\label{eqn:cond1forgk}
l-2D\leq d_X(g_k x, g_{k+1}x) =d_X(x, h_kx)<2l+2D.
\end{align}
Moreover, since the geodesic $r(u_k) r(u_{k+1})$ is a $\Theta'$-regular geodesic of length at least $l>l_1$ and $g_k x \in B_D(r(u_k))\subset B_{D+R}(r(u_k))$ for all $k\in \N_0$, the definition of $l_1$ gives
\begin{align*}
d_\De(x, h_kx)=d_\De(g_k x, g_{k+1}x) \in V(0,W_{\tad}\Theta'')\cap\De.
\end{align*}
Thus every $h_k$ is an element of $\vS_\eta$.

Now we will prove the second statement $\ef(h_k^{-1}, h_{k-1}^{-1}\cdots h_0^{-1}\tau)>2\lambda$.
An easy computation gives that $h_{k-1}^{-1}\cdots h_0^{-1}=g_k^{-1}g_0=g_k^{-1}$ and hence
\[ \ef(h_k^{-1}, h_{k-1}^{-1}\cdots h_0^{-1}\tau)=\ef((g_{k}^{-1}g_{k+1})^{-1}, g_k^{-1}\tau). \]
One can estimate the expansion factor $\ef((g_{k}^{-1}g_{k+1})^{-1}, g_k^{-1}\tau)$ through the distances from $g_k^{-1}g_{k+1}x$ to the Weyl cone $V(x, \st(g_k^{-1}\tau))$ and its boundary $\partial V(x, \st(g_k^{-1}\tau))$ as follows. The conditions of $g_k x \in B_D(r(u_k))$ and $\bar r(u_k) \in B_R(r(u_k))$ yields 
\begin{align}\label{eqn:gkuk} 
d_X(g_k x,\bar r(u_k))\leq d_X(g_k x, r(u_k))+d_X(r(u_k), \bar r(u_k)))\leq D+R.
\end{align}
The fact that the piecewise geodesic ray with vertices $\bar r(u_k)$ is a $\Theta''$-Finsler geodesic ray in $V(x,\st(\tau))$ implies $\bar r(u_{k+1}) \in V(\bar r(u_k), \st_{\Theta''}(\tau))\subset V(\bar r(u_{k}), \st(\tau))$. Thus, 
\begin{align*} 
d_X(g_{k+1} x, V(g_k x, \st(\tau))) &  \leq d_X(g_{k+1} x, \bar r(u_{k+1}))+d_X(\bar r(u_{k+1}), V(g_k x, \st(\tau))) \\
& \leq d_X(g_{k+1} x, \bar r(u_{k+1}))+d_H(V(\bar r(u_{k}), \st(\tau)), V(g_kx, \st(\tau))) \\
& \leq d_X(g_{k+1} x, \bar r(u_{k+1}))+d_X(g_kx,\bar r(u_k)) \\
& \leq 2(D+R),
\end{align*}
where $d_H$ denotes Hausdorff distance between subsets of $X$. The third inequality above follows from the fact that the Hausdorff distance between $V(y,\st(\tau))$ and $V(y', \st(\tau))$ is bounded above by $d_X(y,y')$. 

Next we estimate the distance between $g_k^{-1}g_{k+1}x$ and $\partial V(x, \st(g_k^{-1}\tau))$. From the inequalities (\ref{eqn:cond1forgk}) and (\ref{eqn:gkuk}), 
\[ d_X(\bar r(u_{k}),\bar r(u_{k+1}))\geq d_X(g_k x,g_{k+1}x)-2(D+R)\geq l-2(2D+R)>l_3.\]
Since $\bar r(u_{k+1}) \in V( \bar r(u_{k}),\st_{\Theta''}(\tau))$, it follows from the definition of $l_3$ that
\begin{align*}
d_X(\bar r(u_{k+1}), \partial V(\bar r(u_{k}),\st(\tau)))\geq l_2+2(D+R).
\end{align*}

Let $p_{k}$ be the point on $\partial V(g_kx,\st(\tau))$ where $d_X(g_{k+1}x, \partial V(g_kx, \st(\tau)))$ is realized.  If we denote by $c_{y}^{\xi}:[0,\infty)\to X$ the geodesic ray starting at $y$ toward $\xi \in \partial_\infty X$, then $p_k$ is written as $p_k=c_{g_kx}^{\xi_k}(z_k)$ for some $\xi_k\in \partial \st(\tau)$ and $z_k>0$. 
Since the geodesic ray $c_{\bar r(u_k)}^{\xi_k}$ is contained in $V(\bar r(u_k),\partial \st(\tau))=\partial V(\bar r(u_k), \st(\tau))$ and $d_X(c_{g_kx}^{\xi_k}(z),c_{\bar r(u_k)}^{\xi_k}(z)) \leq d_X(g_kx,\bar r (u_k))\leq D+R$ for all $z>0$, 
\begin{align*} 
l_2+2(D+R) & \leq d_X(\bar r(u_{k+1}), \partial V(\bar r(u_{k}), \st(\tau))) \\ 
& \leq d_X(\bar r(u_{k+1}), c_{\bar r(u_k)}^{\xi_k}(z_k)) \\
& \leq d_X(\bar r(u_{k+1}), g_{k+1}x)+d_X(g_{k+1}x, c_{g_kx}^{\xi_k}(z_k)) +d_X(c_{g_kx}^{\xi_k}(z_k),c_{\bar r(u_k)}^{\xi_k}(z_k)) \\
& \leq 2(D+R)+d_X(g_{k+1} x, \partial V(g_k x, \st(\tau))),
\end{align*}
which is equivalent to 
$ d_X(g_k^{-1}g_{k+1}x, \partial V(x, \st(g_k^{-1}\tau)))=d_X(g_{k+1}x, \partial V(g_k x, \st(\tau)))>l_2$.
Recalling that $d_X(g_k^{-1}g_{k+1}x, V(x, \st(g_k^{-1}\tau)))=d_X(g_{k+1}x, V(g_k x, \st(\tau)))<2(D+R)$, the definition of $l_2$ leads us to conclude that $\ef((g_{k}^{-1}g_{k+1})^{-1}, g_k^{-1}\tau)>2\la$  for all integers $k\geq 0$.
This completes the proof of the lemma.
\end{proof}

Now we are ready to finish the proof of Theorem~\ref{thm:ulattice}. By Lemma~\ref{lemma:gk+1gk}, we can define a sequence $\ga : \N_0 \to \I_\eta$ so that for every $k\in \N$, 
$s_{\ga(k)}=h_k \in \vS_\eta$ and
$\ef(s_{\ga(k)}^{-1}, s_{\ga(k-1)}^{-1}\cdots s_{\ga(0)}^{-1}\tau)=\ef(h_k^{-1}, h_{k-1}^{-1}\cdots h_0^{-1}\tau)>2\la$,
that is, $s_{\ga(k-1)}^{-1}\cdots s_{\ga(0)}^{-1}\tau\in U^{2\lambda}_{s_{\ga(k)}}$.
By (\ref{eqn:deeta}), it holds that $$B_{\de_\eta}(s_{\ga(k-1)}^{-1}\cdots s_{\ga(0)}^{-1}\tau) \subset U^{\la}_{s_{\ga(k)}}$$ 
for every $k\in\N_0$ and therefore $\ga$ is a $(\D_\eta,\de_\eta)$-code for $\tau$. Furthermore,  
\[
c^{\ga}(k)=s_{\ga(0)}s_{\ga(1)}\cdots s_{\ga(k)}=h_0h_1\cdots h_k=g_{k+1}
\]
for any integer $k\geq 0$. Since $\{c^\ga(k)=g_{k+1} \mid k\in \N_0\}$ has infinitely many elements of both $c^\al(\N_0)$ and $c^\be(\N_0)$, we finally conclude that $c^\al\approx^0_{(\D_\eta,\de_\eta)} c^\ga\approx^0_{(\D_\eta,\de_\eta)} c^\be$ and thus complete the proof.

\section{Actions of word-hyperbolic groups}\label{sec:wordhyp}

In this section we first prove that expansion implies uniform S-hyperbolicity for certain nice actions of hyperbolic groups (Theorem~\ref{thm:e=>h}), and then explore to which extent S-hyperbolic actions of hyperbolic groups arise from their actions on Gromov boundaries.

\subsection{Expansion implies S-hyperbolicity} 

In view of Corollary~\ref{cor:qg}, any two $(\D,\de;N)$-equivalent rays in a \emph{hyperbolic group} are $(\D,\de;N')$-fellow-traveling for some $N'\ge1$, since two quasigeodesics in a hyperbolic space $X$ which are Hausdorff-close on unbounded subsets define the same point in $\pa X$. Thus, the meandering hyperbolicity condition enables us to define the following map when $\Ga$ is a hyperbolic group:

\begin{definition}[Postal map]\label{def:pi}
Let $\Ga$ be a hyperbolic group and $\rho:\Ga\to\Homeo(M)$ meandering-hyperbolic at $\La$ with a datum $(\D_0\prec\D;N)$. Then we define the \emph{postal map}
\[
\pi:\La\to\pa\Ga,\; x\mapsto \pi(x)
\]
of $\rho$ as follows: the value $\pi(x)$ of $x\in\La$ is the equivalence class in $\pa\Ga$ (in the sense of Section~\ref{sec:coarse_geometry}) of a ray $c^\al\in\Ray_x(\D_0,\de_0)$. The equivalence class is well-defined since any two $(\D_0,\de_0)$-rays for $x\in \La$ are $(\D,\de;N')$-fellow-traveling as mentioned above.
\end{definition}

The map $\pi$ is clearly equivariant. In Theorem~\ref{thm:pi} we will prove that $\pi$ is a continuous surjective map. We note that, as we show later in Corollary~\ref{cor:hyp}, meandering-hyperbolic actions of a \emph{non-elementary} hyperbolic group are in fact (uniformly) S-hyperbolic.

\begin{theorem}\label{thm:e=>h}
Let $\Ga$ be a non-elementary hyperbolic group. Suppose that an action $\rho:\Ga\to\Homeo(M)$ is expanding at $\La$ and there exists an equivariant continuous nowhere constant map $f:\La\to\pa\Ga$. Then $\rho$ is uniformly S-hyperbolic (and thus meandering-hyperbolic) at $\La$ and $f$ is the postal map of $\rho$.
\end{theorem}

\begin{proof}
Let $\D=(\vS,\de)$ be the expansion datum of $\rho$ (see Definition~\ref{def:expansion}). Recall that $d_\vS$ denotes the word metric on $\Ga$ with respect to $\vS$.

We claim that if $x\in\La$ and $\eta\in(0,\de]$ then, for every $\eta$-code $\al$ for $x$, the $\eta$-ray $c^\al\in\Ray^\eta_\D(x)$ is a uniform quasigeodesic ray in $(\Ga,d_\vS)$ asymptotic to $f(x)$. To see this, we first note that by Corollary~\ref{cor:qg} the ray $c^\al$ is indeed an $(A,C)$-quasigeodesic ray with $A$ and $C$ independent of $x$ and $\eta$.

By Theorem~\ref{thm:minimal}(1), there is a subsequence $(g_j)$ of $(c^\al_k)$ such that $(\rho(g_j))$ converges to $x$ on some ball $B_{\eta/2}(q)\subset\La$. Since $f$ is nowhere constant, the image $S:=f(B_{\eta/2}(q))\subset\pa\Ga$ is not a singleton. By the equivariance of $f$, the subsequence $(g_j)$ converges to $\xi:=f(x)$ pointwise on $S$. Moreover, the initial point $c^\al_0=s_{\al(0)}\in\vS$ is a generator of $\Ga$. Therefore, Lemma~\ref{lem:conv} applies to the ray $c^\al$ and we conclude that the image $c^\al(\N_0)$ is $D$-Hausdorff close to a geodesic ray $\id\xi$ in $(\Ga,d_\vS)$, where the constant $D$ depends only on the hyperbolicity constant of $(\Ga,d_\vS)$ and the quasi-isometry constants $(A,C)$.

Now, suppose that $c^\al,c^\be\in\Ray^\eta_\D(x)$ are rays associated to $\eta$-codes $\al,\be$ for $x\in\La$, respectively. Then the images of $c^\al,c^\be$ are within Hausdorff distance $D$ from a geodesic ray $e\xi$, hence, these images are $2D$-Hausdorff close. Therefore, the rays $2D$-fellow-travel each other. Since $x\in\La$ is arbitrary, we conclude that the action is uniformly S-hyperbolic with data $(\D;N_\eta)=(\de;2D)$; recall Definitions~\ref{def:N-equiv} and \ref{def:shyp}.
\end{proof}

\begin{corollary}
Let $\Ga$ be a non-elementary hyperbolic group. Suppose that the action of $\Ga$ on its Gromov boundary $\pa\Ga$ is expanding with respect to a metric $d_\infty$ compatible with the topology. Then this action is uniformly S-hyperbolic.
\end{corollary}

\begin{proof}
We set $M=\La=\pa\Ga$ and $f=\mathrm{id}:\pa\Ga\to\pa\Ga$ in the preceding theorem.
\end{proof}

Let $\Ga$ be a non-elementary hyperbolic group and $d_a$ a visual metric on $\pa\Ga$ (Definition~\ref{vmetric}). Coornaert \cite{Coo}*{Proposition 3.1, Lemma 6.2} showed that the $\Ga$-action on $(\pa\Ga, d_a)$ is 
expanding. Thus, we have:

\begin{corollary}\label{cor:hyp-hyp}
Let $\Ga$ be a non-elementary hyperbolic group with the Gromov boundary $\pa\Ga$ equipped with a visual metric $d_a$. Then the action of $\Ga$ on $(\pa\Ga,d_a)$ is uniformly S-hyperbolic.
\end{corollary}

\subsection{Meandering-hyperbolic actions of word-hyperbolic groups}

It is natural to ask to what extent the converse of Corollary~\ref{cor:hyp-hyp} is true:

\begin{question}\label{ques:C}
Does every meandering-hyperbolic action $\Ga\to\Homeo(M)$ at $\La$ come from the action of a hyperbolic group on its Gromov boundary?
\end{question}

Assume first that $\card(\La)\ge3$ and the action of $\Ga$ on $\La$ in Question~\ref{ques:C} is a convergence action (see Section~\ref{sec:dynamics} for definition). Then, in view of the expansion condition, it is also a uniform convergence action; see \cite{KLP17}*{Lemma 3.13} or \cite{KL18a}*{Theorem 8.8} for a different argument. If we assume, in addition, that $\La$ is perfect (or that $\La$ is the limit set of the action of $\Ga$ on $\La$, see Theorem~\ref{thm:ne-con-act}), then $\Ga$ is hyperbolic and $\La$ is equivariantly homeomorphic to the Gromov boundary $\pa\Ga$ (Theorem~\ref{thm:bowditch}). To summarize:

\begin{proposition}
Suppose an expanding action $\Ga\to\Homeo(M)$ at $\La$ is a convergence action with limit set $\La$ satisfying $\card(\La)\ge 3$. Then $\Ga$ is hyperbolic and $\La$ is equivariantly homeomorphic to $\pa\Ga$.
\end{proposition}
\noindent Note that we do not even need to assume faithfulness of the action of $\Ga$ on $\La$ since, by the convergence action assumption, such an action necessarily has finite kernel.

As another application of the formalism of convergence group actions we obtain:

\begin{proposition}
Suppose that $\Ga$ is hyperbolic, $d_\infty$ is a compatible metric on the Gromov boundary $\pa\Ga$, and the $\Ga$-action on $(\pa\Ga,d_\infty)$ is expanding. Define the subset $\vS_0\subset\vS$ of the finite generating set $\vS$, consisting of elements $s_\al$ with non-empty expansion subsets $U_\al\subset\La$. Then $\vS_0$ generates a finite index subgroup $\Ga_0<\Ga$.
\end{proposition}

\begin{proof} 
The action of $\Ga_0$ on $\pa\Ga$ is still expanding and convergence, see Remark \ref{rem:consubgroups}. Therefore, as noted above, the action of $\Ga_0$ on $T(\pa\Ga)$ is also cocompact. Since the action of $\Ga$ on $T(\pa\Ga)$ is properly discontinuous, it follows that $\Ga_0$ has finite index in $\Ga$.
\end{proof}

\medskip 
Next, assuming hyperbolicity of $\Ga$ in Question~\ref{ques:C}, we can relate $\La$ and $\pa\Ga$. Recall the relevant definitions from the beginning of Section~\ref{sec:prelim}.

\begin{theorem}\label{thm:pi}
Let $\Ga$ be a non-elementary hyperbolic group. If $\Ga\to\Homeo(M)$ is a meandering-hyperbolic action at $\La$, then the following hold.
\begin{enumerate}[label=\textup{(\arabic*)},leftmargin=*,nosep]
\item
The postal map $\pi:\La\to\pa\Ga$ is an equivariant continuous surjective map.
\item
For each minimal non-empty closed $\Ga$-invariant subset $\La_\mu\subset \La$, the restriction $\pi_\mu:\La_\mu\to\pa\Ga$ of $\pi$ to $\La_\mu$ is a surjective quasi-open map.
\item
Every $\La_\mu$ as above is perfect.
\end{enumerate} 
\end{theorem}

\begin{proof} 
(1) We already noted the map $\pi$ is equivariant. Continuity of $\pi:\La\to\pa\Ga$ can be seen as in Section~\ref{sec:cont}. Namely, if $x,y\in\La$ are close, there exist $\de$-codes $\al$ and $\be$ for $x$ and $y$, respectively, such that $\al(k)=\be(k)$ for all $0\le k\le n$, where $n\in\N$ is sufficiently large. Then
\[
c^\al_k=s_{\al(0)}s_{\al(1)}\cdots s_{\al(k)}
=s_{\be(0)}s_{\be(1)}\cdots s_{\be(k)}=c^\be_k
\]
for all $0\le k\le n$. This means $\pi(y)\in V_n(\pi(x))$ for a sufficiently large $n$ (see Section~\ref{sec:coarse_geometry}), hence $\pi(x)$ and $\pi(y)$ are close.

Since $\La$ is compact, the image $\pi(\La)$ is closed and $\Ga$-invariant. By the minimality of the action of $\Ga$ on $\pa\Ga$, we have $\pi(\La)=\pa\Ga$.

(2) Surjectivity of $\pi_\mu$ follows from the minimality of the action of $\Ga$ on $\pa\Ga$ as in part (1). We now prove that each $\pi_\mu$ is quasi-open. Since $\La_\mu$ is compact, it is locally compact, hence it suffices to prove that for every compact subset $K\subset\La_\mu$ with non-empty interior, the image $\pi(K)\subset\pa\Ga$ also has non-empty interior.

In view of the minimality of the $\Ga$-action on $\La_\mu$ and compactness of $\La_\mu$, there exists a finite subset $\{g_1,\ldots,g_n\}\subset \Ga$ such that
\[
\rho(g_1)(\intr K)\cup \cdots \cup \rho(g_n)(\intr K)= \La_\mu. 
\]
By the equivariance of $\pi$ and surjectivity of $\pi_\mu: \La_\mu\to\pa\Ga$, we also have
\[
g_1(\pi(K))\cup \cdots \cup g_n(\pi(K))= \pa\Ga. 
\] 
Since a finite collection (even a countable collection) of nowhere dense subsets cannot cover $\pa\Ga$, it follows that $\pi(K)$ has non-empty interior.

(3) Suppose that $\La_\mu$ has an isolated point $z$. Since the action of $\Ga$ on $\La_\mu$ is minimal, the compact subset $\La_\mu\subset \La$ consists entirely of isolated points, i.e. is finite. Therefore, $\pi(\La_\mu)\subset \pa\Ga$ is a finite non-empty $\Ga$-invariant subset. This contradicts the minimality of the action of $\Ga$ on $\pa\Ga$.
\end{proof}

\begin{remark} 
For some minimal S-hyperbolic actions of hyperbolic groups $\Ga\to\Homeo(\La)$, the map $\pi$ is not open; see Example~\ref{ex:blowup}.
\end{remark}

\begin{corollary}\label{cor:pi}
Let $\Ga$ be a non-elementary hyperbolic group. If $\Ga\to\Homeo(M)$ is a meandering-hyperbolic action at $\La$, then:
\begin{enumerate}[label=\textup{(\arabic*)},leftmargin=*,nosep]
\item
$\Ga$ acts on $\La$ with finite kernel $K$.
\item
If $(g_i)$ is a sequence in $\Ga$ converging to the identity on $\La$ then the projection of this sequence to $\Ga/K$ is eventually equal to $\id\in\Ga/K$.
\end{enumerate}
\end{corollary}

\begin{proof}
Both statements are immediate consequences of Theorem~\ref{thm:pi}(1) and the convergence property for the action of a $\Ga$ on $\pa\Ga$; see Section~\ref{sec:coarse_geometry}.
\end{proof}

As another immediate corollary of the theorem and Theorem~\ref{thm:e=>h} we obtain:

\begin{corollary}\label{cor:hyp}
Let $\Ga$ be a non-elementary hyperbolic group. Then every   meandering-hyperbolic action of $\Ga$ is in fact uniformly S-hyperbolic.
\end{corollary}

These are positive results regarding Question~\ref{ques:C}. In Sections~\ref{sec:ulattice} and \ref{sec:examples}, however, we present several examples which show that in general the question has negative answer.

\subsection{S-hyperbolicity and stability for Anosov subgroups}

Our goal in this section is to characterize Anosov subgroups in terms of the expansion condition on suitable subsets of partial flag manifolds (Theorem~\ref{thm:equiv}) and show that the corresponding actions are S-hyperbolic. As an application, we give an alternative proof of the stability of Anosov subgroups (Corollary~\ref{cor:stability}). 

For the sake of simplicity, we shall restrict our attention to the case of non-elementary hyperbolic groups and make use of Corollary~\ref{cor:qg} and Theorems~\ref{thm:e=>h} and \ref{thm:pi}. Then Lemma~\ref{lem:c} below says that the condition (d) in Definition~\ref{def:anosov} is also equivalent to the expansion condition (Definition~\ref{def:expansion}) at the image of the boundary embedding. Consequently, we obtain the following characterization of Anosov subgroups:

\begin{theorem}\label{thm:equiv}
For a non-elementary hyperbolic subgroup $\Ga<G$ the following are equivalent.
\begin{enumerate}[label=\textup{(\arabic*)},nosep,leftmargin=*]
\item
$\Ga$ is non-uniformly $\tad$-Anosov with asymptotic embedding $\psi:\pa\Ga\to\La_\Ga(\tad)$.
\item
$\Ga$ is $\tad$-boundary embedded with a boundary embedding $\varphi:\pa\Ga\to\Flag(\tad)$ and the $\Ga$-action on $\Flag(\tad)$ is expanding at $\varphi(\pa\Ga)$.
\item
There exists a closed $\Ga$-invariant antipodal subset $\La\subset\Flag(\tad)$ such that the action $\Ga\to \Homeo(\Flag(\tad))$ is S-hyperbolic at $\La$ with injective postal map $\pi:\La\to\pa\Ga$.
\end{enumerate}
Moreover, the maps $\psi$ and $\varphi$ in \textup{(1)} and \textup{(2)} coincide, and the map $\pi$ in \textup{(3)} equals $\psi^{-1}$.
\end{theorem}

\begin{proof}
(2 $\Rightarrow$ 3)\; Let $\La=\varphi(\pa\Ga)\subset\Flag(\tad)$. It is a closed $\Ga$-invariant antipodal subset such that the action $\Ga\to\Homeo(\Flag(\tad);\La)$ is expanding and the equivariant homeomorphism $\varphi^{-1}:\La\to\pa\Ga$ is nowhere constant. By Theorem~\ref{thm:e=>h}, this action is S-hyperbolic with postal map $\varphi^{-1}$.

(3 $\Rightarrow$ 2)\; By Theorem~\ref{thm:pi}(1) the postal map $\pi$ is equivariant, continuous and surjective. Since $\pi$ is assumed to be injective and $\La$ is compact, $\pi$ is in fact a homeomorphism. Since $\La$ is antipodal, the inverse $\pi^{-1}:\pa\Ga\to\La\subset\Flag(\tad)$ is a boundary embedding. The $\Ga$-action on $\La$ is S-hyperbolic, in particular, expanding.

(1~$\Leftrightarrow$~2) This equivalence reduces to the lemma below.
\end{proof}

\begin{lemma}\label{lem:c}
Suppose $\Ga<G$ is $\tad$-boundary embedded with a boundary embedding $\varphi:\pa\Ga\to\Flag(\tad)$.
\begin{enumerate}[label=\textup{(\arabic*)},nosep,leftmargin=*]
\item
If $\Ga$ is non-uniformly $\tad$-Anosov with the asymptotic embedding $\varphi$, then the $\Ga$-action on $\Flag(\tad)$ is expanding at $\varphi(\pa\Ga)$.
\item
If $\Ga$ is non-elementary and the $\Ga$-action on $\Flag(\tad)$ is expanding at $\varphi(\pa\Ga)$, then it is non-uniformly $\tad$-Anosov and $\varphi$ is the asymptotic embedding for $\Ga$.
\end{enumerate}
\end{lemma}

\begin{proof}
(1) This is a special case of \cite{KLP17}*{Equivalence Theorem 1.1}: every $\tad$-Anosov subgroup $\Ga<G$ is expanding at $\varphi(\pa\Ga)$. To see this directly, rather, note from the condition (d) in Definition~\ref{def:anosov} that for each $\xi\in \pa\Ga$ there is an element $g\in \Ga$ such that $\ef(g,\psi(\xi))>1$. Then Remark~\ref{rem:expansion}(c) implies the $\Ga$-action on $\Flag(\tad)$ is expanding at $\varphi(\pa\Ga)$.

(2) Suppose the $\Ga$-action on $\Flag(\tad)$ is expanding at $\varphi(\pa\Ga)$ with data $\D=(\de)$. Since $\pa\Ga$ is perfect, Corollary~\ref{cor:qg} applies. Thus, for any $\eta\in(0,\de]$, the rays $c^\al$ associated to $\eta$-codes $\al$ for $\varphi(\xi)\in\varphi(\pa\Ga)$ are uniform quasigeodesic rays in $\Ga$.

Let $\xi\in\pa\Ga$ and let $r:\N_0\to\Ga$ be a geodesic ray starting at $\id\in\Ga$ and asymptotic to $\xi$. If $\al$ is an $\eta$-code for $\varphi(\xi)$ then, as in the proof of Lemma~\ref{lem:conv}, the Hausdorff distance between $\{r(k)\}_{k\in \N_0}$ and $\{c^\al_j\}_{j\in \N_0}$ is bounded above by a uniform constant $C>0$. This means that for each $k\in \N_0$, there exist $n_k\in\N_0$ and an element $g_k\in\Ga$ with $|g_k|_\vS\le C$ such that $r(k)=c^\al_{n_k} g_k$. Then we have
\[
\ef(r(k)^{-1},\varphi(\xi))
=\ef(g_k^{-1}(c^\al_{n_k})^{-1},\varphi(\xi))
\ge A \cdot \ef((c^\al_{n_k})^{-1},\varphi(\xi)),
\]
where $A=\inf\{\ef(g,\varphi(\zeta)) \mid \zeta\in\pa\Ga,\ g\in\Ga\textup{ and }|g|_\vS\le C\}$.
Since $\ef( (c^\al_{j})^{-1},\varphi(\xi))$ tends to infinity as $j$ tends to infinity (by the last statement of Lemma~\ref{lem:nest}), it follows that
\[
\sup_{k\in \mathbb N} \ef(r(k)^{-1},\varphi(\xi))=+\infty.
\]
Therefore, the $\Ga$-action satisfies the condition (d) of Definition~\ref{def:anosov}.
\end{proof}

A corollary of this theorem is the stability of Anosov subgroups; see, for example, \cite{GW}*{Theorem 1.2} and \cite{KLP14}*{Theorem 1.10}. Let us denote by
\[
\Hom^\tau(\Ga,G)
\]
the space of faithful representations $\Ga\to G$ with (non-uniformly) $\tad$-Anosov images.

\begin{corollary}\label{cor:stability}
Suppose that $\Ga$ is a non-elementary hyperbolic group. Then
\begin{enumerate}[label=\textup{(\arabic*)},nosep,leftmargin=*]
\item
$\Hom^\tau(\Ga,G)$ is open in $\Hom(\Ga,G)$.
\item
For any sequence of representations $\rho_i\in \Hom^\tau(\Ga,G)$ converging to $\rho\in \Hom^\tau(\Ga,G)$, the asymptotic embeddings $\psi_i:\pa\Ga\to\La_{\rho_i(\Ga)}(\tad)$ converge uniformly to the asymptotic embedding $\psi:\pa\Ga\to\La_{\rho(\Ga)}(\tad)$.
\end{enumerate}
\end{corollary}

\begin{proof}
We start with an embedding $\rho\in\Hom^\tau(\Ga,G)$; let $\La:=\La_{\rho(\Ga)}(\tad)$ and
\[
\psi:\pa\Ga\to\La\subset\Flag(\tad)
\]
denote the asymptotic embedding of $\rho$. By Theorem~\ref{thm:equiv}, the $\Ga$-action on $\La$ is S-hyperbolic. By Theorem~\ref{thm:main} there exists a small neighborhood $U'$ of $\rho$ in $\Hom(\Ga,G)$ such that, for each $\rho'\in U'$, there exists a $\rho'$-invariant compact $\La'\subset\Flag(\tad)$ at which the $\rho'$-action is expanding and there is an equivariant homeomorphism $\phi:\La\to\La'$.

By Corollary~\ref{cor:alg-stab}, for every $\rho'\in U'$, the kernel of the action of $\Ga'= \rho'(\Ga)$ on $\La'$ equals the kernel of the action of $\Ga$ on $\La$. Since $\Ga$ is assumed to be non-elementary, it acts on $\La$ with finite kernel $\Phi$ (Corollary~\ref{cor:pi}). Therefore, the kernel of $\rho'$ is contained in the finite subgroup $\Phi< \Ga$. As explained in \cite{KLP14}*{proof of Corollary 7.34}, rigidity of finite subgroups of Lie groups implies that $U'$ contains a smaller neighborhood $U$ of $\rho$ such that every $\rho'\in U$ is injective on $\Phi$. Therefore, every $\rho'\in U$ is faithful.

Since $\La'$ depends continuously on $\rho'$ (see Section~\ref{sec:cont'}), the antipodality of $\La$ leads to the antipodality of $\La'$. (In order to guarantee this, one may further reduce the size of $U$ if necessary.) Thus $\phi\circ\psi:\pa\Ga\to\La'$ is a boundary embedding of $\Ga'$. From Lemma~\ref{lem:c}(2) we conclude that $\Ga'<G$ is again (non-uniformly) $\tad$-Anosov and the boundary embedding $\phi\circ\psi$ of $\Ga'$ is uniformly close to $\psi$.
\end{proof}

\subsection{Historical remarks on stability}

The history of stability for convex-cocompact (and, more generally, geometrically finite) Kleinian groups goes back to the pioneering work of Marden \cite{Marden} (in the case of subgroups of $\PSL(2,\C)$). It appears that the first proof of stability of geometrically finite subgroups of $\Isom(\H^n)$ (the isometry group of the hyperbolic $n$-space) was given by Bowditch in \cite{Bo}, although many arguments are already contained in \cite{CEG}. Bowditch in his paper also credits this result to P.~Tukia. A generalization of the Sullivan's stability theorem for subgroups of $\Isom({\mathbb H}^n)$, proving the existence of a quasiconformal conjugation on the entire sphere at infinity, was given by Izeki \cite{Izeki}.

After Sullivan's paper \cite{Sul}, a proof of stability for convex-cocompact subgroups of rank one semisimple Lie groups was given by Corlette \cite{Corlette}. Corlette's proof also goes through in the setting of $C^1$-stability as it was observed by Yue \cite{Yue}. Unlike the proofs of Sullivan and Bowditch, however, Corlette's proof is based on an application of Anosov flows; the same tool is used in the subsequent proofs of stability of \emph{Anosov subgroups} of higher rank semisimple Lie groups.

The notion of Anosov subgroups was introduced by Labourie \cite{Lab} as a natural analogue of convex-cocompact groups, and the theory was further developed by Guichard and Wienhard \cite{GW} to include all word-hyperbolic groups.
Subsequently, Kapovich, Leeb and Porti \cite{KLP17} provided new characterizations of Anosov subgroups investigating their properties from many different perspectives.
In particular, they gave an alternative proof of stability of Anosov subgroups in \cite{KLP14} using coarse-geometric ideas. Recently, Bochi, Potrie and Sambarino \cite{BPS} provided a purely dynamical proof of the structural stability of Anosov representations.

On the other hand, a notion of convex-cocompactness for discrete subgroups of $\PGL(n,\R)$ was studied by Danciger, Gu\'eritaud and Kassel in \cite{DGK}*{Definition 1.11}. While every Anosov subgroup is word-hyperbolic, convex-cocompact subgroups of $\PGL(n,\R)$ in their definition are not necessarily word-hyperbolic. Nevertheless, stability in the sense of linear deformations was established for such groups in \cite{DGK}*{Theorem 1.17(D)}.

\section{Other examples}\label{sec:examples} 

We present a number of examples and non-examples of S-hyperbolic actions.

\subsection{Expanding but not S-hyperbolic actions}\label{sec:e/=>h}

In general, even for hyperbolic groups, the expansion condition alone does not imply the S-hyperbolicity condition.

\begin{example}[Action with infinite kernel]
Suppose that $\Ga'$ and $\Ga$ are non-elementary hyperbolic groups and $\phi:\Ga'\to\Ga$ is an epimorphism with infinite kernel. We equip the Gromov boundary $\La=\pa\Ga$ with a visual metric. Then the action $\rho:\Ga\to\Homeo(\La)$ is expanding (compare Corollary~\ref{cor:hyp-hyp}). Thus, the associated $\Ga'$-action $\rho\circ\phi:\Ga'\to\Homeo(\La)$ is expanding as well. But this action cannot be S-hyperbolic: see Corollary~\ref{cor:pi}.
\end{example}

\begin{example}[Non-discrete action]
Suppose that $\Ga$ is a non-elementary hyperbolic group and $\rho:\Ga\to G=\Isom(\H^n)$ is a representation with dense image. Then the associated action $\rho:\Ga\to\Homeo(\La)$ on the visual boundary $\La=\pa\H^n$ is expanding due to the density of $\rho(\Ga)$ in $G$. But the action $\rho:\Ga\to\Homeo(\La)$ cannot be S-hyperbolic. Indeed, by the density of $\rho(\Ga)$ in $G$, there is a sequence of distinct elements $g_i\in\Ga$ such that $\rho(g_i)$ converges to the identity element of $G$. 
If $\rho$ were S-hyperbolic, then Corollary~\ref{cor:pi} would give a contradiction that the set $\{g_i\}$ is finite.

More generally, the same argument works for representations $\rho:\Ga\to G$ of non-elementary hyperbolic groups to a semisimple Lie group $G$ with dense images $\rho(\Ga)$. The associated actions $\rho$ of $\Ga$ on the partial flag manifolds $G/P$ are expanding but not S-hyperbolic.
\end{example}

\subsection{S-hyperbolic actions of hyperbolic groups}\label{sec:ex1}

In addition to the toy examples in Section~\ref{sec:toy} we now give more interesting examples of S-hyperbolic actions of hyperbolic groups. As we proceed, the associated postal maps $\pi:\La\to\pa\Ga$ will be increasingly more complicated.

Recall that a discrete subgroup $\Ga<\Isom(\H^n)$ is \emph{convex-cocompact} if its limit set $\La=\La_\Ga\subset S^{n-1}=\pa\H^n$ is not a singleton and $\Ga$ acts cocompactly on the closed convex hull of $\La$ in $\H^n$. We refer to \cite{Bowditch-gf} for details on convex-cocompact and, more generally, geometrically finite isometry groups of hyperbolic spaces.

Every convex-cocompact (discrete) subgroup $\Ga$ of isometries of $\H^n$ (and, more generally, a rank one symmetric space) is S-hyperbolic. More precisely, for $\La=\La_\Ga$, the action $\Ga\to\Homeo(S^{n-1})$ is S-hyperbolic. This can be either proven directly using a \emph{Ford fundamental domain} (as in \cite{Sul}*{Theorem I} by considering the conformal ball model of $\H^n$ inside $\R^n$) or regarded as a special case of S-hyperbolicity of Anosov subgroups (Theorem~\ref{thm:equiv}).

Unlike the convex-cocompact or Anosov examples, the invariant compact set $\La$ is not equivariantly homeomorphic to the Gromov boundary in the examples below.

\begin{example}[$k$-fold non-trivial covering]\label{ex:E0} 
Let $S_g$ be a closed oriented hyperbolic surface of genus $g\ge 2$; it is isometric to the quotient $\H^2/\Ga$, where $\Ga\cong \pi_1(S_g)$ is a discrete subgroup of $\PSL(2,\R)$. Take any $k\ge 2$ dividing $2g-2$. Since the Euler number of the unit circle bundle of $S_g$ is $2-2g$, same as the Euler number of the action of $\Ga$ on $S^1=\pa\H^2$, it follows that the action of $\Ga$ lifts to a smooth action
\[
\wt\rho:\Ga\to\Diff^1(S^1) 
\]
with respect to the degree $k$ covering $p:\La= S^1\to S^1$. We pull-back the Riemannian metric from the range to the domain $\La$ via the map $p$. Since $\Ga<\PSL(2,\R)$ is convex-cocompact, its action on $\pa \H^2$ is S-hyperbolic. Let $\{U_\al \mid i\in\I\}$ be a collection of expanding subsets (arcs) in $S^1=\pa \H^2$ corresponding to a generating set $\vS=\{s_i \mid i \in \I\}$ of $\Ga$. As in Example~\ref{ex:E}, we lift these arcs to connected components
\[
\{\wt{U}_{j_i}\subset p^{-1}(U_j) \mid j\in\I,\; i=1,\ldots,k\}. 
\]
These will be expanding subsets for the generators $s_{j_i}=s_j$ ($j\in\I$, $i=1,\ldots,k$) of $\Ga$. The action $\wt\rho$ will be minimal and S-hyperbolic (because the original action of $\Ga$ is).

Thus, we obtain an example of a minimal S-hyperbolic action of a hyperbolic group on a set $\La$ which is not equivariantly homeomorphic to $\pa\Ga$.
\end{example} 

Similar examples of S-hyperbolic actions on $S^1$ can be obtained by starting with a general convex-cocompact Fuchsian subgroup $\Ga_0< \PSL(2,\R)$ and lifting it to an action of a finite central extension $\Ga$ of $\Ga_0$ on $S^1$ via a finite covering $S^1\to S^1$. We refer to the paper \cite{Deroin} of Deroin for details as well as the related classification of expanding real-analytic group actions on the circle.

Below is a variation of the above construction.

\begin{example}[Trivial covering]\label{ex:E1}
Let $\Ga_0$ be a non-elementary hyperbolic group with the Gromov boundary $\La_0=\pa\Ga_0$ equipped with a visual metric. Let $\La=\La_0\times \{0,1\}$ and $\Ga_0\to\Homeo(\La)$ be the product action where $\Ga_0$ acts trivially on $\{0, 1\}$. This action is S-hyperbolic but, obviously, non-minimal.

We extend this action of $\Ga_0$ to an action of $\Ga=\Ga_0\times\Z_2$, where the generator of $\Z_2$ acts by the map
\[
(\xi, i)\mapsto (\xi, 1-i)\quad(\textup{for }\xi\in\La_0\textup{ and }i=0,1)
\]
with an empty expansion subset. The action $\Ga \to \Homeo(\La)$ is easily seen to be faithful, S-hyperbolic and minimal. It is clear, however, that $\La$ is not equivariantly homeomorphic to $\pa\Ga\cong \pa\Ga_0$.
\end{example}

In the preceding examples we had an equivariant finite covering map $\pi:\La\to\pa\Ga$. In the next example the postal map $\pi:\La\to\pa\Ga$ is a finite-to-one open map but not a local homeomorphism; one can regard the map $\pi$ as a \emph{generalized branched covering} (in the sense that it is an open finite-to-one map which is a covering map away from a codimension $2$ subset).

\begin{example}[Generalized branched covering]\label{ex:branch}

Let $\Ga<\PSL(2,\R)$ be a \emph{Schottky subgroup}, that is, a convex-cocompact non-elementary free subgroup. Its limit set $\La_\Ga\subset S^1$ is homeomorphic to the Cantor set; it is also equivariantly homeomorphic to the Gromov boundary $\pa\Ga$.

We regard $\Ga$ as a subgroup of $\PSL(2,\C)$ via the standard embedding $\PSL(2,\R)\to\PSL(2,\C)$. The domain of discontinuity of the action of $\Ga$ on $S^2$ is $\Omega_{\Ga}= S^2-\La_\Ga$; the quotient surface $S=\Omega_\Ga/\Ga$ is compact and its genus equals to the rank $r$ of $\Ga$. We let $\chi:\pi_1(S)\to F$ be a homomorphism to a finite group $F$ which is non-trivial on the image of $\pi_1(\Omega_\Ga)$ in $\pi_1(S)$. For concreteness, we take the following homomorphism $\chi: \pi_1(S)\to F=\Z_2$. We let $\{a_1,b_1,\ldots,a_r,b_r\}$ denote a generating set of $\pi_1(S)$ such that $a_1,\ldots,a_r$ lie in the kernel of the natural homomorphism $\phi:\pi_1(S)\to\Ga$, while $\phi$ sends $b_1,\ldots,b_r$ to (free) generators of $\Ga$. Then take $\chi$ such that $\chi(a_1)=1\in \Z_2$, while $\chi$ sends the rest of the generators to $0\in \Z_2$. This homomorphism to $F$, therefore, lifts to an epimorphism $\wt\chi:\pi_1(\Omega_\Ga)\to F$ with $\Ga$-invariant kernel $K< \pi_1(\Omega_\Ga)$. Hence, there exists a non-trivial 2-fold covering
\[
p:\wt\Omega\to\Omega_\Ga
\]
associated to $K$ and the action of the group $\Ga$ on $\Omega_\Ga$ lifts to an action of $\Ga$ on $\wt\Omega$. One verifies that $p$ is a proper map which induces a surjective but not injective map $p_\infty: \End(\wt\Omega)\to\End(\Omega_\Ga)$ between the spaces of ends of the surfaces $\wt\Omega$ and $\Omega_\Ga$. Since $p$ is a 2-fold covering map, the induced map $p_\infty$ is at most 2-to-1 (that is, the fibers of $p_\infty$ have cardinality $\le 2$).

We let $ds^2$ denote the restriction of the standard Riemannian metric on $S^2$ to the domain $\Omega_\Ga$ and let $\wt{ds^2}$ denote the pull-back of $ds^2$ to $\wt\Omega$. The Riemannian metric $ds^2$ is, of course, incomplete; the Cauchy completion of the associated Riemannian distance function $d_{\Omega_\Ga}$ on $\Omega_\Ga$ is naturally homeomorphic to $S^2$ (which is also the end-compactification of $\Omega_\Ga$), as a sequence in $\Omega_\Ga$ is Cauchy with respect to the metric $d_{\Omega_\Ga}$ if and only if it converges in $S^2$. (Here we are using the assumption that $\La_\Ga$ is contained in the circle $S^1$.)

We therefore let $(M,d)$ denote the Cauchy completion of the Riemannian distance function of $(\wt\Omega, \wt{ds^2})$. One verifies that $M$ is compact and is naturally homeomorphic to the end-compactification of $\wt\Omega$. In particular, the covering map $p:\wt\Omega\to\Omega_\Ga$ extends to a continuous open finite-to-one map
\[
p: M\to S^2
\]
sending $\La:= M- \wt\Omega$ to $\La_\Ga$. The map $p: M\to S^2$ is locally one-to-one on $\wt\Omega$ but fails to be a local homeomorphism at $\La$. Since every element of $\Ga$ acts as a Lipschitz map to $(\wt\Omega, \wt{ds^2})$, the action of $\Ga$ on $(\wt\Omega, \wt{ds^2})$ extends to an action of $\Ga$ on $M$ so that every element of $\Ga$ is a Lipschitz map and $\La$ is a $\Ga$-invariant compact subset of $M$. The map $p: M\to S^2$ is equivariant with respect to the actions of $\Ga$ on $M$ and on $S^2$. Similarly to our covering maps examples, the action $\Ga\to\Homeo(M)$ is S-hyperbolic at $\La$. The postal map $\pi:\La\to\La_\Ga=\pa\Ga$ equals the restriction of $p$ to $\La$ and, hence, is not a local homeomorphism.
\end{example}

\medskip 
In the next example the postal map $\pi:\La\to\pa\Ga$ is not even an open map.

\begin{example}[Denjoy blow-up]\label{ex:blowup} 
We let $\Ga$ be the fundamental group of a closed hyperbolic surface $M^2$. Let $c\subset M^2$ be a simple closed geodesic representing the conjugacy class $[\ga]$ in $\Ga$. The Gromov boundary of $\Ga$ is the circle $S^1$. We perform a \emph{blow-up} of $S^1$ at the set $\Phi\subset S^1$ of fixed points of the elements in the conjugacy class $[\ga]$, replacing every fixed point by a pair of points. See Figure~\ref{fig:Denjoy}. The resulting topological space $\La$ is homeomorphic to the Cantor set; the quotient map
\[
q:\La\to S^1
\]
is 1-to-1 over $S^1- \Phi$ and is 2-to-1 over $\Phi$. The map $q$ is quasi-open (with $O_q=\La - q^{-1}(\Phi)$) but not open. (This map is an analogue of the Cantor function $f:C\to[0,1]$ mentioned in the beginning of Section~\ref{sec:prelim}.) The action of $\Ga$ on $S^1$ lifts to a continuous action of $\Ga$ on $\La$ with every $g\in [\ga]$ fixing all the points of the preimage of the fixed-point set of $g$ in $S^1$. In particular, the action of $\Ga$ on $\La$ is minimal. One can metrize $\La$ so that the action $\Ga\to \Homeo(\La)$ is S-hyperbolic with the postal map $\pi:\La\to\pa\Ga$ being equal to the quotient map $q:\La\to S^1$.

\begin{figure}[ht]
\labellist
\pinlabel {$q$} at 460 230
\pinlabel {$\La$} at 400 204
\pinlabel {$S^1$} at 526 204
\pinlabel {$g\in [\ga]$} at 752 204
\pinlabel {$g_+$} at 708 376
\pinlabel {$q^{-1}(g_+)$} at 200 270
\endlabellist
\centering
\includegraphics[width=0.8\textwidth]{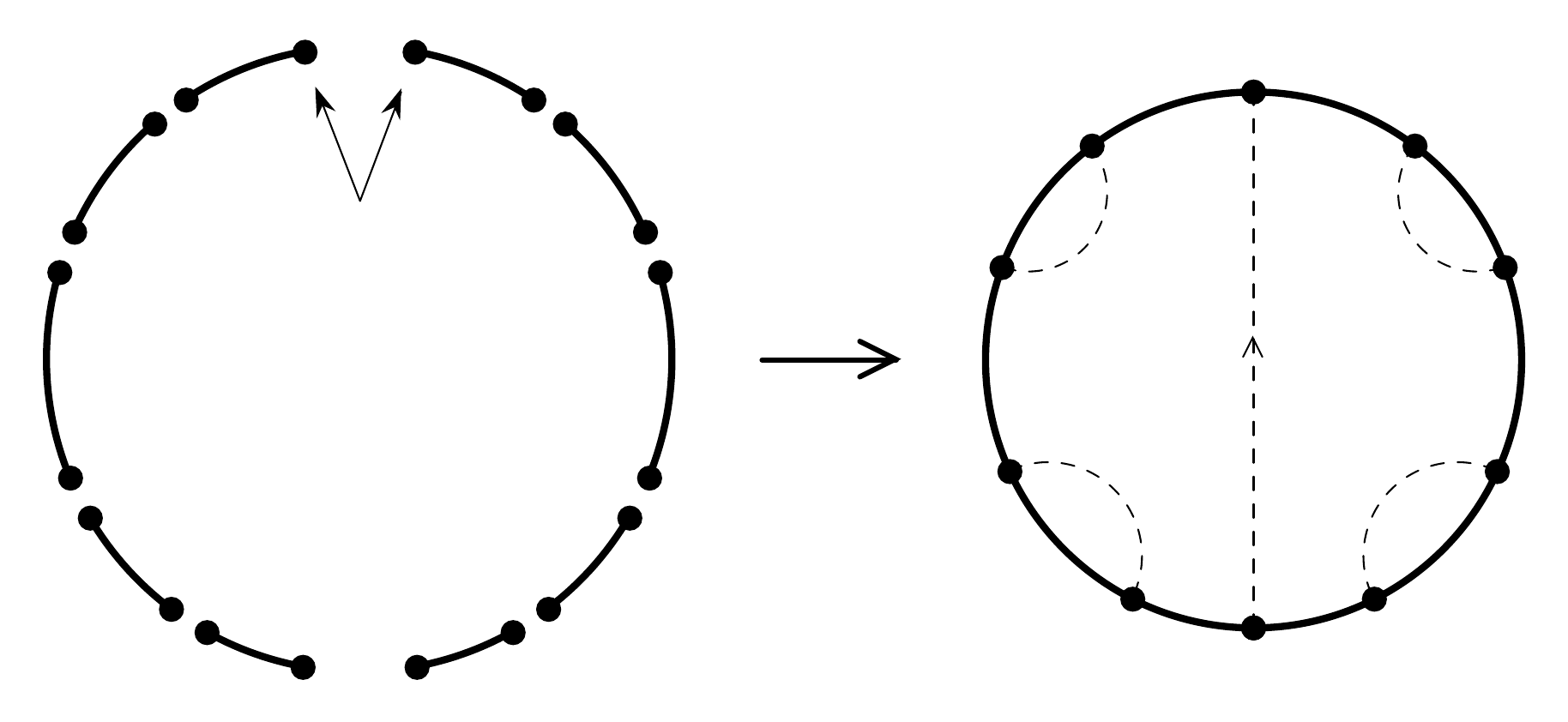}
\caption{Denjoy blow-up.}
\label{fig:Denjoy}
\end{figure}

More generally, one can define a Denjoy blow-up for actions of fundamental groups of higher-dimensional compact hyperbolic manifolds. Let $M^n=\H^n/\Ga$ be a compact hyperbolic $n$-manifold containing a compact totally geodesic hypersurface $C$. Let $A\subset \H^n$ denote the preimage of $C$ in $\H^n$. The visual boundary of each component $A_i$ of $A$ is an $(n-2)$-dimensional sphere $S_i\subset S^{n-1}=\pa \H^n$. The blow-up $\La$ of $S^{n-1}$ is then performed by replacing each sphere $S_i$ with two copies of this sphere. The result is a compact topological space $\La$ equipped with a quotient map $q: \La\to S^{n-1}$ such that $q$ is 1-to-1 over every point not in $S= \cup_i S_i$ and is 2-to-1 over every point in $S$. Each connected component of $\La$ is either a singleton or is homeomorphic to the $(n-2)$-dimensional Sierpinsky carpet. The action of $\Ga$ on $S^{n-1}$ lifts to a continuous action of $\Ga$ on $\La$ which is S-hyperbolic for a suitable choice of a metric on $\La$.
\end{example}

\subsection{S-hyperbolic actions of non-hyperbolic groups}

Here we consider non-hyperbolic groups and examples of their S-hyperbolic actions.

The following example shows that faithfulness of an S-hyperbolic action on $M$ does not imply faithfulness of perturbed actions.

\begin{example}[Non-discrete representation]\label{ex:H^4}
Suppose that $M$ is the standard compactification of the hyperbolic space $\H^4$, $P=\H^2\subset \H^4$ is a hyperbolic plane, and $\La=S^1\subset S^3=\pa\H^4$ is the ideal boundary of $P$. We let
\[
\Ga=\Ga_1\times\Ga_2, 
\]
where $\Ga_1$ is a hyperbolic surface group and $\Ga_2\cong \Z$. We consider a faithful isometric action $\rho$ of $\Ga$ on $\H^4$ where $\Ga_1$ preserves $P$ and acts on it properly discontinuously and cocompactly, while $\Ga_2$ acts as a group of elliptic isometries fixing $P$ pointwise. This action admits a conformal extension to $M$. Since the subgroup $\Ga_1< \Isom(\H^4)$ is convex-cocompact, it is expanding at its limit set, which is equal to $\La$. We take $\vS=\vS_1\times\{\id\}\cup\{\id\}\times\{r,r^{-1}\}$ as a symmetric generating set of $\Ga$, where $\vS_1$ is a finite generating set of $\Ga_1$ (given by its expanding action) and $r$ is a single generator of $\Ga_2$. The expansion subsets $U_r$, $U_{r^{-1}}$ of $r$, $r^{-1}$ (as in the definition of an expanding action) are defined to be the empty set. The action $\rho$ on $M$ is uniformly S-hyperbolic: the expansion property is clear; for uniform S-hyperbolicity, we observe that all the rays in $\Ray_x(\D,\de)$ for the action of $\Ga$ have the form
\[
(s\, c^\al_k)_{k\in\N} 
\] 
where $s\in\vS$ and $c^\al$ is a $\de$-ray in $\Ga_1\times\{\id\}$ associated to a special code $\al$ for $x\in\La$.

The image $\rho(r)$ is an infinite order elliptic rotation. Therefore, we can approximate $\rho$ by isometric actions $\rho_k$ of $\Ga$ on $\H^4$ such that $\rho_k|_{\Ga_1}=\rho|_{\Ga_1}$, while $\rho_k(r)$ is an elliptic transformation of order $i$ fixing $P$ pointwise. In particular, the representations $\rho_k$ are not faithful.
\end{example}

In the next two examples, a non-hyperbolic $\Ga$ acts faithfully and S-hyperbolically on $\La$.

\begin{example}[Actions of product groups]\label{ex:E2} 
For $k=1,2$, we consider S-hyperbolic actions $\rho_k: \Ga_k\to \Homeo(\La_k)$. Let $(\I_k,\U_k,\vS_k)$ be the respective expansion datum. Since $\La_k$'s are non-empty, the groups $\Ga_k$ are infinite. Consider the group
\[
\Ga=\Ga_1\times\Ga_2.
\]
This group is non-hyperbolic since its Cayley graph contains a quasi-flat (the product of complete geodesics in the Cayley graphs of $\Ga_1$ and $\Ga_2$). Let $\I=\I_1\cup\I_2$ and we equip $\Ga$ with a symmetric generating set
\[
 \vS=\vS_1\times\{\id\}\sqcup\{\id\}\times\vS_2. 
\]
Define the space $\La=\La_1\sqcup\La_2$. The group $\Ga$ acts on $\La$ as follows:
\[
(\ga_1, \ga_2)(x)= \ga_k(x)\;\textup{ if } x\in \La_k,
\]
where $(\ga_1,\ga_2)\in\Ga_1\times\Ga_2$. If the actions of $\Ga_i$ on $\La_i$ are both faithful, so is the action of $\Ga$ on $\La$. As a cover $\U$ of $\La$ we take the union $\U_1\cup\U_2$ of respective covers. We leave it to the reader to verify that the action of $\Ga$ on $\La$ is S-hyperbolic.

This example can be modified to a minimal action. Namely, take identical actions $\rho_1=\rho_2$ of the same group $\Ga_1=\Ga_2$ and then extend the action of $\Ga=\Ga_1\times \Ga_1$ on $\La$ to a minimal action of $\Ga\rtimes \Z_2$ as in Example~\ref{ex:E1}. Here the generator of $\Z_2$ swaps the direct factors of $\Ga$.
\end{example}

\begin{example}[$\Z^n$ acting on $\p^n(\R)$]
There is an S-hyperbolic action of $\Z^n$ on $M=\p^n(\R)$ by projective transformations.

Let $\{E_i\mid i=0,1,\ldots,n\}$ be the standard basis of $\R^{n+1}$. Let $\Z^n<\GL(n+1,\R)$ be the free abelian group of rank $n$ generated by \emph{bi-proximal} diagonal matrices $G_j$ ($1\le j\le n$) for which $E_0$ (resp. $E_j$) is the eigenvector of the biggest (resp. smallest) modulus eigenvalue. Denote by $e_i\in\p^n(\R)$ and $g_j\in\PGL(n+1,\R)$ the projectivizations of $E_i$ and $G_j$, respectively. Let $\La=\{e_i\mid i=0,1,\ldots,n\}\subset\p^n(\R)=M$ and $\vS=\{g_j,g_j^{-1}\mid 1\le j\le n\}$. We claim that the action $\Z^n\to\Homeo(M)$ is S-hyperbolic.

Let $U_{g_j}\subset M$ denote an expanding subset of $g_j$ and similarly $U_{g_j^{-1}}$ for $g_j^{-1}$. We \emph{assume} that $U_{g_j^{-1}}=\emptyset$ for $j=2,3,\ldots,n$ but that $U_{g_1^{-1}}$ as well as $U_{g_j}$ ($1\le j\le n$) are non-empty. Then $\U=\{U_{g_j},U_{g_j^{-1}}\mid 1\le j\le n\}$ covers $\La$, since $e_0\in U_{g_1^{-1}}$ and $e_j\in U_{g_j}$ for $1\le j\le n$. Thus the action is expanding.

A ray associated to $e_j\in\La$ is of the form $\left(g(g_j^{-1})^k\right)_{k\in\N_0}$ with $g\in\vS$ and a ray associated to $e_0\in\La$ is of the form $\left(g(g_1)^k\right)_{k\in\N_0}$ with $g\in\vS$. In any case, each point in $\La$ has only a finite number of rays associated to it. Hence the S-hyperbolicity follows.
\end{example}

\subsection{Embedding into Lie group actions on homogeneous manifolds}

Examples~\ref{ex:E0}, \ref{ex:E1} and \ref{ex:E2} can be embedded in smooth Lie group actions on homogeneous manifolds.

For instance, consider the action of $\GL(3,\R)$ on the space of oriented lines in $\R^3$, which we identify with the 2-sphere $\s^2$ equipped with its standard metric. We have the equivariant 2-fold covering $p:\s^2\to\p^2(\R)$ with the covering group generated by the antipodal map $-I\in \GL(3,\R)$. Let $\H^2\subset \p^2(\R)$ be the Klein model of the hyperbolic plane invariant under a subgroup $\PSO(2,1)<\PSL(3,\R)$. Then $p^{-1}(\H^2)$ consists of two disjoint copies of the hyperbolic plane bounded by two circles $\La_1$ and $\La_2$. Taking a discrete convex-cocompact subgroup $\Ga_1< \SO(2,1)< \SL(3,\R)$, $\Ga= \Ga_1\times \langle -I \rangle\cong \Ga_1\times \Z_2$ and $\La= \La_1\cup \La_2$, we obtain an S-hyperbolic action
\[
\Ga\to \Homeo(\s^2), 
\]
which restricts on $\La$ to Example~\ref{ex:E1}. (When the group $\Ga_1$ is free, such an action is studied in detail by Choi and Goldman \cite{CG}; see also \cite{ST} for a greater generalization of the Choi-Goldman construction.)

Below is a more general version of the preceding construction. Recall that the set $\Hom^\tau(\Ga,G)$ of $\tad$-Anosov representations $\Ga\to G$ forms an open subset of the representation variety $\Hom(\Ga,G)$ (see Corollary~\ref{cor:stability}). Let $\Ga=\pi_1(S_g)$ ($g\ge2$) and $G=\PSL(n,\R)$ ($n\ge3$). We will consider two types of simplices $\tad$ for the Lie group $G$:
\begin{itemize}[nosep]
\item
$\sid$; the corresponding flag manifold $\Flag(\sid)$ consists of full flags in $\R^n$.
\item
$\tad$ of the type ``pointed hyperplanes''; the corresponding flag manifold $\Flag(\tad)$ consists of pairs $V_1\subset V_{n-1} \subset \R^n$ of lines contained in hyperplanes in $\R^n$.
\end{itemize}
In both cases we have a natural fibration $q: \Flag(\tad)\to \p(\R^n)$ sending each flag to the line in the flag.

\begin{example}[Hitchin and Barbot representations]
The Hitchin representations are $\sid$-Anosov representations belonging to a connected component $\Hom^\Hit(\Ga,G)$ of $\Hom^\si(\Ga,G)$ containing a representation
\[
\Ga\hookrightarrow\PSL(2,\R)\hookrightarrow\PSL(n,\R),
\]
where the first map is a Fuchsian representation of $\Ga$ and the second an irreducible embedding of $\PSL(2,\R)$.

We may also consider the standard reducible embedding $\iota:\SL(2,\R)\hookrightarrow\SL(n,\R)$ given by $A\mapsto\begin{psmallmatrix}A&0\\0&I\end{psmallmatrix}$. Let $\varphi:\Ga\hookrightarrow\PSL(2,\R)$ be a Fuchsian representation and $\tilde{\varphi}:\Ga\hookrightarrow\SL(2,\R)$ one of the $2g$ lifts of $\varphi$. Let $p:\SL(n,\R)\to\PSL(n,\R)$ denote the covering map, which is of degree $2$ if and only if $n$ is even. Representations of the form
\[
p\circ\iota\circ\tilde{\varphi}:\Ga\hookrightarrow\SL(2,\R)\hookrightarrow\SL(n,\R)\to\PSL(n,\R)
\]
are $\tad$-Anosov, where $\tad$ has the type of pointed hyperplanes. Let $\Hom^\Ba(\Ga,G)$ be the union of connected components of $\Hom^\tau(\Ga,G)$ containing such representations. We say representations in $\Hom^\Ba(\Ga,G)$ are of \emph{Barbot type}; see \cite{Bar} for the case $n=3$.

Let $\rho:\Ga\to\PSL(n,\R)$ be a Hitchin representation. It is known \cite{Lab} that the projection $q\circ\psi:\pa\Ga\to\Flag(\sid)\to\p^{n-1}(\R)$ of the asymptotic embedding $\psi$ is a hyper-convex curve. This curve is homotopically trivial if and only if $n$ is odd. On the other hand, if $\rho$ is of Barbot type, the curve $q\circ\psi:\pa\Ga\to\Flag(\tad)\to\p^{n-1}(\R)$ is always homotopically non-trivial.

We now lift to the space of oriented full flags (resp. oriented line-hyperplane flags). Accordingly, we lift the action of $\Ga$ on the sphere $\s^{n-1}$. The preimage $\La$ of $(q\circ\psi)(\pa\Ga)$ in $\s^{n-1}$ is either a Jordan curve or a disjoint union of two Jordan curves, with the 2-fold equivariant covering map
\[
\La\to (q\circ\psi)(\pa \Ga)\cong \p^1(\R)\cong S^1. 
\]
The result is an S-hyperbolic action $\tilde\rho: \Ga\to \Diff(\s^n)$ where $\tilde\rho(\Ga)$ is contained in the image of the group $\SL(n,\R)$ in $\Diff(\s^n)$. The restrictions of the $\Ga$-actions to $\La$ are as in Example~\ref{ex:E0} (with $k=2$) and Example~\ref{ex:E1}.
\end{example}

\begin{example}[Embedded product examples]
We embed Example~\ref{ex:E2} in the action of $\SL(4,\R)$ on $\p^3(\R)$. Consider $G_1\times G_2= \SL(2,\R)\times \SL(2,\R)< G= \SL(4,\R)$. The action of $G_1\times G_2$ on $\R^4$ is reducible, preserving a direct sum decomposition
\[
\R^4= V_1\oplus V_2,
\]
where $V_1$ and $V_2$ are 2-dimensional subspaces and $G_i$ acts trivially on $V_{3-k}$ for $k=1,2$. Let $\tau$ be an involution of $\R^4$ swapping $V_1$ and $V_2$. For $i=1,2$, take $\Ga_k< G_k$ to be an infinite (possibly elementary) convex-cocompact subgroup with the limit set $\La_k\subset\p(V_k)$. Then the subgroup $\Ga= \Ga_1 \times \Ga_2< G$ acts on $\p^3(\R)$ preserving the union $\La=\La_1\sqcup \La_2\subset\p(V_1)\sqcup\p(V_2)$. We equip $\p^3(\R)$ with its standard Riemannian metric. The action
\[
\Ga\to \Homeo(\p^3(\R))
\]
is S-hyperbolic and restricts on $\La$ to Example~\ref{ex:E2}. As in Example~\ref{ex:E2}, taking $\Ga_1=\Ga_2$ and an index two extension $\Ga$ of $\Ga_1\times \Ga_2$ we can extend this action to an S-hyperbolic action minimal on $\La$ using the involution $\tau$.
\end{example}

\subsection{Algebraically stable but not convex-cocompact}

We provide an example to the claim made in Remark~\ref{rem:sul}  that for groups with torsion the implication (4 $\Rightarrow$ 1) in Sullivan's paper \cite{Sul} is false.

\begin{example}[Quasiconformally stable non-convex-cocompact subgroups of $\PSL(2,\C)$]\label{ex:VD}
We recall also that a \emph{von Dyck group} $D(p,q,r)$ is given by the presentation
\[
\langle a, b, c \mid a^p=b^q=c^r=1,\ abc=1 \rangle.
\]
Such groups are called \emph{hyperbolic} (resp. \emph{parabolic}, \emph{elliptic}) if the number
\[
\chi= \chi(p,q,r)= \frac{1}{p} + \frac{1}{q} + \frac{1}{r}
\]
is $<1$ (resp. $=1$, $>1$). Depending on the type $D(p,q,r)$ can be embedded (uniquely up to conjugation) as a discrete cocompact subgroup of isometries of hyperbolic plane (if $\chi<1$), a discrete cocompact subgroup of the group $\Aff({\C})$ of complex affine transformations of $\C$ (if $\chi=1$), or is finite and embeds in the group of isometries of the 2-sphere (if $\chi>1$).

Let $\Ga_k$ ($k=1,2$) be two discrete elementary subgroups of $\Aff(\C)<\PSL(2,\C)$ isomorphic to parabolic von Dyck groups
\[
D(p_k, q_k, r_k)\; (k=1,2).
\]
These groups consist of elliptic and parabolic elements and are virtually free abelian of rank 2; hence they cannot be contained in a convex-cocompact group. Subgroups of $\PSL(2,\C)$ isomorphic to von Dyck groups are (locally) rigid. One can choose embeddings of the groups $\Ga_1$ and $\Ga_2$ into $\PSL(2,\C)$ such that they generate a free product
\[
\Ga= \Ga_1\star \Ga_2< \PSL(2,\C),
\]
which is geometrically finite and every parabolic element of $\Ga$ is conjugate into one of the free factors. (The group $\Ga$ is obtained via the Klein combination of $\Ga_1$ and $\Ga_2$, see \cite{Kapovich-book}*{\S 4.18} for example). The discontinuity domain $\Omega$ of $\Ga$ in $\p^1(\C)$ is connected and the quotient orbifold ${\mathcal O}= \Omega/\Ga$ is a sphere with six cone points of the orders $p_i$, $q_i$ and $r_i$ ($i=1,2$).

Being geometrically finite, the group $\Ga$ is relatively stable (relative its parabolic elements): let
\[
\Hom_{\mathrm{par}}(\Ga,\PSL(2,\C))
\]
denote the \emph{relative representation variety}, which is the subvariety in the representation variety defined by the condition that images of parabolic elements of $\Ga$ are again parabolic. Let $\iota_{\Ga}:\Ga\to\PSL(2,\C)$ denote the identity embedding. Then there is a small neighborhood $U$ of $\iota_{\Ga}$ in $\Hom_{\textup{par}}(\Ga,\PSL(2,\C))$ which consists entirely of faithful geometrically finite representations which are, moreover, given by quasiconformal conjugations of $\Ga$. Since the subgroups $\Ga_1$ and $\Ga_2$ are rigid, there is a neighborhood $V$ of $\iota_{\Ga}$ such that
\[
V\cap\Hom_{\mathrm{par}}(\Ga,\PSL(2,\C)) = V\cap\Hom(\Ga,\PSL(2,\C)). 
\]
It follows that the action of $\Ga$ on its limit set is structurally stable in $\PSL(2,\C)$, in particular, algebraically stable. However, $\Ga$ is not convex-cocompact. Lastly, the group $\Ga$ is not rigid, the (complex) dimension of the character variety
\[
X(\Ga,\PSL(2,\C))\sslash\PSL(2,\C)
\]
near $[\iota_\Ga]$ equals the (complex) dimension of the Teichm\"uller space of ${\mathcal O}$, which is $3$.
\end{example}

On the other hand, one can show that if $\Ga<\PSL(2,\C)$ is a finitely generated discrete subgroup which is not a lattice and contains no parabolic von Dyck subgroups, then algebraic stability of $\Ga$ implies quasiconvexity of $\Ga$. We refer the reader to the paper by Matsuzaki \cite{Mat} for the precise description of structurally stable finitely generated Kleinian subgroups of $\PSL(2,\C)$.

\vspace{3\baselineskip}\noindent
Department of Mathematics,
University of California,
1 Shields Ave., Davis, CA 95616,
USA\\
\url{kapovich@math.ucdavis.edu}

\vspace{1\baselineskip}\noindent
Department of Mathematics,
Jeju National University,
Jeju 63243,
Republic of Korea\\
\url{sungwoon@jejunu.ac.kr}

\vspace{1\baselineskip}\noindent
Research Institute of Mathematics,
Seoul National University,
Seoul 08826,
Republic of Korea\\
\url{jaejeong@snu.ac.kr}


\begin{bibdiv}
\begin{biblist}[\normalsize]

\bib{BGS}{book}{
   author={Ballmann, Werner},
   author={Gromov, Mikhael},
   author={Schroeder, Viktor},
   title={Manifolds of nonpositive curvature},
   series={Progress in Mathematics},
   volume={61},
   publisher={Birkh\"auser Boston, Inc., Boston, MA},
   date={1985},
   pages={vi+263},
   isbn={0-8176-3181-X},
}

\bib{Bar}{article}{
   author={Barbot, Thierry},
   title={Three-dimensional Anosov flag manifolds},
   journal={Geom. Topol.},
   volume={14},
   date={2010},
   number={1},
   pages={153--191},
   issn={1465-3060},
}

\bib{BPS}{article}{
   author={Bochi, Jairo},
   author={Potrie, Rafael},
   author={Sambarino, Andr\'{e}s},
   title={Anosov representations and dominated splittings},
   journal={J. Eur. Math. Soc. (JEMS)},
   volume={21},
   date={2019},
   number={11},
   pages={3343--3414},
   issn={1435-9855},
}

\bib{Bowditch-gf}{article}{
   author={Bowditch, Brian H.},
   title={Geometrical finiteness for hyperbolic groups},
   journal={J. Funct. Anal.},
   volume={113},
   date={1993},
   number={2},
   pages={245--317},
   issn={0022-1236},
}

\bib{Bo}{article}{
   author={Bowditch, Brian H.},
   title={Spaces of geometrically finite representations},
   journal={Ann. Acad. Sci. Fenn. Math.},
   volume={23},
   date={1998},
   number={2},
   pages={389--414},
   issn={1239-629X},
}

\bib{Bowditch}{article}{
   author={Bowditch, Brian H.},
   title={A topological characterisation of hyperbolic groups},
   journal={J. Amer. Math. Soc.},
   volume={11},
   date={1998},
   number={3},
   pages={643--667},
   issn={0894-0347},
}

\bib{BH}{book}{
   author={Bridson, Martin R.},
   author={Haefliger, Andr\'e},
   title={Metric spaces of non-positive curvature},
   series={Grundlehren der Mathematischen Wissenschaften},
   volume={319},
   publisher={Springer-Verlag, Berlin},
   date={1999},
}

\bib{CEG}{article}{
   author={Canary, R. D.},
   author={Epstein, D. B. A.},
   author={Green, P.},
   title={Notes on notes of Thurston},
   book={
      series={London Math. Soc. Lecture Note Ser.},
      volume={111},
      publisher={Cambridge Univ. Press, Cambridge},
   },
   date={1987},
   pages={3--92},
}

\bib{CG}{article}{
   author={Choi, Suhyoung},
   author={Goldman, William},
   title={Topological tameness of Margulis spacetimes},
   journal={Amer. J. Math.},
   volume={139},
   date={2017},
   number={2},
   pages={297--345},
   issn={0002-9327},
}

\bib{Coo}{article}{
   author={Coornaert, Michel},
   title={Mesures de Patterson-Sullivan sur le bord d'un espace hyperbolique au sens de Gromov},
   journal={Pacific J. Math.},
   volume={159},
   date={1993},
   number={2},
   pages={241--270},
}

\bib{CP}{book}{
   author={Coornaert, Michel},
   author={Papadopoulos, Athanase},
   title={Symbolic dynamics and hyperbolic groups},
   series={Lecture Notes in Mathematics},
   volume={1539},
   publisher={Springer-Verlag, Berlin},
   date={1993},
   pages={viii+138},
}

\bib{Corlette}{article}{
   author={Corlette, Kevin},
   title={Hausdorff dimensions of limit sets. I},
   journal={Invent. Math.},
   volume={102},
   date={1990},
   number={3},
   pages={521--541},
   issn={0020-9910},
}

\bib{DGK}{article}{
   author={Danciger, Jeffrey},
   author={Gu\'eritaud, Fran\c{c}ois},
   author={Kassel, Fanny},
   title={Convex cocompact actions in real projective geometry},
   date={2017},
   journal={\tt arXiv:1704.08711 [math.GT]},
}

\bib{Deroin}{article}{
   author={Deroin, Bertrand},
   title={Locally discrete expanding groups of analytic diffeomorphisms of
   the circle},
   journal={J. Topol.},
   volume={13},
   date={2020},
   number={3},
   pages={1216--1229},
   issn={1753-8416},
}

\bib{DK}{book}{
   author={Dru\c{t}u, Cornelia},
   author={Kapovich, Michael},
   title={Geometric group theory},
   series={American Mathematical Society Colloquium Publications},
   volume={63},
   note={With an appendix by Bogdan Nica},
   publisher={American Mathematical Society, Providence, RI},
   date={2018},
   pages={xx+819},
   isbn={978-1-4704-1104-6},
}

\bib{Eber}{book}{
   author={Eberlein, Patrick B.},
   title={Geometry of nonpositively curved manifolds},
   series={Chicago Lectures in Mathematics},
   publisher={University of Chicago Press, Chicago, IL},
   date={1996},
   pages={vii+449},
   isbn={0-226-18197-9},
   isbn={0-226-18198-7},
}

\bib{Freden}{article}{
   author={Freden, Eric M.},
   title={Negatively curved groups have the convergence property. I},
   journal={Ann. Acad. Sci. Fenn. Ser. A I Math.},
   volume={20},
   date={1995},
   number={2},
   pages={333--348},
   issn={0066-1953},
}

\bib{Gromov}{article}{
   author={Gromov, M.},
   title={Hyperbolic groups},
   conference={
      title={Essays in group theory},
   },
   book={
      series={Math. Sci. Res. Inst. Publ.},
      volume={8},
      publisher={Springer, New York},
   },
   date={1987},
   pages={75--263},
}

\bib{GW}{article}{
   author={Guichard, Olivier},
   author={Wienhard, Anna},
   title={Anosov representations: domains of discontinuity and applications},
   journal={Invent. Math.},
   volume={190},
   date={2012},
   number={2},
   pages={357--438},
   issn={0020-9910},
}

\bib{Izeki}{article}{
   author={Izeki, Hiroyasu},
   title={Quasiconformal stability of Kleinian groups and an embedding of a
   space of flat conformal structures},
   journal={Conform. Geom. Dyn.},
   volume={4},
   date={2000},
   pages={108--119},
   issn={1088-4173},
}

\bib{Kapovich-book}{book}{
   author={Kapovich, Michael},
   title={Hyperbolic manifolds and discrete groups},
   series={Progress in Mathematics},
   volume={183},
   publisher={Birkh\"{a}user Boston, Inc., Boston, MA},
   date={2001},
   pages={xxvi+467},
   isbn={0-8176-3904-7},
}

\bib{KL18a}{article}{
   author={Kapovich, Michael},
   author={Leeb, Bernhard},
   title={Discrete isometry groups of symmetric spaces},
   conference={
      title={Handbook of group actions. Vol. IV},
   },
   book={
      series={Adv. Lect. Math. (ALM)},
      volume={41},
      publisher={Int. Press, Somerville, MA},
   },
   date={2018},
   pages={191--290},
}

\bib{KL18b}{article}{
   author={Kapovich, Michael},
   author={Leeb, Bernhard},
   title={Finsler bordifications of symmetric and certain locally symmetric
   spaces},
   journal={Geom. Topol.},
   volume={22},
   date={2018},
   number={5},
   pages={2533--2646},
   issn={1465-3060},
}

\bib{KLP14}{article}{
   author={Kapovich, Michael},
   author={Leeb, Bernhard},
   author={Porti, Joan},
   title={Morse actions of discrete groups on symmetric space},
   date={2014},
   journal={\tt arXiv:1403.7671 [math.GR]},
}

\bib{KLP16}{article}{
   author={Kapovich, Michael},
   author={Leeb, Bernhard},
   author={Porti, Joan},
   title={Some recent results on Anosov representations},
   journal={Transform. Groups},
   volume={21},
   date={2016},
   number={4},
   pages={1105--1121}
}

\bib{KLP17}{article}{
   author={Kapovich, Michael},
   author={Leeb, Bernhard},
   author={Porti, Joan},
   title={Anosov subgroups: dynamical and geometric characterizations},
   journal={Eur. J. Math.},
   volume={3},
   date={2017},
   number={4},
   pages={808--898},
   issn={2199-675X},
}

\bib{KLP18}{article}{
   author={Kapovich, Michael},
   author={Leeb, Bernhard},
   author={Porti, Joan},
   title={A Morse lemma for quasigeodesics in symmetric spaces and Euclidean
   buildings},
   journal={Geom. Topol.},
   volume={22},
   date={2018},
   number={7},
   pages={3827--3923},
   issn={1465-3060},
}

\bib{Lab}{article}{
   author={Labourie, Fran\c{c}ois},
   title={Anosov flows, surface groups and curves in projective space},
   journal={Invent. Math.},
   volume={165},
   date={2006},
   number={1},
   pages={51--114},
   issn={0020-9910},
}

\bib{MM}{article}{
   author={Mann, Kathryn},
   author={Manning, Jason Fox},
   title={Stability for hyperbolic groups acting on boundary spheres},
   date={2021},
   journal={\tt arXiv:2104.01269 [math.GT]},
}

\bib{Marden}{article}{
   author={Marden, Albert},
   title={The geometry of finitely generated Kleinian groups},
   journal={Ann. of Math. (2)},
   volume={99},
   date={1974},
   pages={383--462},
   issn={0003-486X},
}

\bib{Maskit}{book}{
   author={Maskit, Bernard},
   title={Kleinian groups},
   series={Grundlehren der Mathematischen Wissenschaften},
   volume={287},
   publisher={Springer-Verlag, Berlin},
   date={1988},
   pages={xiv+326},
   isbn={3-540-17746-9},
}

\bib{Mat}{article}{
   author={Matsuzaki, Katsuhiko},
   title={Structural stability of Kleinian groups},
   journal={Michigan Math. J.},
   volume={44},
   date={1997},
   number={1},
   pages={21--36},
   issn={0026-2285},
}

\bib{Series}{article}{
   author={Series, Caroline},
   title={Symbolic dynamics for geodesic flows},
   journal={Acta Math.},
   volume={146},
   date={1981},
   number={1-2},
   pages={103--128},
   issn={0001-5962},
}

\bib{ST}{article}{
   author={Stecker, Florian},
   author={Treib, Nicolaus},
   title={Domains of discontinuity in oriented flag manifolds},
   date={2018},
   journal={\tt arXiv:1806.04459 [math.DG]},
}

\bib{Sul}{article}{
   author={Sullivan, Dennis},
   title={Quasiconformal homeomorphisms and dynamics. II. Structural
   stability implies hyperbolicity for Kleinian groups},
   journal={Acta Math.},
   volume={155},
   date={1985},
   number={3-4},
   pages={243--260},
   issn={0001-5962},
}

\bib{Tukia}{article}{
   author={Tukia, Pekka},
   title={Convergence groups and Gromov's metric hyperbolic spaces},
   journal={New Zealand J. Math.},
   volume={23},
   date={1994},
   number={2},
   pages={157--187},
   issn={1171-6096},
}

\bib{Tukia98}{article}{
   author={Tukia, Pekka},
   title={Conical limit points and uniform convergence groups},
   journal={J. Reine Angew. Math.},
   volume={501},
   date={1998},
   pages={71--98},
   issn={0075-4102},
}

\bib{Yue}{article}{
   author={Yue, Chengbo},
   title={Dimension and rigidity of quasi-Fuchsian representations},
   journal={Ann. of Math. (2)},
   volume={143},
   date={1996},
   number={2},
   pages={331--355},
   issn={0003-486X},
}


\end{biblist}
\end{bibdiv}
\end{document}